\catcode`\@=11
\newif\if@fewtab\@fewtabtrue
{\count255=\time\divide\count255 by 60
\xdef\hourmin{\number\count255}
\multiply\count255 by-60\advance\count255 by\time
\xdef\hourmin{\hourmin:\ifnum\count255<10 0\fi\the\count255}}
\def\ps@draft{\let\@mkboth\@gobbletwo
     \def\@oddfoot{\hbox to 7 cm{\tiny \versionno
        \hfil}\hskip -7cm\hfil\rm\thepage \hfil {\tiny\draftdate}}
     \def\@oddhead{}
     \def\@evenhead{}\let\@evenfoot\@oddfoot}
\def\draftdate{\number\month/\number\day/\number\year\ \ \ \hourmin }

\def\citen#1{\if@filesw \immediate\write \@auxout {\string\citation{#1}}\fi%
\@tempcntb\m@ne \let\@h@ld\relax \def\@citea{}%
\@for \@citeb:=#1\do {\@ifundefined {b@\@citeb}%
     {\@h@ld\@citea\@tempcntb\m@ne{\bf ?}%
     \@warning {Citation `\@citeb ' on page \thepage \space undefined}}%
     {\@tempcnta\@tempcntb \advance\@tempcnta\@ne
     \setbox\z@\hbox\bgroup\ifcat0\csname b@\@citeb \endcsname \relax
     \egroup \@tempcntb\number\csname b@\@citeb \endcsname \relax
     \else \egroup \@tempcntb\m@ne \fi \ifnum\@tempcnta=\@tempcntb
     \ifx\@h@ld\relax \edef \@h@ld{\@citea\csname b@\@citeb\endcsname}%
     \else \edef\@h@ld{\hbox{--}\penalty\@highpenalty
     \csname b@\@citeb\endcsname}\fi
     \else \@h@ld\@citea\csname b@\@citeb \endcsname \let\@h@ld\relax \fi}%
\def\@citea{,\penalty\@highpenalty\hskip.13em plus.13em minus.13em}}\@h@ld}
\def\@citex[#1]#2{\@cite{\citen{#2}}{#1}}%
\def\@cite#1#2{\leavevmode\unskip\ifnum\lastpenalty=\z@\penalty\@highpenalty\fi%
   \ [{\multiply\@highpenalty 3 #1%
   \if@tempswa,\penalty\@highpenalty\ #2\fi}]}   %
\makeatother 
\catcode`\@=12

\def\apo           {\mbox{\sc s}}
\def\apoi          {\mbox{\sc s}^{-1}}
\def\be            {\begin{equation}}
\def\bearl         {\begin{array}{l}}
\def\bearll        {\begin{array}{ll}}
\def\BF            {\ensuremath{F}}
\newcommand\BG[2]  {B_{#1,#2}}
\def\bico          {_{F}}
\def\bicoa         {_\bicoaa}
\def\bicoaa        {{\triangleright\triangleleft}}
\def\bbico         {_{\!F}}

\def\bohr          {{\ohr\bico}}

\def\brho          {{\rho\bbico^{}}}
\def\C             {{\ensuremath{\mathcal C}}}
\def\catpic        {{\footnotesize \shadowbox{\C}}}
\def\catpicH       {{\footnotesize \shadowbox{\HBimod}}}

\def\cir           {\,{\circ}\,}

\newcommand\coend[1]{\int^{#1}\hspace*{-.31em}#1^\vee{\otimes}#1}
\newcommand\Coend[2]{\int^{#2}\hspace*{-.23em}#1(#2,#2)}
\def\coa           {_\triangleright}
\def\coar          {_\triangleleft}

\newcommand\Cor[2] {\ensuremath{\mathrm{Cor}^{}_{#1;#2}}}
\newcommand\Corw[2]{\ensuremath{\mathrm{Cor}^{\omega}_{#1;#2}}}
\newcommand\Corr[3]{\ensuremath{\mathrm{Cor}^{}_{#1;#2,#3}}}

\def\ee            {\end{equation}}

\def\eear          {\end{array}}

\def\EndC          {{\ensuremath{\mathrm{End}_\C}}}

\def\Endk          {{\ensuremath{\mathrm{End}_\ko}}}
\def\eps           {\varepsilon}

\def\eq            {\,{=}\,}

\newcommand\erf[1] {(\ref{#1})}
\def\Fomega        {F^\omega}
\def\fftc          {factorizable finite tensor category}
\def\findim        {fi\-ni\-te-di\-men\-si\-o\-nal}
\def\Haa           {{\ensuremath{H\hspace*{-6pt}H^*\bicoa}}}
\def\Hb            {{\ensuremath{F}}}
\def\HBimod        {{\ensuremath{H\mbox{-}\mathrm{Bimod}}}}
\def\HK            {{{\ensuremath K}}}
\def\HMod          {{\ensuremath{H}\text{-Mod}}}
\def\Hom           {{\ensuremath{\mathrm{Hom}}}}

\def\HomC          {{\ensuremath{\mathrm{Hom}_\C}}}

\def\HomHH         {{\ensuremath{\mathrm{Hom}_{H|H}}}}
\def\Homk          {{\ensuremath{\mathrm{Hom}_\ko}}}

\def\Hs            {{\ensuremath{H^*}}}

\def\Hss           {{H^*_{}}}

\def\id            {\mbox{\sl id}}

\def\idHs          {\ensuremath{\id_{{H^{\phantom:}}^{\!\!*}}}}

\def\idsm          {\mbox{\footnotesize\sl id}}

\def\iHb           {\imath^{\Hb}}
\def\iHK           {\imath^\HK}

\def\iN            {\,{\in}\,}
\def\J             {{\ensuremath{\mathcal J}}}
\def\ko            {{\ensuremath{\Bbbk}}}
\def\LAak          {\ensuremath{z_{a_k}}}
\def\LAal          {\ensuremath{z_{a_l}}}
\def\LAbk          {\ensuremath{z_{b_k}}}
\def\LAdk          {\ensuremath{z_{d_k}}}
\def\LAek          {\ensuremath{z_{e_k}}}
\def\LAel          {\ensuremath{z_{e_l}}}
\def\LASk          {\ensuremath{z_{S_k}}}
\def\LAtjk         {\ensuremath{z_{t_{j,k}}}}
\def\LAwi          {\ensuremath{z_{\omega_i}}}

\def\LARj          {\ensuremath{z_{\R_j}}}
\newcommand\labl[1]{\label{#1}\ee}
\newcommand\Map[2] {\ensuremath{\text{Map}_{#1;#2}}}
\def\Mapgn         {\ensuremath{\text{Map}_{g:n}}}
\def\Mapgpq        {\ensuremath{\text{Map}_{g:p,q}}}
\def\Mapgppq       {\ensuremath{\text{Map}_{g:p+q}}}
\def\Mod           {\mbox{-Mod}}

\newcommand\nxl[1] {\\[#1mm]}
\newcommand\Nxl[1] {\\[-1.3em]\\[#1mm]}
\def\ohr           {\reflectbox{$\rho$}}
\def\ohrad         {\ohr_\diamond}
\def\ohrv          {\ohr_{\scriptscriptstyle\!\vee}^{}}
\def\ohrV          {{}_{\scriptscriptstyle\vee\!}^{}\ohr}
\def\one           {{\bf1}}

\def\op            {^{\mathrm{op}}}
\def\oti           {\,{\otimes}\,}

\def\otik          {\,{\otimes_\ko}\,}
\def\Otik          {{\otimes_\ko}}

\def\QB            {\mathcal Q^K}
\def\QQ            {\ensuremath{\mathcal O_{\!K}}}
\def\qquand        {\qquad{\rm and}\qquad}
\def\R             {\vartheta}
\def\rep           {representation}
\def\rhov          {\rho_{\scriptscriptstyle\!\vee}^{}}
\def\rhoV          {{}_{\scriptscriptstyle\vee\!}^{}\rho}

\newcommand\setulen[2]{\setlength\unitlength{.#1#2pt}}
\def\SK            {S_\HK}
\def\slz           {\ensuremath{\mathrm{SL}(2,\zet)}}
\def\sse           {\scriptsize }
\newcommand\Surf[2] {\Sigma_{#1,#2}}
\def\tauHH         {\tau^{}_{\!H,H}}
\def\tauHHv        {\tau^{}_{\!H,H^*_{}}}
\def\tauHvH        {\tau^{}_{\!H^*_{}\!,H}}
\def\Times         {\,{\times}\,}
\def\TK            {T_\HK}

\def\uvi           {{t}}
\def\Vect          {\ensuremath{\mathcal V}\mbox{\sl ect}}
\def\Vectk         {\ensuremath{\mathcal V\mbox{\sl ect}_\ko}}
\def\Vee           {{}^{\vee\!}}
\def\Xs            {X^{*_{}}_{\phantom:}}
\def\zet           {{\ensuremath{\mathbb Z}}}

\def\ak   {\ensuremath{a_k}}
\def\bk   {\ensuremath{b_k}}
\def\dk   {\ensuremath{d_k}}
\def\ek   {\ensuremath{e_k}}
\def\sk   {\ensuremath{S_k}}

\def\wi   {\ensuremath{\omega_i}}
\def\Ri   {\ensuremath{\R_i}}

\newcommand\includepichtft[1] {{\begin{picture}(0,0)(0,0)
                    \scalebox{.304}{\includegraphics{imgs/pic_htft_#1.eps}}\end{picture}}}
\newcommand\Includepichtft[1] {{\begin{picture}(0,0)(0,0)
                    \scalebox{.38}{\includegraphics{imgs/pic_htft_#1.eps}}\end{picture}}}
\newcommand\INcludepichtft[2] {{\begin{picture}(0,0)(0,0)
                   \scalebox{.#2}{\includegraphics{imgs/pic_htft_#1.eps}}\end{picture}}}
\newcommand\includepichopf[1] {{\begin{picture}(0,0)(0,0)
                    \scalebox{.304}{\includegraphics{imgs/pic_hopf_#1.eps}}\end{picture}}}

\newcommand\eqpic[4]{\begin{eqnarray}
                    \begin{picture}(#2,#3){}\end{picture}\nonumber\\
                    \raisebox{-#3pt}{ \begin{picture}(#2,#3) #4 \end{picture} }
                    \label{#1} \\~\nonumber \end{eqnarray} }
\newcommand\Eqpic[4]{\begin{eqnarray}
                    \begin{picture}(#2,#3){}\end{picture}\nonumber\\
                    \raisebox{-#3pt}{ \begin{picture}(#2,#3) #4 \end{picture} }
                    \nonumber \\[-3pt]~\label{#1} \end{eqnarray} }

\documentclass[12pt]{article}
\usepackage{latexsym, amsmath, amsthm, amsfonts, enumerate, amssymb, bbm, xspace, xypic }
\usepackage{fancybox}
\usepackage[all]{xy}
\usepackage[mathscr]{eucal}
\usepackage{graphicx} \usepackage{rotating}
\usepackage{epstopdf,hyperref}

\newtheorem{thm}{Theorem}
\newtheorem{lem}[thm]{Lemma}
\newtheorem{lemma}[thm]{Lemma}

\newtheorem{prop}[thm]{Proposition}
\newtheorem{cor}[thm]{Corollary}
\theoremstyle{definition}
\newtheorem{rem}[thm]{Remark}
\newtheorem{defi}[thm]{Definition}

\begin{document}

\def\cir{\,{\circ}\,} 
\numberwithin{equation}{section}
\numberwithin{thm}{section}

\begin{flushright}
    {\sf ZMP-HH/12-13}\\
    {\sf Hamburger$\;$Beitr\"age$\;$zur$\;$Mathematik$\;$Nr.$\;$447}\\[2mm]
    July 2012
\end{flushright}
\vskip 3.5em
\begin{center}
\begin{tabular}c \Large\bf Higher genus mapping class group invariants \\[2mm]
                 \Large\bf from factorizable Hopf algebras
\end{tabular}
\end{center}\vskip 2.1em
\begin{center}
   ~J\"urgen Fuchs\,$^{\,a}$,~
   ~Christoph Schweigert\,$^{\,b}$,~
   ~Carl Stigner\,$^{\,a}$
\end{center}

\vskip 9mm

\begin{center}\it$^a$
   Teoretisk fysik, \ Karlstads Universitet\\
   Universitetsgatan 21, \ S\,--\,651\,88\, Karlstad
\end{center}
\begin{center}\it$^b$
   Organisationseinheit Mathematik, \ Universit\"at Hamburg\\
   Bereich Algebra und Zahlentheorie\\
   Bundesstra\ss e 55, \ D\,--\,20\,146\, Hamburg
\end{center}
\vskip 5.3em

\noindent{\sc Abstract}
\\[3pt]
Lyubashenko's construction associates representations of mapping class groups
\Mapgn\ of Rie\-mann surfaces of any genus $g$ with any number $n$ of
holes to a factorizable ribbon category. We consider this construction as
applied to the category of bimodules over a finite-dimensional factorizable
ribbon Hopf algebra $H$. For any such Hopf algebra we find an invariant of
\Mapgn\ for all values of $g$ and $n$. More generally, we obtain such invariants
for any pair $(H,\omega)$, where $\omega$ is a ribbon automorphism of $H$.
\\
Our results are motivated by the quest to understand higher genus
correlation functions of bulk fields in two-dimensional conformal field
theories with chiral algebras that are not necessarily semisimple, so-called
logarithmic conformal field theories.

  \newpage

\section{Introduction}

The mapping class groups of Riemann surfaces with holes form an interesting
system with deep properties and rich relations to geometry and arithmetic.
It is therefore remarkable that a relatively simple algebraic structure -- a
finite-dimensional factorizable ribbon Hopf algebra $H$ -- gives rise
\cite{lyub6} to a family of (projective) representations of all these mapping
class groups. The construction of mapping class group representations in
\cite{lyub6} does not require $H$ to be semisimple.
For semisimple $H$ the so obtained system of representations obeys even tighter
constraints: it is part of a so-called modular functor, or a three-dimensional
topological field theory. In the present article we do \emph{not} require
semisimplicity.

Another algebraic structure leading to a system of representations of mapping
class groups are vertex algebras which arise in chiral conformal
field theories. More specifically, in this case the mapping class group
representations are derived from the monodromies of the conformal blocks
\cite{FRbe} associated with the vertex algebra. In fortunate situations
the representation category of a vertex algebra is equivalent to the one of a
factorizable ribbon Hopf algebra at least as an abelian category.
Surprisingly, such a ``Kazhdan-Lusztig correspondence'' seems to work
particularly well for some classes of vertex algebras for which the category in
question is \emph{not} semisimple \cite{fgst2,naTs2}. The chiral conformal field
theories associated with these cases are known as {\em logarithmic} theories.

In a full, local conformal field theory, one aims in particular at constructing
correlators of bulk fields as specific \emph{bilinear} combinations of conformal
blocks. Translating the situation to the Hopf algebra setting, this amounts to
considering the mapping class group representations coming from the
factorizable ribbon Hopf algebra $H\oti H\op_{}$, the enveloping algebra of $H$.
(By factorizability, the category of $H\oti H\op_{}$-modules is equivalent as a
ribbon category to the Drinfeld center of \HMod.)
In \cite{fuSs3} we have constructed, for any ribbon automorphism $\omega$ of
$H$, an invariant of the mapping class group of the torus with one hole. The
construction in \cite{fuSs3} is based on a family of symmetric Frobenius
algebras $\Fomega$ in the braided monoidal category \HBimod,
which is braided equivalent to $(H{\otimes} H\op_{})$\Mod.
In \cite{fuSs4} we have in addition
derived, for the case that the automorphism $\omega$ is the identity morphism,
integrality properties of the partition function, i.e.\ of the correlator for
the torus without holes, relating it to the Cartan matrix of the category \HMod.
This shows in particular that the so obtained invariants
are non-zero.

In the present paper we solve the general problem of obtaining mapping class
group invariants at \emph{arbitrary} genus. Given a factorizable Hopf algebra
$H$ and a ribbon automorphism $\omega$ of $H$, we identify,
for each non-negative value of $g$ and of $n$, a natural invariant
$\Corw gn$ under the action of the mapping class group \Mapgn\ of Riemann
surfaces of genus $g$ with $n$ holes on a space of morphisms
that is obtained by the construction of Lyubashenko \cite{lyub6}
when taking all $n$ insertions to be given by the $H$-bimodule $\Fomega$.
Rephrased in conformal field theory terms, we identify natural candidates for
bulk correlation functions in a full, local conformal field theory and prove
their modular invariance, for any number of insertions and at any genus.

This paper is organized as follows. In Section \ref{sec:basic} we introduce
pertinent concepts and notations that are needed to describe the morphisms
$\Corw gn$ and to state our main result. This assertion, Theorem \ref{thm:main},
is formulated in Section \ref{sec:thm}. To establish it requires quite a few
detailed calculations which, for the case $\omega \eq \id_H$, take up Sections
\ref{sec:lemmata} and \ref{sec:proofmain}. Section \ref{sec:lemmata} is
essentially a collection of lemmas that are instrumental in Section
\ref{sec:proofmain}; their proofs can be safely skipped by readers primarily
interested in the results. Finally, in Section \ref{sec:omega} we complete
our main result by extending the analysis of Sections \ref{sec:lemmata} and
\ref{sec:proofmain}, and thus the proof of Theorem \ref{thm:main}, to the
case of non-trivial ribbon automorphisms.

We expect that our considerations generalize from the categories \HBimod\ to a
larger class of braided finite tensor categories \C. In particular, the analogue
of the $H$-bimodule $\Fomega$ should be the coend of a natural functor from
$\C\op \Times \C$ to the enveloping category $\C \,{\boxtimes}\, \C^{\rm rev}$.
Accordingly we formulate various statements in such a more general context, e.g.\
we give the invariants $\Corw gn$ first as morphisms in \HBimod\ in entirely
categorical terms (see formula \erf{Sk_morph}) before we present their concrete
expressions as linear maps (see \erf{CorrgnH}). However, a generalization of
our main result to such a context is still elusive. Concretely, the explicit
expressions for the coalgebra structure of the coend in \erf{pic-Hb-Frobalgebra}
involve the integral of the Hopf algebra $H$ over the field \ko\ 
defining the category. What we
are missing is a corresponding structure of the \emph{category} \HMod\ that
comes from the integral of $H$ and endows the coend with a coalgebra structure.


\section{Background}\label{sec:basic}

In this section we collect some basic definitions and notation for a class of
Hopf algebras and for representations of mapping class groups associated with
these Hopf algebras, which will be needed for formulating our main result,
Theorem \ref{thm:main}.

\subsection{Factorizable Hopf algebras}

Throughout this paper, the symbol \ko\ stands for an algebraically closed field
of characteristic zero, while $H$ is a \findim\ ribbon Hopf algebra over \ko,
which in addition is factorizable.  In the sequel, for brevity we will refer 
to $H$ just as a \emph{factorizable ribbon Hopf algebra}, suppressing 
finite-dimensionality over \ko.
All modules and bimodules in this paper will be finite-dimensional as \ko-vector
spaces as well. Similarly, all categories to be considered are assumed to be
abelian and \ko-linear, with all morphism sets being \findim\ \ko-vector spaces.
We denote by $m$, $\eta$, $\Delta$, $\eps$ and $\apo$ the product,
unit, coproduct, counit and antipode of the Hopf algebra $H$.

Let us recall the meaning of a Hopf algebra to be factorizable ribbon:

\begin{defi} ~\nxl1
(a)\, A Hopf algebra $H \,{\equiv}\, (H,m,\eta,\Delta,\eps,\apo)$ is called
\emph{quasitriangular} iff it comes with an invertible element $R$ of $H\otik H$
such that the coproduct and opposite coproduct are intertwined by $R$, i.e.\
$R\, \Delta\, R^{-1} \eq \tauHH \cir \Delta\,{\equiv}\,\Delta^{\!\rm op}_{}$, and
   \be
   (\Delta \oti \id_H) \circ R = R_{13}\cdot R_{23} \qquand
   (\id_H \oti \Delta) \circ R = R_{13}\cdot R_{12} \,.
   \labl{deqf-qt}
The element $R$ is called the \emph{R-matrix} of $H$.
\\[2pt]
(b)\, For $(H,R)$ a quasitriangular Hopf algebra, the invertible element
$Q \,{:=}\, R_{21}\,{\cdot}\, R$ of $H\otik H$
is called the \emph{monodromy matrix} $Q$ of $H$.
\\[2pt]
(c)\, A quasitriangular Hopf algebra $(H,R)$ is called a \emph{ribbon} Hopf
algebra iff it comes with a central invertible element $v\iN H$ obeying
   \be
   \apo \circ v = v \,, \qquad \eps \circ v = 1 \qquand
   \Delta \circ v = (v\oti v) \cdot Q^{-1} .
   \labl{def-ribbon}
$v$ is called the \emph{ribbon element} of $H$.
\\[2pt]
(d)\, A quasitriangular Hopf algebra $(H,R)$ is called \emph{factorizable}
iff the monodromy matrix $Q$ can be expressed as $\sum_\ell h_\ell \oti k_\ell$,
where $\{h_\ell\}$ and $\{k_\ell\}$ are two vector space bases of $H$.
\end{defi} \smallskip

Here for \ko-vector spaces $V$ and $W$ the linear map $\tau_{V,W}\colon V\otik W
\,{\stackrel\simeq\to}\, W\otik V$ is the flip map that exchanges the factors
in a tensor product. Also recall that a \findim\ Hopf algebra $H$ has a
left integral $\Lambda\iN H \,{\equiv}\, \Homk(\ko,H)$
and a right cointegral $\lambda\iN\Hs \,{\equiv}\, \Homk(H,\ko)$ (as well as
a right integral and a left cointegral), which are unique up to scalars.
Moreover, if $H$ is factorizable, then it is unimodular, i.e.\ the
integral $\Lambda\iN H$ is two-sided.
With applications in logarithmic conformal field theory in mind \cite{fgst}, it
should be appreciated that factorizability can be formulated for Hopf algebras
which do not have an R-matrix, but still a monodromy matrix.

Factorizability of $H$ is equivalent to invertibility of the \emph{Drinfeld map}
$f_Q \iN \Homk(\Hs,H)$, which is given by
$f_Q \,{:=}\, (d_H\oti \id_H) \cir (\idHs\oti Q)$, with $d$ the
evaluation morphism in \Vectk. The Drinfeld map, and likewise the map
$f_{Q^{-1}}$ that is obtained when replacing the monodromy matrix by its
inverse, is not just a linear map from \Hs\ to $H$, but it also intertwines
the left-coadjoint action of $H$ on \Hs\ and the left-adjoint action of $H$ on
itself. If $f_Q$ is invertible, then $f_Q(\lambda)$ is a non-zero multiple of
$\Lambda$. The normalizations of the integral and cointegral can then be chosen
in such a way that $\lambda\cir\Lambda \eq 1$ and $f_Q(\lambda) \eq \Lambda$
(this determines $\lambda$ and $\Lambda$ uniquely up to a common sign factor).
Doing so one arrives at the following identities \cite[(5.18)]{fuSs3},
which we present graphically:
  \Eqpic{fQS_Psi} {430} {41} { \put(-4,-4){
    \put(0,0) { \Includepichtft{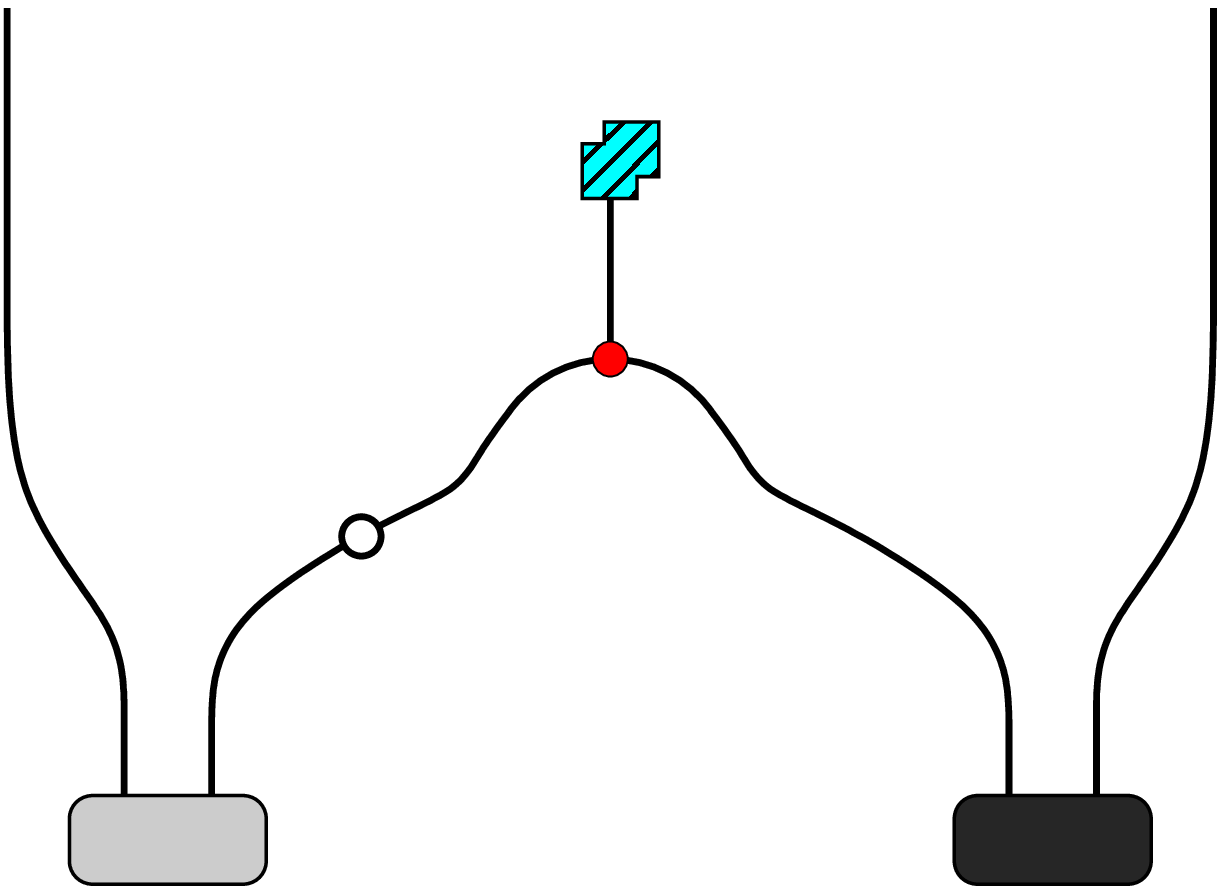}
  \put(-2.5,100)   {\sse$ H $}
  \put(14.5,3.4)   {\sse$ Q $}
  \put(43.6,35)    {\sse$ \apo $}
  \put(73.1,75.5)  {\sse$ \lambda $}
  \put(69.4,59.9)  {\sse$ m $}
  \put(128.3,3.2)  {\sse$ Q^{-1} $}
  \put(130,100)    {\sse$ H $}
  }
  \put(167,46)     {$ = $}
    \put(193,13) { \Includepichtft{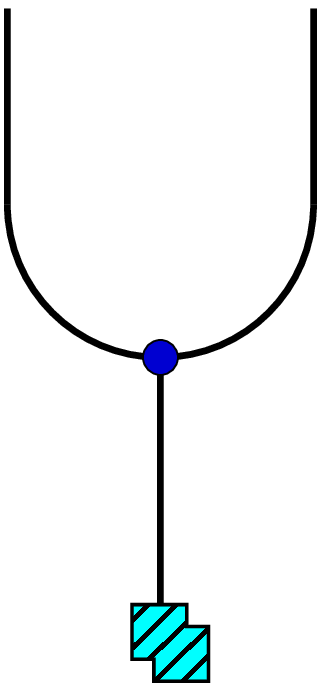}
  \put(-2.2,79)    {\sse$ H $}
  \put(7.7,-1)     {\sse$ \Lambda $}
  \put(14.4,40.2)  {\sse$ \Delta $}
  \put(30.8,79)    {\sse$ H $}
  }
  \put(256,46)     {$ = $}
  \put(294,0) { \Includepichtft{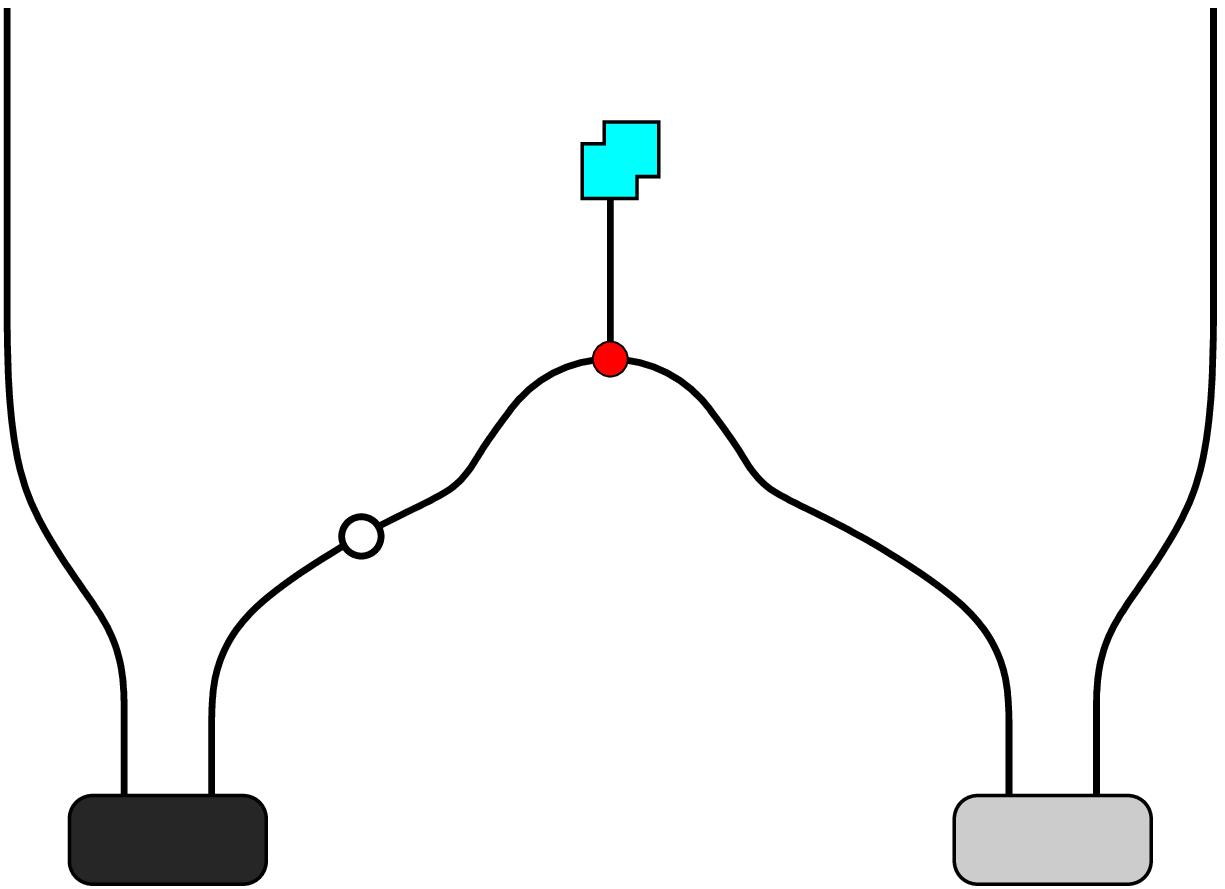}
  \put(-3.1,100)   {\sse$ H $}
  \put(31.1,3.2)   {\sse$ Q^{-1} $}
  \put(43.6,35)    {\sse$ \apo $}
  \put(73.1,75.5)  {\sse$ \lambda $}
  \put(69.4,59.9)  {\sse$ m $}
  \put(112.5,3.4)  {\sse$ Q $}
  \put(128.8,100)  {\sse$ H $}
  } } }
Such diagrams are to be read from bottom to top.
Below we will often suppress the labels indicating the product and
coproduct of $H$, as well as those for the antipode $\apo$ (which
is drawn as an empty circle) and its inverse $\apoi$ (full circle).

The following facts about the category \HMod\ of finite-dimensional left modules
over a factorizable ribbon
Hopf algebra $H$ are well-known: \HMod\ is a braided
rigid monoidal category, and even a factorizable ribbon category. Moreover, it
is a finite tensor category in the sense of \cite{etos}, i.e.\ it has finitely
many isomorphism classes of simple objects, each of them has a projective
cover, and every object has finite length.

We endow the category \HBimod\ of finite-dimensional $H$-bimodules in an
analogous manner with the structure of a finite factorizable ribbon category:
We use the pull-back of left and right actions along the coproduct to obtain
the structure of a monoidal category: for $H$-bimodules $(X,\rho_X,\ohr_X)$
and $(Y,\rho_Y,\ohr_Y)$ we define the left and right actions of $H$ on the
tensor product vector space $X \otik Y$ by
   \be
   \bearl
   \rho_{X\otimes Y}^{} := (\rho_X \oti \rho_Y) \circ (\id_H \oti \tau_{H,X} \oti
                        \id_Y) \circ (\Delta \oti \id_X \oti \id_Y)  \qquand
   \Nxl3
   \ohr_{X\otimes Y}^{} := (\ohr_X \oti \ohr_Y) \circ (\id_X \oti \tau_{Y,H} \oti
                        \id_H) \circ (\id_X \oti \id_Y \oti \Delta) \,.
   \eear
   \labl{def-tp}
The monoidal unit for this tensor product is the one-dimensional vector
space \ko\ with left and right $H$-action given by the counit,
$\one_{H\text{-Bimod}} \eq (\ko,\eps,\eps)$. To endow the monoidal category
\HBimod\ with a braiding, we use the action of the R-matrix of $H$ from the
right and the action of its inverse from the left. When regarding the resulting
braiding morphism $c_{X,Y}^{}$ in \HBimod\ as a linear map, i.e.\ as a morphism
in the category \Vect\ of \ko-vector spaces, we represent it pictorially as
   \eqpic{bibraid} {130} {46} {
   \put(0,52)      {$ c_{X,Y}^{} ~= $}
     \put(59,0) {\Includepichtft{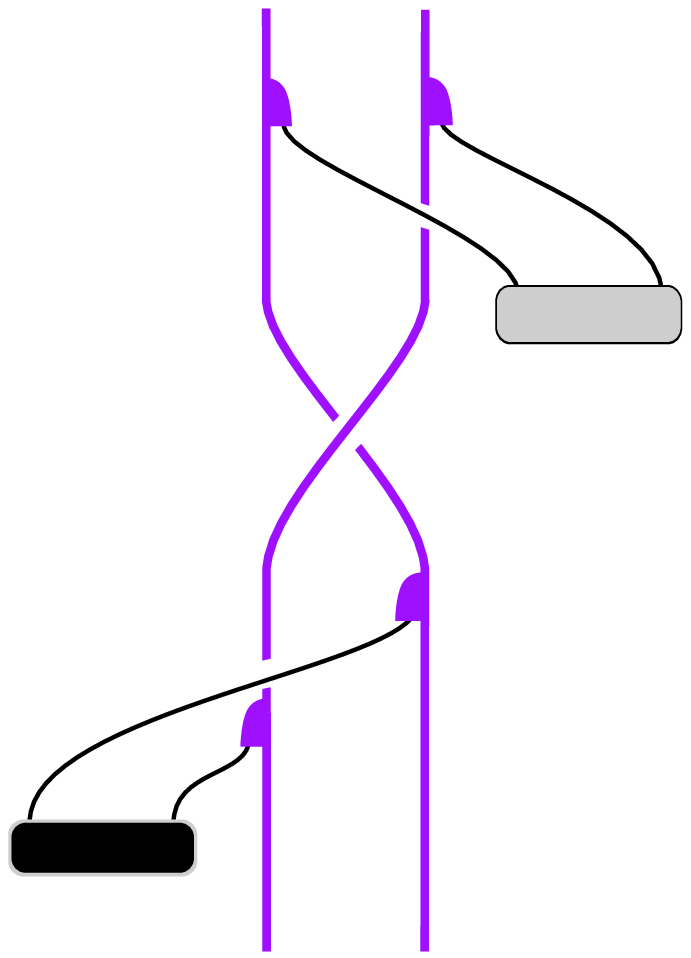}
   \put(-17,9)     {\sse$ R^{-1} $}
   \put(15.5,93)   {\sse$ \ohr_Y^{} $}
   \put(24,-9.2)   {\sse$ X $}
   \put(25.8,108)  {\sse$ Y $}
   \put(29.8,25)   {\sse$ \rho_X^{} $}
   \put(42,56)     {\sse$ \tau_{X,Y}^{} $}
   \put(42.3,-9.2) {\sse$ Y $}
   \put(42.8,108)  {\sse$ X $}
   \put(47.2,38.8) {\sse$ \rho_Y^{} $}
   \put(50.2,94)   {\sse$ \ohr_X^{} $}
   \put(75,68)     {\sse$ R $}
   } }
Here the quarter disks refer to left or right actions of the Hopf
algebra $H$, while the crossing is just the flip map of vector spaces.
(As the braiding on the category of vector spaces is symmetric, the use of
over- and under-crossings in these pictures does not contain any mathematical
information, but is merely for graphical convenience.)

Using the equivalence of $H$-bimodules with $H\Otik H$-modules, it is not hard
to verify that these prescriptions endow
\HBimod\ with the structure of a braided monoidal category.
This is the braided monoidal category we are interested in in this paper. We
endow it with more structure: we introduce right and left duals by associating
to a bimodule $X \eq (X,\rho,\ohr)\iN\HBimod$ the bimodules
   \be
   X^\vee := (X^*,\rhov,\ohrv)  \qquand {}^{\vee}\!X:= (X^*,\rhoV,\ohrV\,)
   \labl{def-duals}
with left and right $H$-actions defined by
   \Eqpic{Hbim_dualactions} {420} {49} { \put(-21,19){
   \put(0,35)    {$\rhov ~:= $}
   \put(34,0)  {\Includepichtft{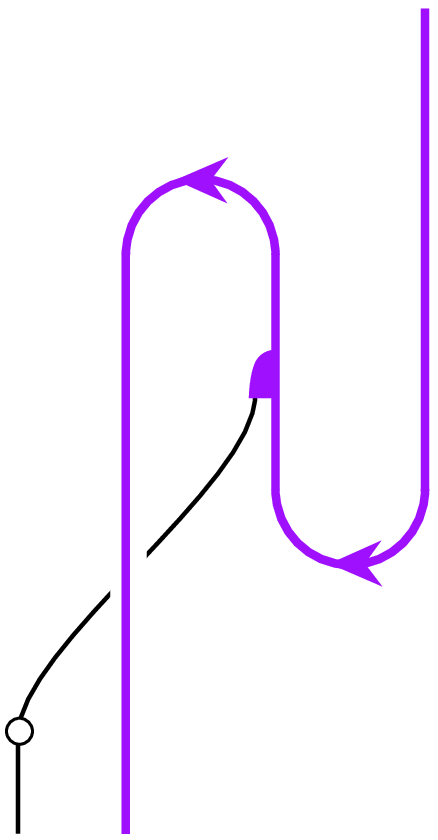}
   \put(-2.5,-9.2) {\sse$ H $}
   \put(10,-9.2) {\sse$ X^*_{} $}
   \put(31.3,49.8) {\sse$ \rho $}
   \put(42.2,94) {\sse$ X^*_{\phantom|} $}
   }
   \put(126,35)  {$\ohrv ~:= $}
   \put(175,0) {\Includepichtft{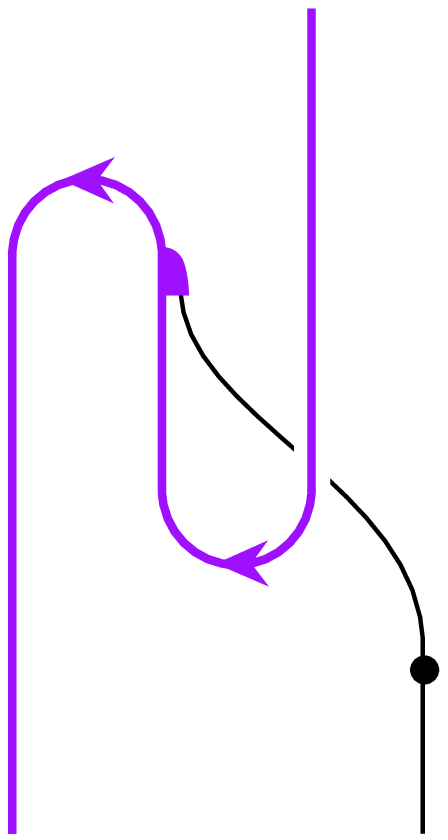}
   \put(-5,-9.2) {\sse$ X^*_{} $}
   \put(11,60.7) {\sse$ \ohr $}
   \put(29.8,94) {\sse$ X^*_{\phantom|} $}
   \put(41,-9.2) {\sse$ H $}
   }
   \put(259,35)  {$\rhoV ~:= $}
   \put(298,0) {\Includepichtft{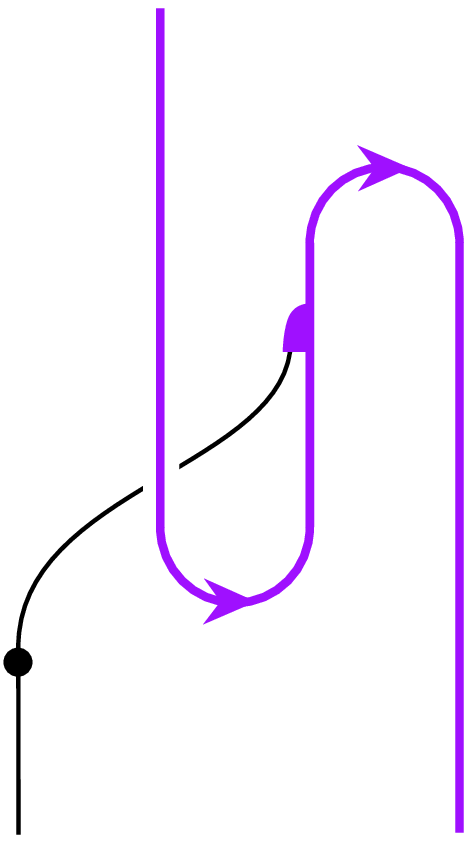}
   \put(-2.4,-9.2) {\sse$ H $}
   \put(13.6,94) {\sse$ X^*_{\phantom|} $}
   \put(35,54.7) {\sse$ \rho $}
   \put(44,-9.2) {\sse$ X^*_{} $}
   }
   \put(390,35)  {$\ohrV ~:= $}
   \put(429,0) {\Includepichtft{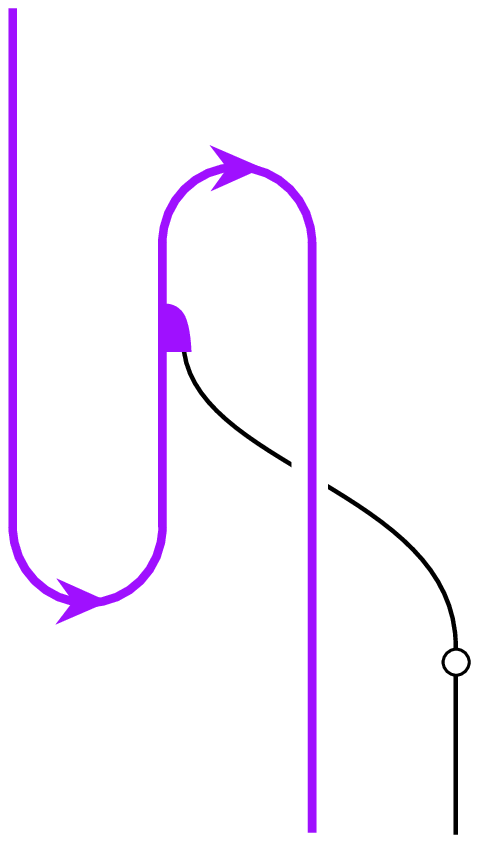}
   \put(11,54.1) {\sse$ \ohr $}
   \put(27,-9.2) {\sse$ X^*_{\phantom|} $}
   \put(-2.9,94) {\sse$ X^*_{\phantom|} $}
   \put(45,-9.2) {\sse$ H $}
   } } }
The left and right dualities are naturally compatible -- the category \HBimod\
has a natural structure of a sovereign tensor category. To see this, denote
by $t \,{:=}\, uv^{-1}$ the product of the Drinfeld element
   \be
   u := m \circ (\apo\oti\id_H) \circ \tauHH \circ R  ~\in H
   \labl{u-R}
with the inverse of the ribbon element $v$ of $H$; $t$ is a group-like element 
of $H$. Then the family of endomorphisms
   \eqpic{pivX} {170} {38} {
   \put(0,39)    {$ \pi_X ~:= $}
   \put(50,0)  {\Includepichtft{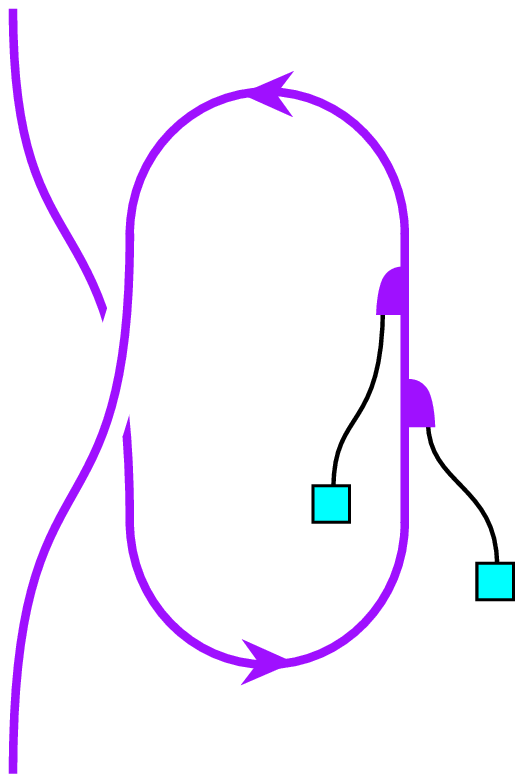}
   \put(-4.5,-8) {\sse$ \Xs $}
   \put(-3,88)   {\sse$ \Xs $}
   \put(28,28)   {\sse$ \uvi $}
   \put(58,20)   {\sse$ \uvi $}
   }
   \put(127,39)  {$\in \Endk(X^*_{\phantom|}) $}
   }
is a natural monoidal isomorphism between the left and right duality functors
and thus endows \HBimod\ with the structure of a sovereign tensor category.
Being sovereign, \HBimod\ is also endowed with a balancing and thus has the
structure of a ribbon category. From here on, the symbol \HBimod\ stands for this
sovereign ribbon category.  Explicitly, the twist endomorphism $\theta_X$ of an
$H$-bimodule $(X,\rho,\ohr)$ is given by acting with the ribbon element $v$ from
the left and with its inverse $v^{-1}$ from the right \cite[Lemma\,4.8]{fuSs3},
  \be
  \theta_X = \rho \circ (\id_H \oti \ohr) \circ (v \oti \id_X \oti v^{-1}) \,.
  \labl{deftwist}

\begin{rem}
The category \HBimod\ with this structure of sovereign ribbon category is
braided equivalent to $(H{\otimes_\ko} H\op_{})$\Mod\ and thus to
$\HMod \,{\boxtimes}\, \HMod^{\rm rev}$. This is a very
natural category indeed: it can be regarded as a categorification of the notion
of enveloping algebra. Factorizibility of \HBimod\ amounts to the statement
that \HBimod\ is braided equivalent to the Drinfeld center of \HMod.
\end{rem}


\subsection{The handle Hopf algebra and half-monodromies}\label{subsec:hHa}

We now recall \cite{maji25,lyub8} that any finite sovereign braided tensor
category \C\ contains a canonical Hopf algebra object. It can be constructed
as the coend
  \be
  \HK = \Coend FX = \coend X
  \labl{defK}
of the functor $F$ that maps a pair $(X,Y)$ of objects of \C\ to the object
$X^\vee \oti Y \iN \C$. As a coend, $K$ comes with a dinatural family
$(\iHK_X)_{X\in\C}$ of morphisms $\iHK_X\iN\Hom_\C(X^\vee{\otimes}\, X, \HK)$.
The structure morphisms of the Hopf algebra object \HK\ are obtained with the
help of the family $\iHK$ and the braiding and duality of \C. They can be
found e.g.\ in \cite{lyub6,vire4}; we refrain from reproducing them here.
We refer to $K$ as the \emph{handle Hopf algebra} for the category \C.

The Hopf algebra \HK\ is also endowed with a Hopf pairing. We call the finite
sovereign braided tensor category \C\ a \emph{\fftc} iff this Hopf pairing is
non-degenerate. This terminology is motivated by the fact that in the case of
$\C \eq \HMod$, non-degeneracy of the Hopf pairing is equivalent to invertibility
of the Drinfeld map, i.e.\ to factorizability of $H$ \cite[Eq.\,(5.14)]{fuSs3}.
If and only if \C\ is factorizable in this sense, then the Hopf algebra \HK\ has
(two-sided) integrals and cointegrals \cite[Prop.\,5.2.9]{KEly}.

Let now $H$ be a factorizable ribbon Hopf algebra and $\C \eq \HMod$. In this case,
which was already considered in \cite{lyub8}, the coend $\HK\eq H^*\coa$ is the
dual space $\Hs$ endowed with the coadjoint action of $H$. Similarly, if \C\ is
the braided tensor category \HBimod, the bimodule $\HK \,{=:}\, \Haa$
can be described as follows: the underlying vector space is the tensor
product $\Hs\otik\Hs$, and introducing the morphisms
   \eqpic{def_lads} {370} {49} {
   \put(0,52)      {$ \rho\coa ~:= $}
   \put(42,0) { \includepichopf{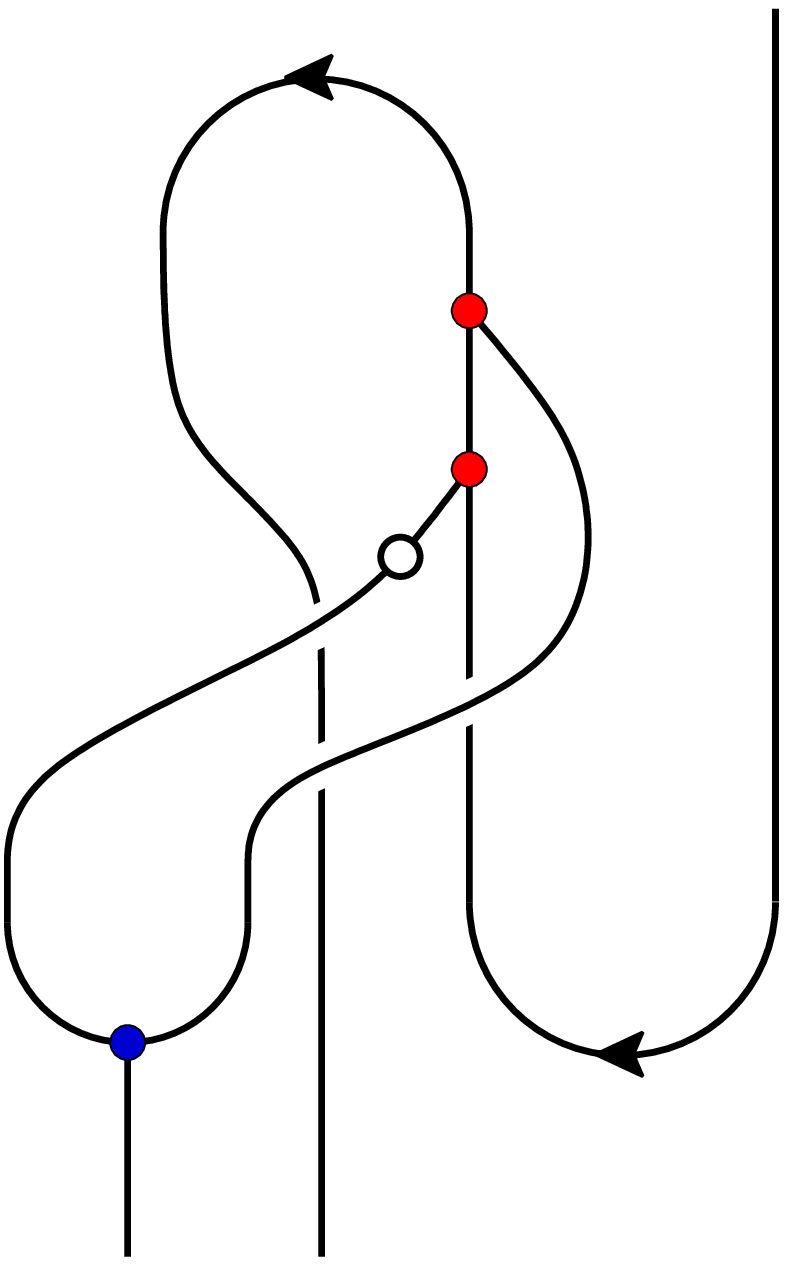}
   \put(6.5,-9.2){\sse$ H $}
   \put(23.4,-9.2) {\sse$ \Hss $}
   \put(32.3,65.6) {\sse$ \apo $}
   \put(64.5,114)  {\sse$ \Hss $}
   }
   \put(167,52)  {and}
   \put(228,52)    {$ \ohr\coar ~:= $}
   \put(278,0) { \includepichtft{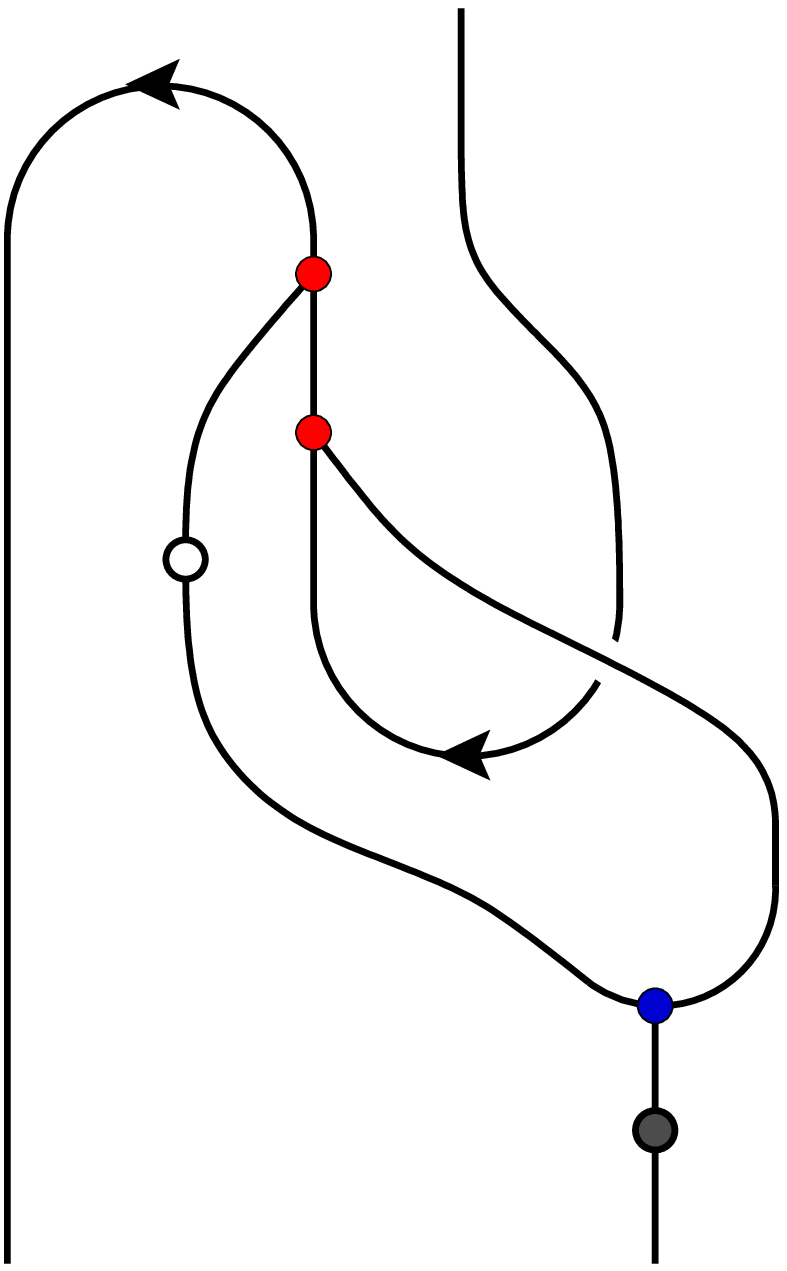}
   \put(-4,-9.2)   {\sse$ \Hss $}
   \put(8.8,60.5)  {\sse$ \apo $}
   \put(37,114)    {\sse$ \Hss $}
   \put(53.4,-9.2) {\sse$ H $}
   \put(61.4,10.5) {\sse$ \apoi $}
   } }
the bimodule structure is given by
  \be
  \Haa = (\Hs \Otik\, \Hs,\rho\coa\oti\idHs,\idHs\oti\rho\coar) \,.
  \labl{defHaa}
(For the dinatural family of the coend \Haa\ see \cite[Eq.\,(A.29)]{fuSs3}.)

\medskip

Below we will recall that representations of mapping class groups can be
constructed with the help of morphisms involving the Hopf algebra \HK\ in a
\fftc\ \C. Here we provide the main building blocks of those morphisms. It
is natural to use the universal property of \HK\ as a coend
to specify these building blocks in terms of dinatural families.
\\[-2.36em]

\def\leftmargini{1.57em}~\\[-1.45em]\begin{enumerate}\addtolength{\itemsep}{-3pt}

\item
Denote by $(\theta_X)_{X\in\C}$ the twist on the tensor category
$\C$. Then we define an endomorphism $\TK$ of $\HK$ in terms of dinatural
families by
   \eqpic{p9} {92}{24} {
   \put(0,0)   {\Includepichtft{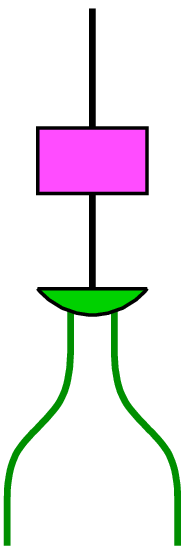}
   \put(-11.1,39.7){\small$\TK$}
   \put(-5,-9.2)   {\sse$ X^{\!\vee} $}
   \put(17.4,25.7) {\sse$ \iHK_X $}
   \put(15.3,-9.2) {\sse$ X $}
   \put(6.8,63.3)  {\sse$ \HK $}
   }
   \put(47,28)     {$=$}
   \put(83,0)  {\Includepichtft{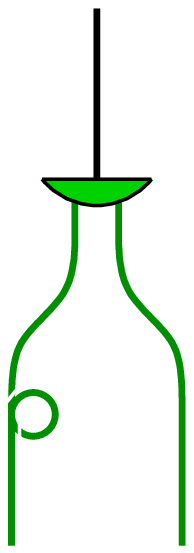}
   \put(-11.9,14)  {\sse$ \theta_{\!X^{\!\vee}_{}}^{}$}
   \put(-3,-9.2)   {\sse$ X^{\!\vee} $}
   \put(8.4,63.3)  {\sse$ \HK $}
   \put(18.9,37.7) {\sse$ \iHK_X $}
   \put(17.4,-9.2) {\sse$ X $}
   }
   \put(133,46)    {\catpic}
   }
Here we indicate explicitly in the figure that the diagram has to be read in
the ribbon category \C\ (rather than, as in all pictures displayed so far, in
\Vectk), so that in particular over- and under-crossings must be carefully distinguished.

\item
Similarly, the monodromies
$c_{Y^\vee_{\phantom,}\!,X} \cir c_{X,Y^\vee_{\phantom,}}$
of \C\ allow us to deduce an endomorphism \QQ\ of $K\oti K$ from the equality
   \eqpic{QHH} {135} {45} {
   \put(0,0)  {\Includepichtft{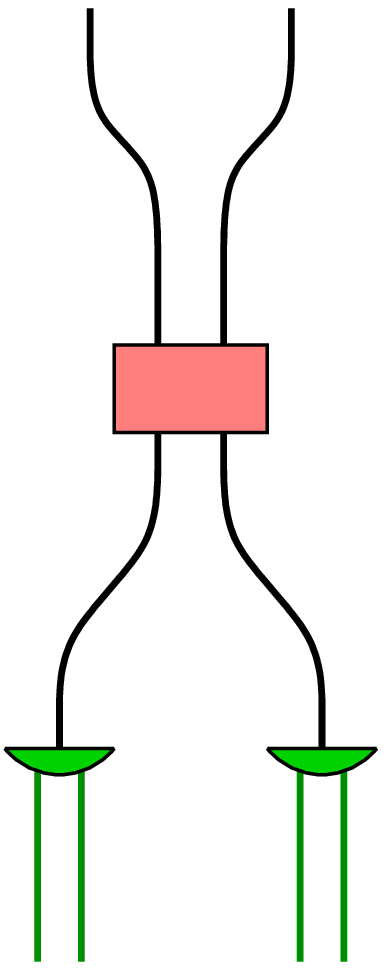}
   \put(-4.2,60.1){\small$ \QQ $}
   \put(-9.8,21)  {\sse$ \iHK_X $}
   \put(-4,-9.2)  {\sse$ X^{\!\vee} $}
   \put(5.3,109)  {\sse$ \HK $}
   \put(7,-9.2)   {\sse$ X $}
   \put(26.5,-9.2){\sse$ Y^{\!\vee} $}
   \put(28.8,109) {\sse$ \HK $}
   \put(38,-9.2)  {\sse$ Y $}
   \put(42.4,21)  {\sse$ \iHK_Y $}
   }
   \put(62,50)    {$ = $}
   \put(97,0) {\Includepichtft{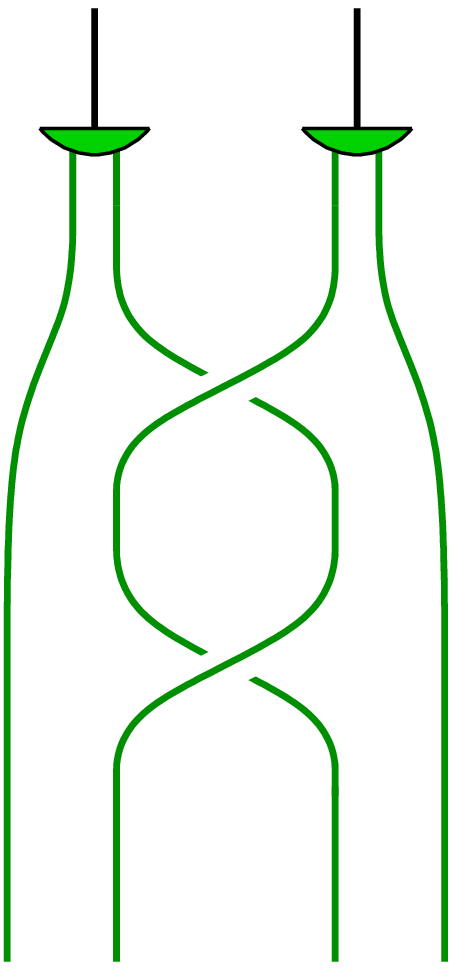}
   \put(-6,-9.2)  {\sse$ X^{\!\vee} $}
   \put(-5.8,89)  {\sse$ \iHK_X $}
   \put(8,-9.2)   {\sse$ X $}
   \put(6.1,109)  {\sse$ \HK $}
   \put(22.8,27.3){\sse$ c $}
   \put(22.8,58.4){\sse$ c $}
   \put(32,-9.2)  {\sse$ Y^{\!\vee} $}
   \put(36,109)   {\sse$ \HK $}
   \put(46.2,89)  {\sse$ \iHK_Y $}
   \put(46,-9.2)  {\sse$ Y $}
   }
   \put(181,91)   {\catpic}
   }
of dinatural transformations.

\item
We use the endomorphism \QQ\ to define an endomorphism $S_\HK\iN\EndC(K)$ by
   \be
   \SK := (\eps_\HK \oti \id_\HK) \circ \QQ \circ (\id_\HK \oti \Lambda_\HK) \,,
   \labl{S-HK}
where $\eps_\HK$ and $\Lambda_\HK$ are the counit and the two-sided integral of
the Hopf algebra \HK, respectively.

\item
For any object $Y\iN\C$ we use the monodromies $c_{Y,-} \cir c_{-,Y}$ to
define an endomorphism $\QB_Y\iN\EndC(\HK\oti Y)$:
   \eqpic{QHX} {120} {44} { \put(0,1){
   \put(0,0)  {\begin{picture}(0,0)(0,0)
         \scalebox{.38}{\includegraphics{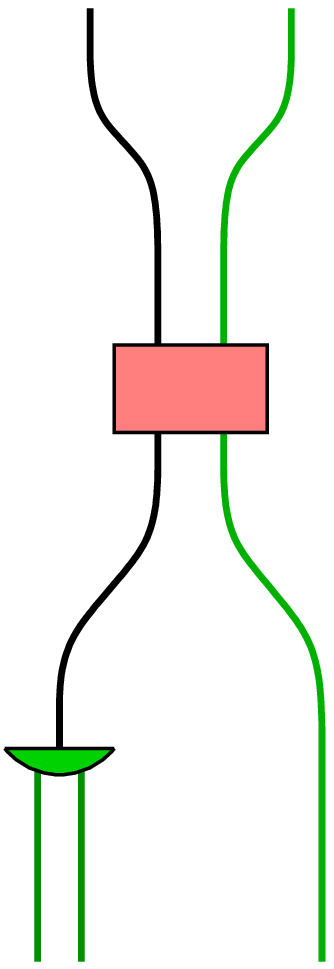}}\end{picture}
   \put(-6.6,60.5){$ \QB_Y $}
   \put(-4,-9.2)  {\sse$ X^{\!\vee} $}
   \put(6.8,109)  {\sse$ \HK $}
   \put(7,-9.2)   {\sse$ X $}
   \put(31.3,-9.2){\sse$ Y $}
   \put(30.4,109) {\sse$ Y $}
   \put(14.3,21.5){\sse$ \iHK_X $}
   }
   \put(60,50)    {$ := $}
   \put(94,0) {\begin{picture}(0,0)(0,0)
         \scalebox{.38}{\includegraphics{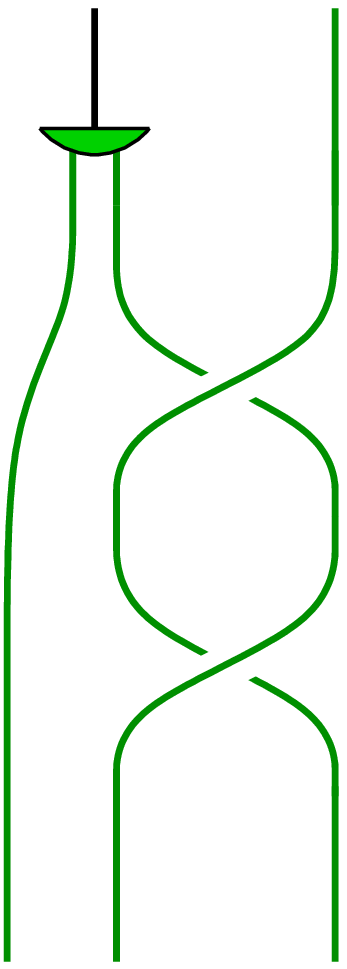}}\end{picture}
   \put(-6,-9.2)  {\sse$ X^{\!\vee} $}
   \put(7.4,109)  {\sse$ \HK $}
   \put(8,-9.2)   {\sse$ X $}
   \put(22.2,27.3){\sse$ c $}
   \put(22.2,58.4){\sse$ c $}
   \put(34,-9.2)  {\sse$ Y $}
   \put(34.3,109) {\sse$ Y $}
   \put(17.8,90)  {\sse$ \iHK_X $}
   } }
   \put(165,93)    {\catpic}
   }

\item
It can now be checked that for any object $Y\iN\C$ the morphism
$\rho_Y^\HK \iN \HomC(\HK\oti Y,Y)$ defined by
  \be
  \rho_Y^K := (\eps_\HK\oti \id_Y)\cir \QB_Y
  \labl{def-rho}
endows $Y$ with the structure of a left \HK-module.
\end{enumerate}

\begin{rem}\label{rem:YD}
Via the duality, the dinaturality morphisms of the coend $K$ provide
a right coaction $(\id_X \oti \iHK_X) \cir (b_X \oti \id_X)$ of $K$ on any
object of \C. This coaction fits together with the action $\rho^K_X$ to
endow the object $X$ with a structure of left-right Yetter-Drinfeld module in
the monoidal category \C. Moreover, any morphism in \C\ is compatible with
this structure, so that we indeed have a fully faithful embedding of \C\ into
the category of left-right Yetter-Drinfeld modules over $K$.
\end{rem}


\subsection{Representations of mapping class groups}\label{ssec:mpg}

One of the main results of \cite{lyub11} is the construction of representations
of mapping class groups for surfaces with holes (i.e., with open disks excised).
Denote by \Mapgn\ the mapping class group of closed oriented surfaces of genus
$g$ with $n$ boundary components.
Various finite presentations of \Mapgn\ have been discussed in the
literature. Since our purpose is to check invariance under the mapping class
group, for us it is sufficient to display a finite set of generators.
One such set of generators arises from the exact sequence
  \be
  1\rightarrow\BG gn\rightarrow\Mapgn\rightarrow \Map g0\rightarrow 1 \,,
  \ee
(compare \cite[Thm.\,9.1]{FAma}), where $\BG gn$ is a central extension of the
surface braid group by $\zet^n$. Owing to this sequence, one can take as a
generating set the union of those for a presentation of \Map g0\ \cite{wajn}
and for a presentation of $\BG gn$ \cite{scotG3};
this particular presentation has been advocated and used in \cite{lyub6,lyub11}.

To describe these generators, we first introduce a collection of cycles
$a_m$, $b_m$, $d_m$, $e_m$ and $t_{j,k}$ on a genus-$g$ surface $\Surf gn$
with $n$ holes.  The following picture indicates these cycles on $\Surf gn$
(see \cite[Figs.\,2\,\&\,6]{lyub6}):
  \Eqpic{surf_gn_PIC} {420} {42} { \put(35,0){ \setlength\unitlength{1.3pt}
    \put(-42,0)  {\INcludepichtft{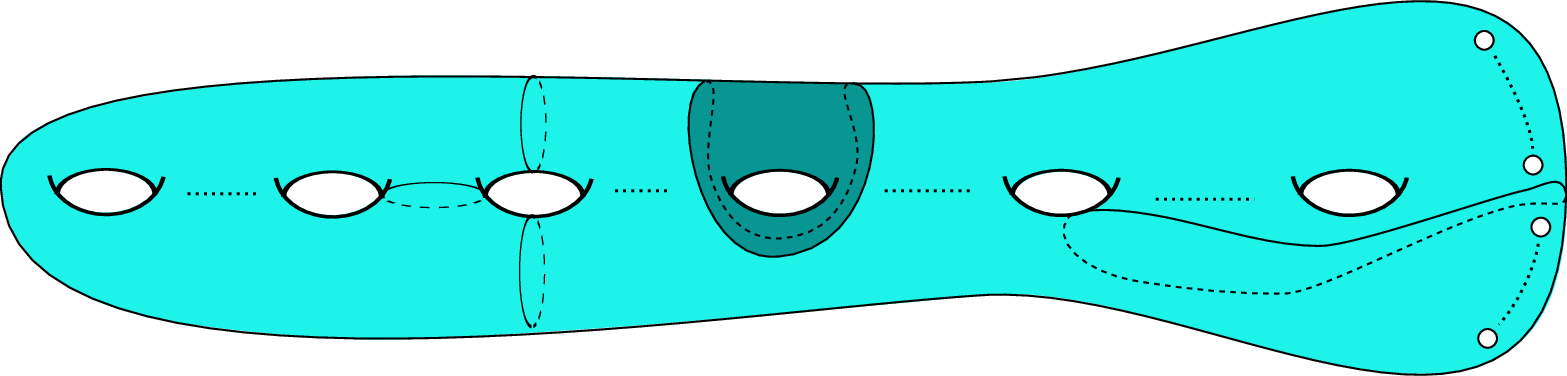}{65}}
  \put(-23,38)      {\sse$b_1$}
  \put(23,38.5)     {\sse$b_{m{-}1}$}
  \put(49,46)      {\sse$a_m$}
  \put(71,38)      {\sse$b_m$}
  \put(61,54)      {\sse$d_m$}
  \put(61,22)      {\sse$e_m$}
  \put(125,54)     {\sse$S_l$}
  \put(127,38)     {\sse$b_{l}$}
  \put(188,38)     {\sse$b_{k}$}
  \put(253,39)     {\sse$b_{g}$}
  \put(223,26)     {\sse$t_{j,k}$}
  \put(291,80)     {\sse$U_1$}
  \put(307,44)     {\sse$U_j$}
  \put(307,31)     {\sse$U_{j+1}$}
  \put(291,-1)     {\sse$U_n$}
  } }
For brevity we refer to the inverse Dehn twist about any of these
cycles by the same symbol as for the cycle itself. The shaded region in the
picture \eqref{surf_gn_PIC} is a neighborhood of the $l$th handle with the
topology of a one-holed torus; we denote this region by $F_l$, and by $F_l'$ the
slightly smaller neighborhood that is indicated by the dotted line inside $F_l$.
The generators considered in \cite{lyub6,lyub11} are then the following:

\def\leftmargini{1.53em}~\\[-2.65em]\begin{enumerate}\addtolength{\itemsep}{-3pt}
    \item
Braidings $\wi$, for $i \eq 1,2,...\,,n{-}1$, which permute the $i$th and
$i{+}1$\,st boundary circle.
    \item
Dehn twists $\Ri$ about the $i$th boundary circle, for $i \eq 1,2,...\,,n$.
    \item
Homeomorphisms $S_l$, for $l \eq 1,2,...\,,g$, which act as the identity
outside the region $F_l$ and as a modular S-transformation of the one-holed
torus $F_l' \,{\subset}\, F_l^{}$.
    \item
Inverse Dehn twists in tubular neighborhoods of the cycles $a_m$
and $e_m$, for $m \eq 2,3,$ $...\,,g$.
    \item
Inverse Dehn twists in tubular neighborhoods of the cycles $b_m$ and $d_m$,
for $m \eq 1,2,$ $...\,,g$.
    \item
Inverse Dehn twists in tubular neighborhoods of the cycles $t_{j,k}$, for
$j\eq 1,2,...\,,n{-}1$ and $k \eq 1,2,...\,,g$.
\end{enumerate}

\noindent
This system of generators is not minimal. Specifically, the generators $\sk$
can be expressed in terms of the generators $\bk$ and $\dk$. Nevertheless
we keep the $\sk$ in the list, because in the case
$g \eq 1$ and $n \eq 0$, $S \eq \sk$ and $T \eq \dk$ are the usual S- and
T-transformations generating the modular group \slz. As already pointed out,
since the aim of the present article is to determine invariants, it is
sufficient to know the action of some set of generators; relations are not
needed.

\medskip

The representations of our interest involve decorations of the boundary circles
of $\Surf gn$ by (not necessarily distinct) objects $X_1, X_2, ...\,, X_n$ of a
\fftc\ \C. Denote by $\mathfrak N \eq \mathfrak N(X_1{,}...{,} X_n)$ the
subgroup of the symmetric group $\mathfrak S_n$ that is generated by
those permutations $\sigma\iN \mathfrak S_n$ for which $X_i$ and $X_{\sigma(i)}$
are non-isomorphic for at least one value of the label $i$, and set
  \be
  X := \bigoplus_{\sigma\in \mathfrak N}
  X_{\sigma(1)} \oti X_{\sigma(2)} \oti \cdots \oti X_{\sigma(n)} \,.
  \labl{defX}
Then the representation space for \Mapgn\ relevant to us is the vector space
  \be
  V^{X}_{g:n} := \HomC(K^{\otimes g},X)
  \labl{Vpqg}
of morphisms of \C. In line with the designation `handle Hopf algebra' for \HK,
this involves one copy of \HK\ for each handle of $\Surf gn$.

To describe the action of \Mapgn\ on the space $V^{X}_{g:n}$, we introduce the
following collections of morphisms: First, the endomorphisms \LAwi\ with $i \iN
\{1,2,...\,,n{-}1\}$ and \LARj\ with $j \iN \{1,2,...\,,n\}$ of $X$ that act as
  \be
  \bearl
  \LAwi{\big|}_{X_1\otimes X_2\otimes\cdots\otimes X_n}
  := \id_{X_1\otimes\cdots\otimes X_{i-1}} \oti c_{X_i,X_{i+1}}
          \oti\id_{X_{i+2}\otimes\cdots\otimes X_n}
  \qquand \\{}\\[-.6em]
  \LARj{\big|}_{X_1\otimes X_2\otimes\cdots\otimes X_n}
  := \id_{X_1\otimes\cdots\otimes X_{j-1}} \oti \theta_{X_j}
          \oti\id_{X_{j+1}\otimes\cdots\otimes X_n}
  \eear
  \labl{LyubactC1}
on the direct summand $X_1\oti X_2\oti\cdots\oti X_n$ of $X$ and analogously on
the other summands in \erf{defX}, with $c$ and $\theta$ the braiding and the
twist of \C, respectively. Second, the endomorphisms
  \be
  \bearl
  \LASk := \id_{\HK^{\otimes g-k}_{}}\oti\SK\oti\id_{\HK^{\otimes k-1}}
  \,, \\{}\\[-.6em]
  \LAal := \id_{\HK^{\otimes g-l}_{}}
           \oti[\QQ\cir (\TK\oti\TK)]\oti \id_{\HK^{\otimes l-2}_{}}
  \,, \\{}\\[-.6em]
  \LAbk := \id_{\HK^{\otimes g-k}_{}}
           \oti(\SK^{-1}\cir \TK\cir \SK)\oti\id_{\HK^{\otimes k-1}_{}}
  \,, \\{}\\[-.6em]
  \LAdk := \id_{\HK^{\otimes g-k}_{}}\oti\TK\oti\id_{\HK^{\otimes k-1}}
  \,, \\{}\\[-.6em]
  \LAel := \id_{\HK^{\otimes g-l}_{}}\oti \big[ (\TK\oti
           \theta_{\HK^{\otimes l-1}}) \circ \QB_{K^{\otimes l-1}_{}} \big]
  \eear
  \labl{LyubactC2}
of $K^{\otimes g}_{}$, where $k \iN \{1,2,,...\,,g\}$ and $l \iN \{2,3,,...\,,g\}$,
and where $\TK$, $\SK$, \QQ\ and $\QB_Y$ are the morphisms introduced
in Section \ref{subsec:hHa}. And third, for any $j\iN \{1,2,...\,,n{-}1\}$ and
$k \iN \{1,2,...\,,g\}$ the linear map \LAtjk\ that maps any morphism
$f\iN \HomC(K^{\otimes g},X_1\oti\cdots\oti X_n)$ to
  \be
  \bearll
  \LAtjk(f) := \!\!& \big(\, \big[\, ( \id_{X_1\otimes\cdots\otimes X_j}
  \oti \tilde d_{X_{j+1}\otimes\cdots\otimes X_n} )
  \circ
  ( f \oti \id_{\Vee X_n\otimes\cdots\otimes \Vee X_{j+1}} )
  \\{}\\[-.6em]& \quad
  \circ\, \{ \id_{K^{\otimes g-k}_{}} \oti [
  \QB_{K^{\otimes k-1}\otimes \Vee X_n\otimes\cdots\otimes \Vee X_{j+1}} \cir
  (\TK \oti \theta_{K^{\otimes k-1}\otimes \Vee X_n\otimes\cdots\otimes
  \Vee X_{j+1}}) ] \} \,\big]
  \\{}\\[-.8em]& \hspace*{23.9em}
  \otimes\, \id_{X_{j+1}\otimes\cdots\otimes X_n} \,\big)
  \\{}\\[-.6em]& \hspace*{21.8em}
  \circ\, \big(
  \id_{\HK^{\otimes g}_{}} \oti \tilde b_{X_{j+1}\otimes\cdots\otimes X_n} \big)
  \eear
  \labl{LyubactC3}
in $\HomC(K^{\otimes g},X_1\oti\cdots\oti X_n)$ and acts analogously on
morphisms in $\HomC(K^{\otimes g},X_{\sigma(1)} \oti X_{\sigma(2)} $ 
    \linebreak[0]$
{\otimes}\, \cdots \oti X_{\sigma(n)})$ for $\sigma\iN \mathfrak N$. (A
graphical description of the map $\LAtjk$ will be given in picture
\erf{t_act} below.)

Then we can rephrase the results of \cite[Sect.\,4]{lyub6} and
\cite[Sect.\,3]{lyub11} as follows:

\begin{prop}\label{Lyubact_prop}
For any collection of objects $X_i$ of \C\ as above, the linear endomorphisms
of the space $V^{X}_{g:n}$ that act as
  \be
  \pi^{X}_{g:n}(\gamma) := \left\{
  \bearll
  \big( z_\gamma \big)_{*}^{}
  & {\rm for}~~ \gamma \eq \omega_i,\,\Ri
  ~~(\,i \eq 1,2,...\,,n{-}1 ~{\rm resp.}~ i \eq 1,2,...\,,n\,)
  \,, \\{}\\[-.6em]
  \big( z_\gamma \big)^{*}_{}
  & {\rm for}~~ \gamma \eq S_k ~(\,k \eq 1,2,...\,,g\,)
  \\{}\\[-.8em]&
  ~{\rm or}~~ \gamma \eq a_m, b_m, d_m, e_m ~(\,m \eq 1,2,...\,,g
   ~{\rm resp.}~ m \eq 2,3,...\,,g\,)
  \,, \\{}\\[-.6em]
  \LAtjk & {\rm for}~~ \gamma \eq t_{j,k}
  ~(\,j\eq 1,2,...\,,n{-}1 ~{\rm and}~ k \eq 1,2,...\,,g\,)
  \eear \right.
  \labl{piXgn}
generate a projective representation $\pi^{X}_{g:n}$ of the
mapping class group \Mapgn\ on the vector space $V^{X}_{g:n}$.
\end{prop}

\begin{rem}
In \cite{lyub6,lyub11} the role of source and target in the vector space
$V^{X}_{g:n}$ \erf{Vpqg}, and correspondingly the role of pre- and
post-composition in the formulas \erf{piXgn}, is interchanged.
Accordingly in our description the inverses of the morphisms $z_\gamma$
used in \cite{lyub6,lyub11} appear.
\end{rem}

\begin{rem}\label{rem:pq}
In applications (for details see Remark \ref{rem:pq2} below) one is also
interested in the following variant. Partition the set of boundary circles
of the surface into two subsets of sizes
$p$ and $q$ and denote by \Mapgpq\ the subgroup of the mapping class group
\Mapgppq\ that leaves each of these two subsets separately invariant.
Further, denote the corresponding decorations by $X_1, X_2, ...\,, X_p$ and
by $Y_1, Y_2, ...\,, Y_q$, respectively, and define objects $X$ and $Y$
analogously as in \erf{defX}. Finally note that the right duality of \C\
provides a linear isomorphism
  \be
  \varphi:\quad \HomC(K^{\otimes g}{\otimes}\,Y,X)
  \stackrel\cong\longrightarrow \HomC(K^{\otimes g},X\oti Y^\vee) \,.
  \ee
Then
  \be
  \pi^{Y,X}_{g,p,q}(\gamma)
   := \varphi^{-1} \circ \pi^{X\otimes Y^\vee}_{g,p+q}(\gamma) \circ \varphi
  \labl{piXYgpq}
defines a \rep\ $\pi^{Y,X}_{g,p,q}$ of the group \Mapgpq\ on the space
$\HomC(K^{\otimes g}{\otimes}\,Y,X)$.
\end{rem}


\section{Mapping class group invariants}\label{sec:thm}

We now consider the category $\C \eq \HBimod$ of
finite-dimensional bimodules over a factorizable ribbon Hopf algebra $H$.
Recall that there is a canonical Hopf algebra $K$ in \HBimod, obtainable as the
coend \erf{defK} for \HBimod. Its bimodule structure is given in \erf{defHaa};
for a detailed description of the structure morphisms of $K$ as a Hopf algebra,
see \cite[App.\,A.3]{fuSs3}. Also recall the action $\pi_{g:n}^X$ of \Mapgn\
described in Proposition \ref{Lyubact_prop}.

Our goal is to provide, given the Hopf algebra $H$ in \Vectk, and thus the
Hopf algebra $K$ in \HBimod, the following collection of data:

\begin{itemize}

\item
An object $F$ in the category \HBimod\ that carries a natural structure of
a commutative symmetric Frobenius algebra.

\item
For any choice of non-negative integers $g$ and $n$ a morphism
  \be
  \Cor gn \,\in\, \HomHH(K^{\otimes g}, F^{\otimes n})
  \labl{taskCor}
that is invariant under the action $\pi_{g:n}^X$ of \Mapgn\ with
$X \eq F^{\otimes n}$ (which corresponds to taking
$X_1 \eq \cdots \eq X_n \eq F$ as objects in formula \erf{defX}).

\end{itemize}

\begin{rem}
Any such object $F$ is a candidate for a \emph{space of bulk states} in a
conformal quantum field theory whose chiral data are described by the category
\HMod\ of $H$-modules. The morphisms $\Cor gn$ are then candidates for modular
invariant bulk correlation functions of the conformal field theory, for world
sheets of any genus $g$ and for any number $n$ of bulk field insertions. The
classification of spaces of bulk states for given chiral data and
the construction of correlation functions for a given space of bulk states
are two of the most desirable issues in the study of conformal field theory.
 \\
In order to allow for a consistent interpretation as a partition function,
$\Cor 10$, i.e.\ the zero-point correlator on a torus, should be non-zero.
\end{rem}

It is worth stressing that the object $F$ is by no means uniquely determined
by the existence of invariants \erf{taskCor}. Indeed a whole family
$\{\Fomega\}$ of commutative symmetric Frobenius algebras in \HBimod\ that
are candidates for such an object, as well as the corresponding invariant
morphisms for the case that $g\eq 1$ and $n \,{\in}\, \{0,1\}$,
have already been obtained in \cite{fuSs3}. Here $\omega$ is
any choice of a ribbon automorphism of the Hopf algebra $H$. Moreover,
for the case of the identity automorphism it was shown \cite{fuSs4} that
the resulting invariant for $g \eq 1$ and $n \eq 0$ is non-zero.

The main result of the present paper is that each of these $H$-bimodules
$\Fomega$ indeed has all the desired properties: for any choice of ribbon
automorphism $\omega$ we are able to provide the morphism $\Cor gn$ and
establish its \Mapgn-invariance for arbitrary integers $g,n \,{\ge}\, 0$.
On the other hand, generically the family $\{\Fomega\}$
can \emph{not} be expected to
exhaust all solutions to the problem posed above; our results
do not suggest any concrete approach to such a classification.

\medskip

As it turns out, the presence of a general ribbon automorphism $\omega$
only constitutes a minor modification of the issues that already arise
in the case of the identity automorphism. Accordingly the bimodule of central
importance for our discussion is $F^{\!\idsm_H}$, the one obtained when
$\omega \eq \id_H$. Henceforth we slightly abuse notation and simply use
the symbol $F$ for this object. For the family $\{\Fomega\}$ obtained in
\cite{fuSs3}, setting $\omega \eq \id_H$ yields the \emph{coregular bimodule}
in \HBimod. By this we mean the dual of the regular bimodule $(H,m,m)$, i.e.\
the vector space \Hs\ endowed with the dual of the regular left and right
actions of $H$ on itself. Explicitly,
   \be
   \Hb = (\Hs,\brho,\bohr)
   \ee
with $\brho \iN \Hom(H\oti\Hs,\Hs)$ and $\bohr \iN \Hom(\Hs\oti H,\Hs)$ given by
   \be
   \bearl
   \brho:= (d_H\oti\idHs) \cir (\idHs\oti m\oti\idHs) \cir (\idHs\oti\apo\oti b_H)
   \cir\tauHHv \qquand
   \nxl3
   \bohr:= (d_H\oti\idHs)\cir(\idHs\oti m\oti\idHs)
   \cir(\idHs\oti\id_H\oti\tauHvH)\cir(\idHs\oti b_H \oti\apoi) \,.
   \eear
   \labl{rhorho}
It has been demonstrated \cite{fuSs3} that the coregular bimodule $F$ carries a
natural structure of a commutative symmetric Frobenius algebra in the ribbon
category \HBimod. Moreover, $F$ can be characterized as the coend of a suitable
functor from $H\Mod\op \Times H\Mod$ to \HBimod. This way the structural
morphisms endowing $F$ with the structure of an algebra in \HBimod\ can be
obtained with the help of the universal property of this coend. In contrast,
the coalgebra structure of $F$ involves the integral and cointegral of $H$.
Specifically, the product $m_F$ is the dual of the coproduct of $H$ and the
unit $\eta_F$ is the dual of the counit of $H$, in the coproduct $\Delta_F$
the cointegral $\lambda\iN H^*$ enters, and the counit $\eps_F$ is the dual
of the integral $\Lambda\iN H$. Explicitly, in graphical notation we have
   \Eqpic{pic-Hb-Frobalgebra} {440} {47} { \put(0,-3){
   \put(0,45)      {$ m\bico~= $}
     \put(48,0)  {\Includepichtft{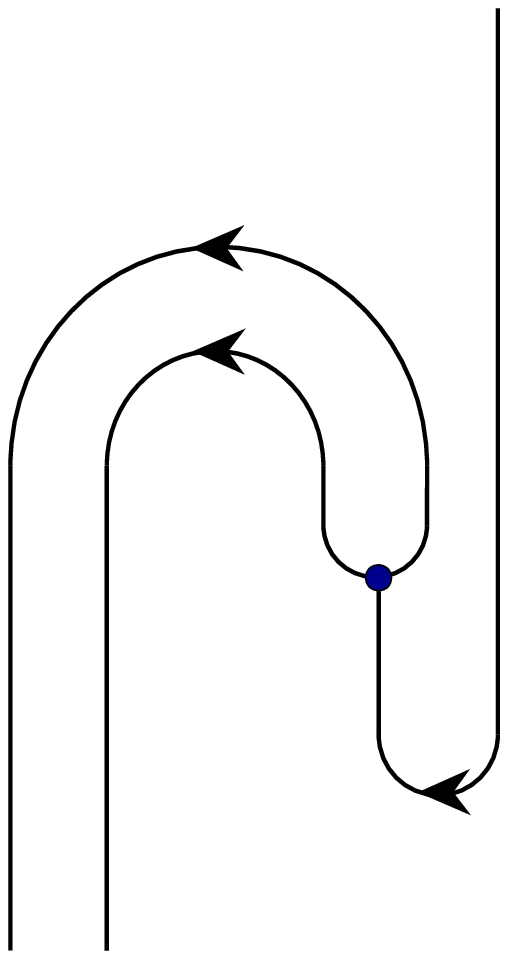}
   \put(-5.9,-8.8) {\sse$ \Hss $}
   \put(6.5,-8.8)  {\sse$ \Hss $}
   \put(31.5,34.7) {\sse$ \Delta $}
   \put(49.7,106.8){\sse$ \Hss $}
   }
   \put(142,45)    {$ \eta\bico~= $}
     \put(185,24) {\Includepichtft{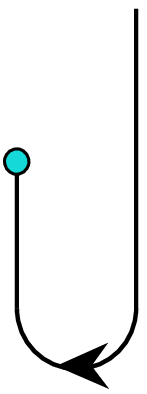}
   \put(-5.9,23.6) {\sse$ \eps $}
   \put(10.7,44.1) {\sse$ \Hss $}
   }
   \put(243,45)    {$ \Delta\bico~= $}
     \put(291,0) {\Includepichtft{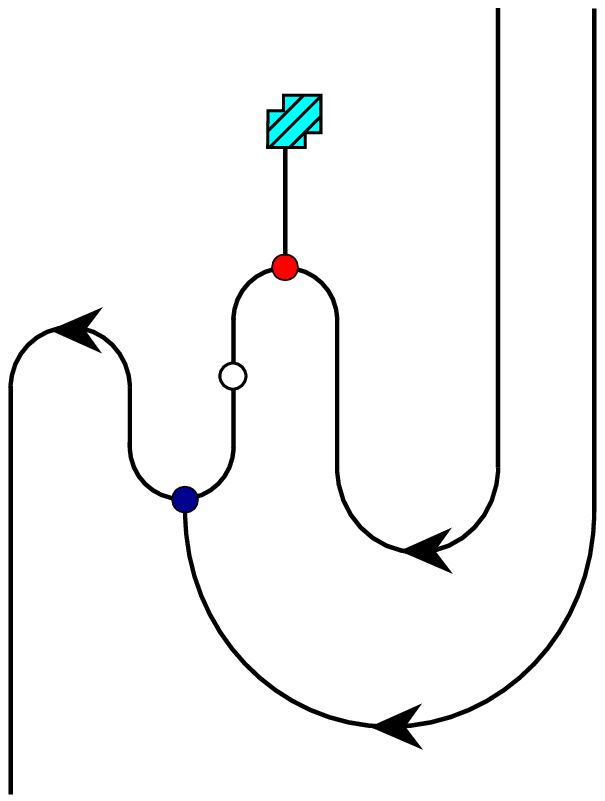}
   \put(-4.3,-8.8) {\sse$ \Hss $}
   \put(11.1,26.8) {\sse$ \Delta $}
   \put(17.7,43.8) {\sse$ \apo $}
   \put(21.5,71)   {\sse$ \lambda $}
   \put(32.8,59.3) {\sse$ m $}
   \put(48.2,89.4) {\sse$ \Hss $}
   \put(61.1,89.4) {\sse$ \Hss $}
   }
   \put(395,45)    {$ \eps\bico~= $}
     \put(438,24) {\Includepichtft{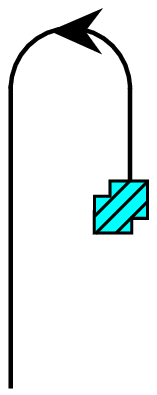}
   \put(-4.3,-8.8) {\sse$ \Hss $}
   \put(15.8,16.3) {\sse$ \Lambda $}
   } } }

We are now in a position to give our result for the morphisms $\Cor gn$ in the
category $\C \eq \HBimod$. We first present $\Cor gn$ in purely categorical
terms: as an element of the morphism space
$\Hom_\C(K^{\otimes n},F^{\otimes n})$ of the category \C\ it is given by
  \eqpic{Sk_morph} {180} {148} { \put(0,13){
  \put(0,129)      {$\Cor gn :=~$}
    \put(60,0) {\Includepichtft{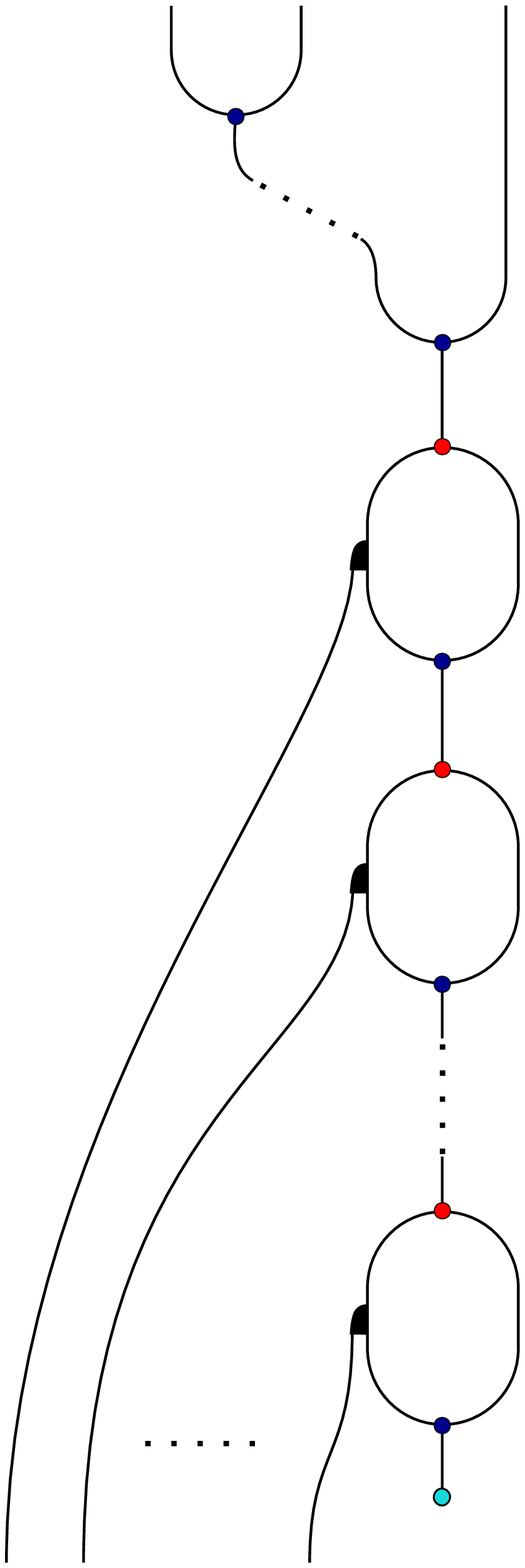}
  \put(-8.7,-12)   {\sse$ \underbrace{\hspace*{7.1em}}_{g~ \rm factors} $}
  \put(28.3,288)   {\sse$ \overbrace{\hspace*{7.6em}}^{n~ \rm factors} $}
  \put(-5,-9.2)    {\sse$ K $}
  \put(8.8,-9.2)   {\sse$ K $}
  \put(27,-9.2)    {\sse$ \cdots $}
  \put(49.2,-9.2)  {\sse$ K $}
  \put(27.2,278)   {\sse$ F $}
  \put(49.5,278)   {\sse$ F $}
  \put(66,278)     {\sse$ \cdots $}
  \put(86,278)     {\sse$ F $}
  \put(65,176)     {\sse$ \rho^K_F $}
  \put(78.6,198)   {\sse$ m_F^{} $}
  \put(71.3,219)   {\sse$ \Delta_F $}
  \put(79.8,12.2)  {\sse$ \eta_F^{} $}
  }
  \put(186,256)    {\catpic}
  } }
for $n \,{>}\, 0$, and as $\Cor g0 \,{:=}\, \eps_F \cir \Cor g1$.

\medskip

Let us describe the rationale for arriving at this expression for $\Cor gn$.

\def\leftmargini{1.57em}~\\[-2.65em]\begin{enumerate}\addtolength{\itemsep}{-3pt}
    \item
Draw a skeleton for the surface $\Surf g{n}$, including outward-oriented edges
attached to the boundary components, and label each edge of this skeleton
with the Frobenius algebra $F$ in \C. (Instead of with edges, a priori one may 
want to work with ribbons, but owing to $\theta_F \eq \id_F$ this is insignificant.)
    \item
Orient the internal edges in such a way that each of the vertices of the
skeleton has either two incoming and one outgoing edge or vice versa.
Then label each vertex with the product $m_F$ or coproduct $\Delta_F$ of $F$,
depending on whether two or one of its edges are incoming.
    \item
Further, for each handle of the surface, attach another edge
labeled by the handle Hopf algebra $K \iN \C$ to the corresponding
loop of the skeleton, and label the resulting new trivalent vertex
by the action $\rho^K_F$, as defined in \erf{def-rho}, of the Hopf algebra $K$
on the object in \C\ underlying the Frobenius algebra $F$.
    \item
The resulting graph is interpreted as a morphism in \C. Using the fact that
the algebra $F$ is commutative symmetric Frobenius, the vertices can be
rearranged in such a way that one ends up with the morphism given in
\erf{Sk_morph}.
\end{enumerate}

\noindent
Our main result can now be stated as follows:

\begin{thm}\label{thm:main}
Let $\C \eq \HBimod$ for $H$ a finite-dimensional factorizable ribbon Hopf
algebra. Then for any pair of integers $g,n \,{\ge}\,0$ the morphism $\Cor gn$
is invariant under the action $\pi_{g:n}^{F^{\otimes n}}$ of the mapping class
group $\Mapgn$ described in Proposition \ref{Lyubact_prop}.
\end{thm}

\medskip

\begin{rem}\label{rem:pq2}
As already pointed out, a major motivation for our investigations are
applications to conformal field theory. In that context, the morphisms $\Cor gn$
describe correlation functions of bulk fields.  As such they not only
have to be invariant under the relevant action of the mapping class group,
but must also satisfy so-called factorization constraints
(for a precise formulation in the semisimple case see \cite{fjfrs}). The latter
constraints relate correlators for surfaces of different topology and
necessarily involve  correlators with (using quantum field theory terminology)
both incoming and outgoing field insertions. We do not have anything to say
about factorization constraints in this paper, but including the possibility
of having both incoming and outgoing insertions is
easy. Indeed, for any choice of non-negative integers $g$, $p$ and $q$,
according to Remark \ref{rem:pq} we have a \rep\ of \Mapgpq\ on the space
$\HomHH(K^{\otimes g}\oti F^{\otimes q}, F^{\otimes p})$. A morphism
  \be
  \Corr gpq\in \HomHH(K^{\otimes g}\oti F^{\otimes q}, F^{\otimes p})
  \ee
that is invariant under this action is then obtained as follows. We have
$\Corr gp0 \eq \Cor gp$ as given in \erf{Sk_morph}, while for $q \,{>0}$ the
morphism $\Corr gpq$ is obtained from $\Cor gp$ by just replacing the unit
$\eta_F \iN \HomHH(\one,F) \,{\equiv}\, \HomHH(F^{\otimes 0},F)$
by the morphism in $\HomHH(F^{\otimes q},F)$ that is given by 
a $q{-}1$-fold product of $F$.
\end{rem}

\begin{rem}
As already pointed out, we actually find a family of suitable objects $\Fomega$,
labeled by ribbon automorphisms $\omega$ of $H$, and provide invariant vectors 
$\Corw g {p,q}$ in the corresponding morphism spaces $\HomHH\big(K^{\otimes g}
\oti (\Fomega)^{\otimes q}_{}, (\Fomega)^{\otimes p}_{}\big)$ for all values 
of $g$, $p$ and $q$.  Thus for any ribbon automorphism $\omega$ we obtain 
candidates for the bulk state space and for bulk correlation functions in 
conformal field theory. For brevity we have
concentrated above to the case of the identity automorphism.
Sections \ref{sec:lemmata} and \ref{sec:proofmain} below contain
the proof of our main result for $\omega \eq \id_H$, while the proof for
the general case will be accomplished in Section \ref{sec:omega}.
\end{rem}

The picture \erf{Sk_morph} represents $\Cor gn$ as a morphism of
the braided tensor category of $H$-bi\-modules. We conclude this section
by expressing $\Corr gpq$ as a \ko-linear map, thereby obtaining a pictorial
description in the category of vector spaces. To this end we insert
explicit expressions for all the structural morphisms appearing in
\erf{Sk_morph}. The expressions for the structural morphisms of $F$
have already been displayed in the picture \erf{pic-Hb-Frobalgebra}.
It therefore suffices to describe in addition the morphisms $\Corr g11$
with one incoming and one outgoing insertion of $F$.

Let us first consider the case $g \eq p \eq 1$ and $q \eq 0$:

\begin{lem}
The morphism $\Corr 110 \,{\equiv}\, \Cor 11$ satisfies the following chain
of equalities of linear maps:
  \Eqpic{PF_1} {440} {57} {\setulen95
  \put(-3,59)      { $ \Cor 11 ~= $ }
  \put(63,0)  {\INcludepichtft{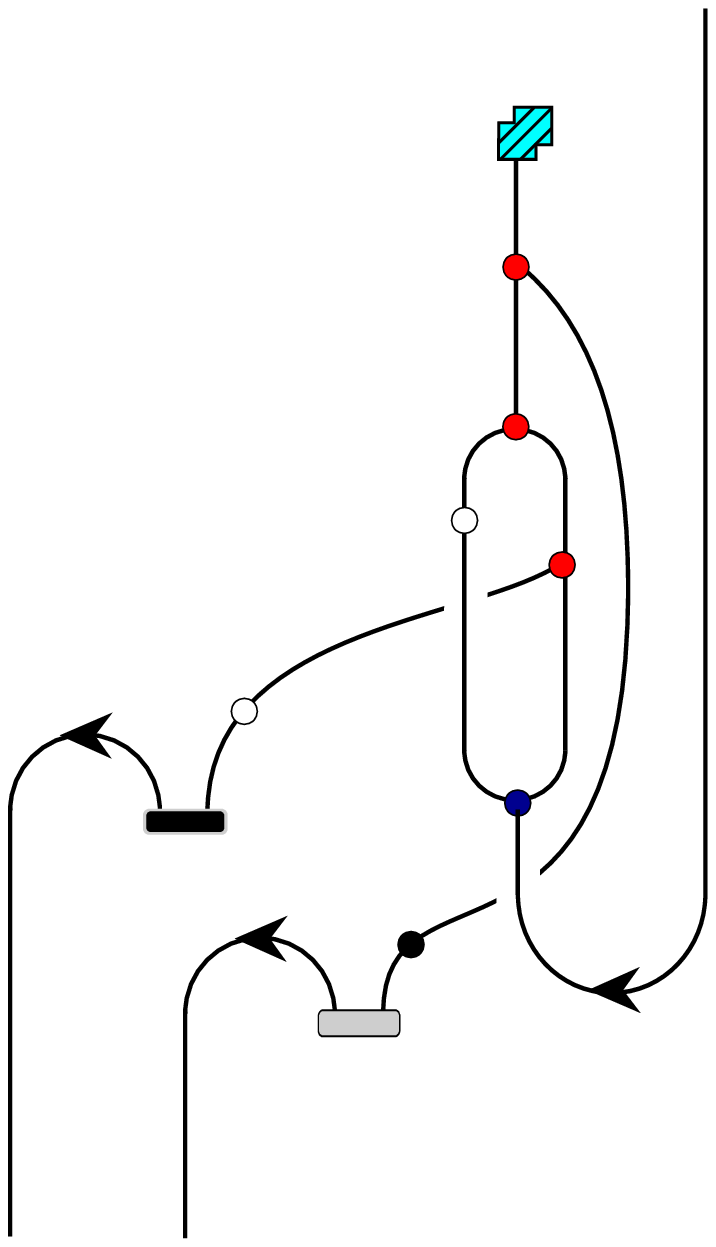}{361}
  \put(-5.5,-9.2)  {\sse$ \Hss $}
  \put(14.4,-9.2)  {\sse$ \Hss $}
  \put(25.2,41)    {\sse$ Q^{-1} $}
  \put(33.5,15)    {\sse$ Q $}
  \put(61,118.7)   {\sse$ \lambda $}
  \put(73,139)     {\sse$ \Hss $}
  }
  \put(163,59)     {$=$}
  \put(199,0) {\INcludepichtft{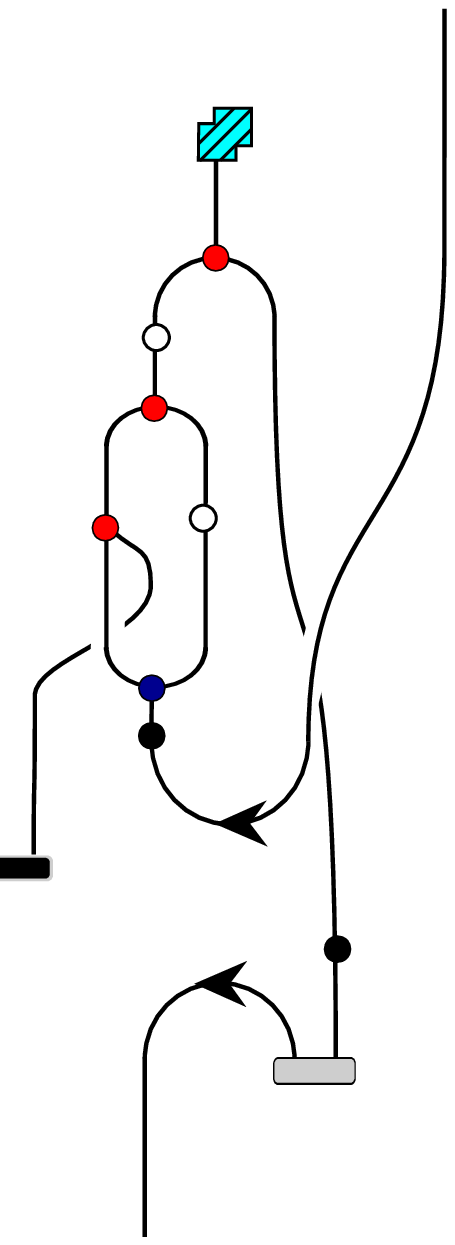}{361}
  \put(-4.5,-9.2)  {\sse$ \Hss $}
  \put(28,-9.2)    {\sse$ \Hss $}
  \put(14.5,33)    {\sse$ Q^{-1} $}
  \put(47.5,10.2)  {\sse$ Q $}
  \put(45.1,118.5) {\sse$ \lambda $}
  \put(62,139)     {\sse$ \Hss $}
  }
  \put(271,59)     {$=$}
  \put(307,0) {\INcludepichtft{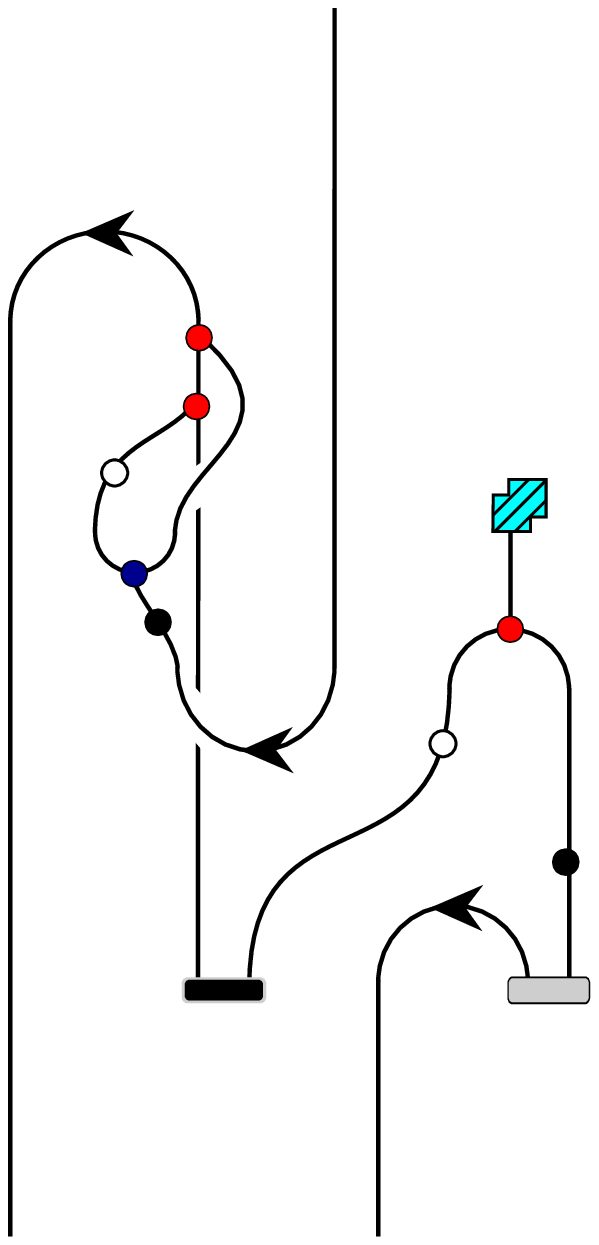}{361}
  \put(-5,-9.2)    {\sse$ \Hss $}
  \put(36,-9.2)    {\sse$ \Hss $}
  \put(20,19)      {\sse$ Q^{-1} $}
  \put(55.6,19)    {\sse$ Q $}
  \put(59.7,78.5)  {\sse$ \lambda $}
  \put(32,139)     {\sse$ \Hss $}
  }
  \put(390,59)    {$=$}
  \put(422,0) {\INcludepichtft{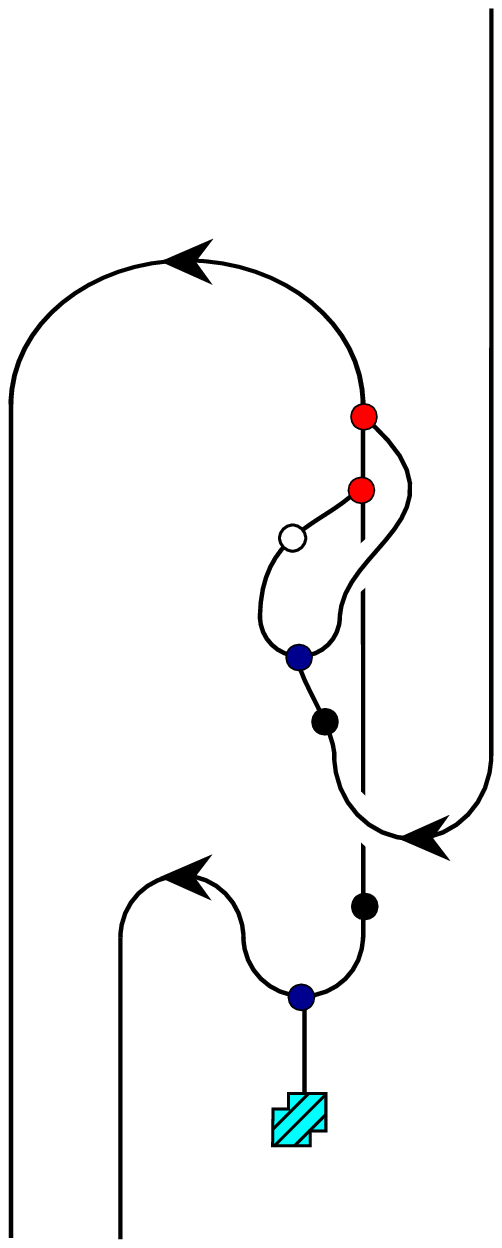}{361}
  \put(-5,-9.2)    {\sse$ \Hss $}
  \put(8.2,-9.2)   {\sse$ \Hss $}
  \put(36,10.5)    {\sse$ \Lambda $}
  \put(49.5,139)   {\sse$ \Hss $}
  } }
Here $Q$, $\lambda$ and $\Lambda$ are the monodromy matrix, the cointegral and
the integral of $H$, respectively.
\end{lem}

\begin{proof}
Insert the expressions for $m_F$, $\eta_F$ and $\Delta_F$ as well as for the
\HK-action $\rho^K_F$ (with braidings according to \erf{bibraid} and with the
formula for $\rho^K_X$ that we will present for a general $H$-bimodule $X$ in
\erf{Lyubact_HKH} below) into \erf{Sk_morph} with $g\eq n\eq 1$. Then by using
associativity of the product $m$ of $H$, we arrive at the first picture in the
chain \erf{PF_1} of equalities. The second equality follows by using several
times the anti-(co)algebra property of the antipode $\apo$ of $H$.
In the resulting morphism we can recognize the left-adjoint $H$-action
on the right leg of the inverse monodromy matrix $Q^{-1}$. The third
equality is then just the statement that the morphism $f_{Q^{-1}}$
intertwines the left-adjoint and left-coadjoint actions. The last
equality follows with the help of the identity \erf{fQS_Psi} together with
$(\apo\oti\apo)\cir Q \eq \tauHH\cir Q$, the anti-coalgebra property
of the inverse antipode and the fact that, by unimodularity of $H$,
 $\apo\cir\Lambda \eq \Lambda$.
\end{proof}

To proceed to $\Corr 111$ we simply note that, by just using the unit property
of $\eta_F$ and the Frobenius property and associativity of $F$, we have
  \be
  \Corr 111 = m_F \circ \big[\,\big(\Corr 111 \cir (\id_K\oti\eta_F) \big)
  \oti \id_F \,\big] =  m_F \circ ( \Corr 110 \oti \id_F) \,.
  \labl{111from110}
With the result \erf{PF_1} of Lemma 2.6 and the explicit form of the
multiplication $m_F$ from formula \erf{pic-Hb-Frobalgebra}, this amounts to
  \eqpic{Corr12H} {300} {48} {\setulen80
  \put(5,60)       {$\Corr 111 ~=$}
    \put(101,0) {\includepichtft{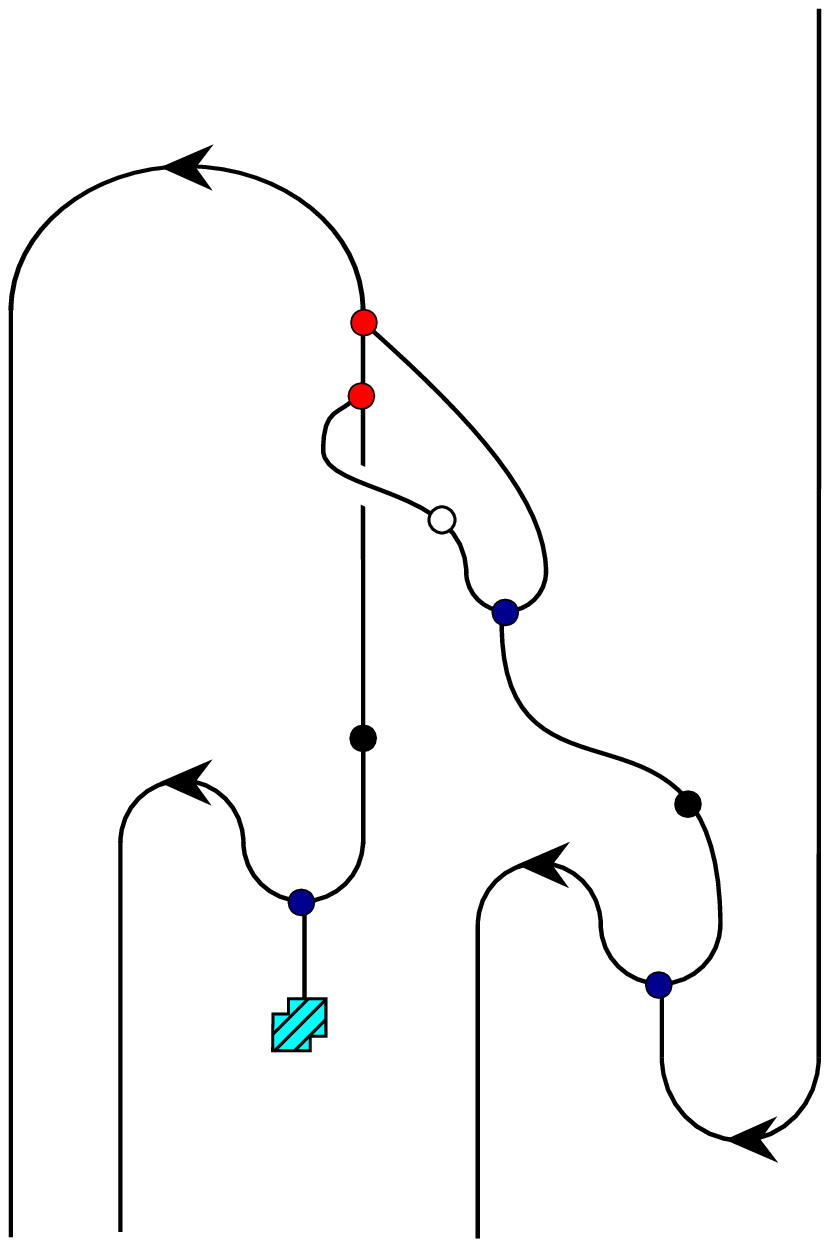}
  \put(-5,-9.2)    {\sse$ \Hss $}
  \put(8.2,-9.2)   {\sse$ \Hss $}
  \put(47.2,-9.2)  {\sse$ \Hss $}
  \put(36,19.5)    {\sse$ \Lambda $}
  \put(86,139)     {\sse$ \Hss $}
  }
  \put(219,60)     {$ \equiv $}
    \put(261,0) {\includepichtft{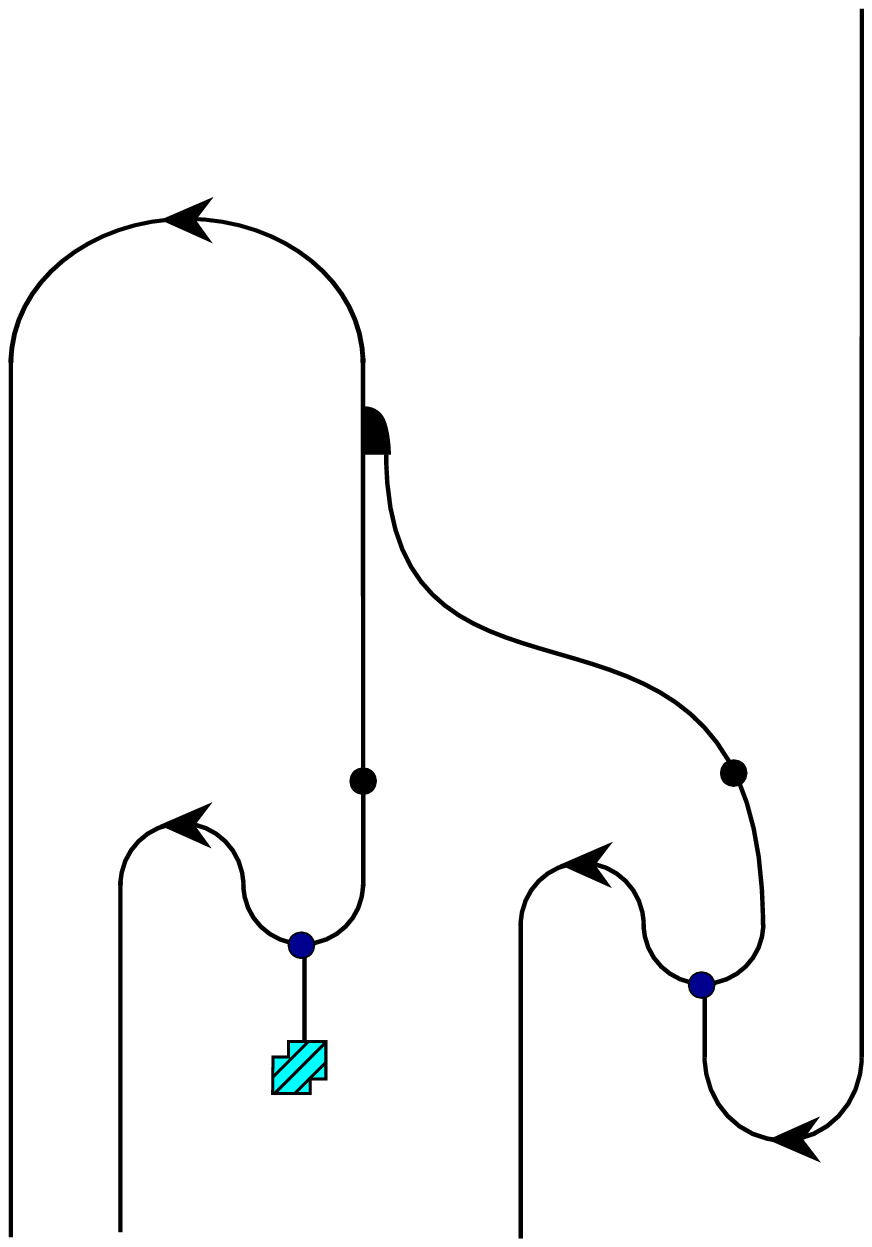}
  \put(-5,-9.2)    {\sse$ \Hss $}
  \put(8.2,-9.2)   {\sse$ \Hss $}
  \put(51.2,-9.2)  {\sse$ \Hss $}
  \put(36,15.5)    {\sse$ \Lambda $}
  \put(43,87.5)    {\sse$ \ohrad $}
  \put(89,139)     {\sse$ \Hss $}
  } }
with $\ohrad$ the right-adjoint action of $H$ on itself.
This expression extends in a straightforward manner to
  \eqpic{CorrgnH} {300} {117} { \put(0,18){ \setulen90
  \put(0,120)      {$\Corr g11=~$}
    \put(81,0) {\INcludepichtft{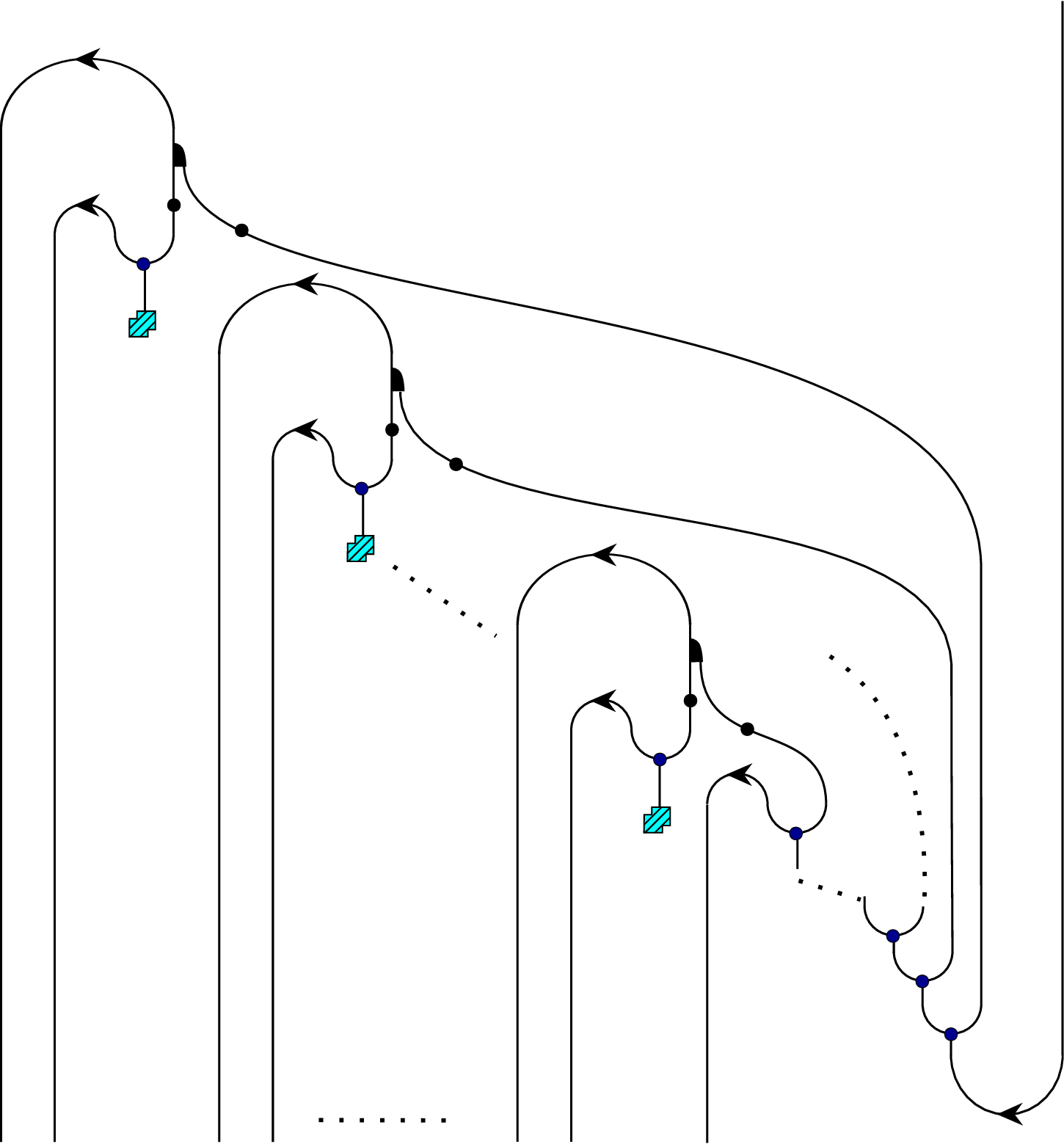}{342}
  \put(-3.9,-13)   {\sse$ \underbrace{\hspace*{14.5em}}
                    _{g~ {\rm factors~of}~\Hs{\otimes}\Hs} $}
  \put(-5,-9.2)    {\sse$ \Hss $}
  \put(8.2,-9.2)   {\sse$ \Hss $}
  \put(44.5,-9.2)  {\sse$ \Hss $}
  \put(56.5,-9.2)  {\sse$ \Hss $}
  \put(76.5,-9.2)  {\sse$ \dots\dots$}
  \put(111.5,-9.2) {\sse$ \Hss $}
  \put(123.5,-9.2) {\sse$ \Hss $}
  \put(154.5,-9.2) {\sse$ \Hss $}
  \put(234,260)    {\sse$ \Hss $}
  \put(36,180.5)   {\sse$ \Lambda$}
  \put(70,130.5)   {\sse$ \Lambda$}
  \put(137,69.5)   {\sse$ \Lambda$}
  \put(43,220)     {\sse$ \ohrad $}
  \put(91.4,169.4) {\sse$ \ohrad $}
  \put(158.2,108.5){\sse$ \ohrad $}
  } } }

Finally, the morphisms with $q$ or $p$ larger than 1 are obtained by pre- and 
post-composition of \erf{CorrgnH} with an appropriate multiple product and
multiple coproduct of $F$, respectively.
Explicitly, writing $m_F^{(2)}\eq m_F$ and $m_F^{(l)}\eq m_F
\cir (m_F^{(l-1)} \oti \id_F)$ for $l\,{>}\,2$ as well as $m_F^{(0)} \eq \eta_F$
and $m_F^{(1)} \eq \id_F$, and analogously for $\Delta_F^{(l)}$, we have
  \be
  \Corr gpq = \Delta_F^{(p)} \circ \Corr g11 \circ
  \big( \id^{}_{K^{\otimes g}_{}} \oti m_F^{(q)} \big)
  \labl{gpqfromg11}
for all $p,q$ and $g$.


\section{Useful identities}\label{sec:lemmata}

In this section we present a few preliminary results that will be instrumental
in the proof of Theorem \ref{thm:main}. Our first task is to specialize the
half-monodromies and related morphisms that we introduced, in Section
\ref{subsec:hHa}, for a general \fftc\ \C\ to the specific case of
$\C \eq \HBimod$. To do so we use the specific form \erf{bibraid} of the
braiding in \HBimod\ and the explicit expressions \cite{lyub6,vire4} for the
dinatural family and for the Hopf algebra
structure morphisms of $K$. For the morphism $\QB_X$ defined in formula
\erf{QHX} we then obtain, for any bimodule $X \eq (X,\rho_X,\ohr_X)$,
  \eqpic{Qq_X} {150} {49} {\setulen80
  \put(9,57)       {$ \QB_X =~$}
    \put(80,0) {\includepichtft{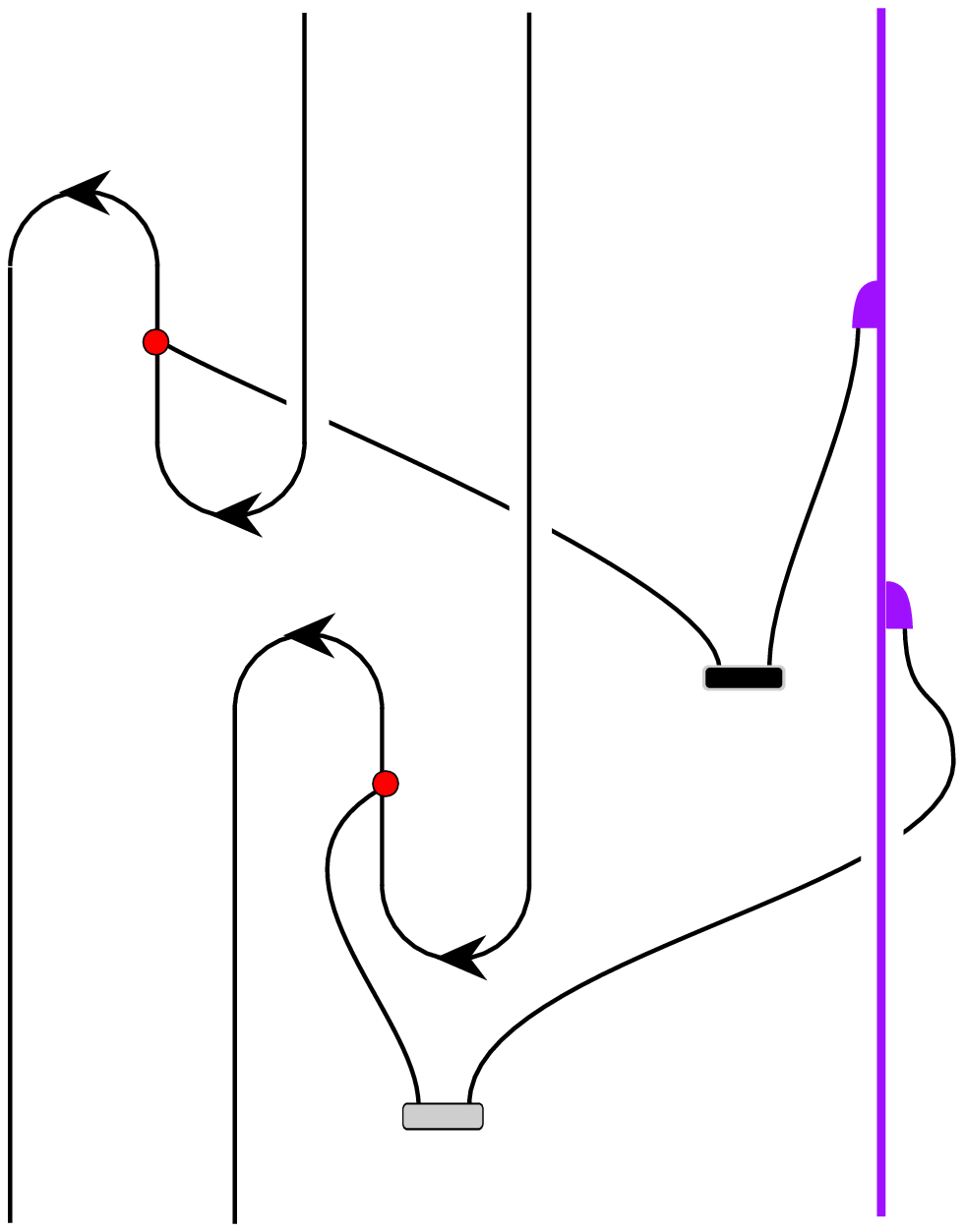}
  \put(-5.4,-9.2)  {\sse$ \Hss $}
  \put(19.3,-9.2)  {\sse$ \Hss $}
  \put(27.4,140.4) {\sse$ \Hss $}
  \put(52.8,140.4) {\sse$ \Hss $}
  \put(92.5,-9.2)  {\sse$ X $}
  \put(92.9,140.4) {\sse$ X $}
  \put(53,4.5)     {\sse$ Q $}
  \put(72,50.7)    {\sse$ Q^{-1} $}
  \put(103.1,68.7) {\sse$ \ohr_X^{} $}
  \put(80.9,101.8) {\sse$ \rho_X^{} $}
  } }
where $Q$ is the monodromy matrix of $H$,
while the endomorphism \erf{QHH} of $K\oti K$ becomes
  \eqpic{QQ_Haa} {240}{63} {
  \put(22,57)      {$ \QQ ~= $}
  \put(81,0)  { \Includepichtft{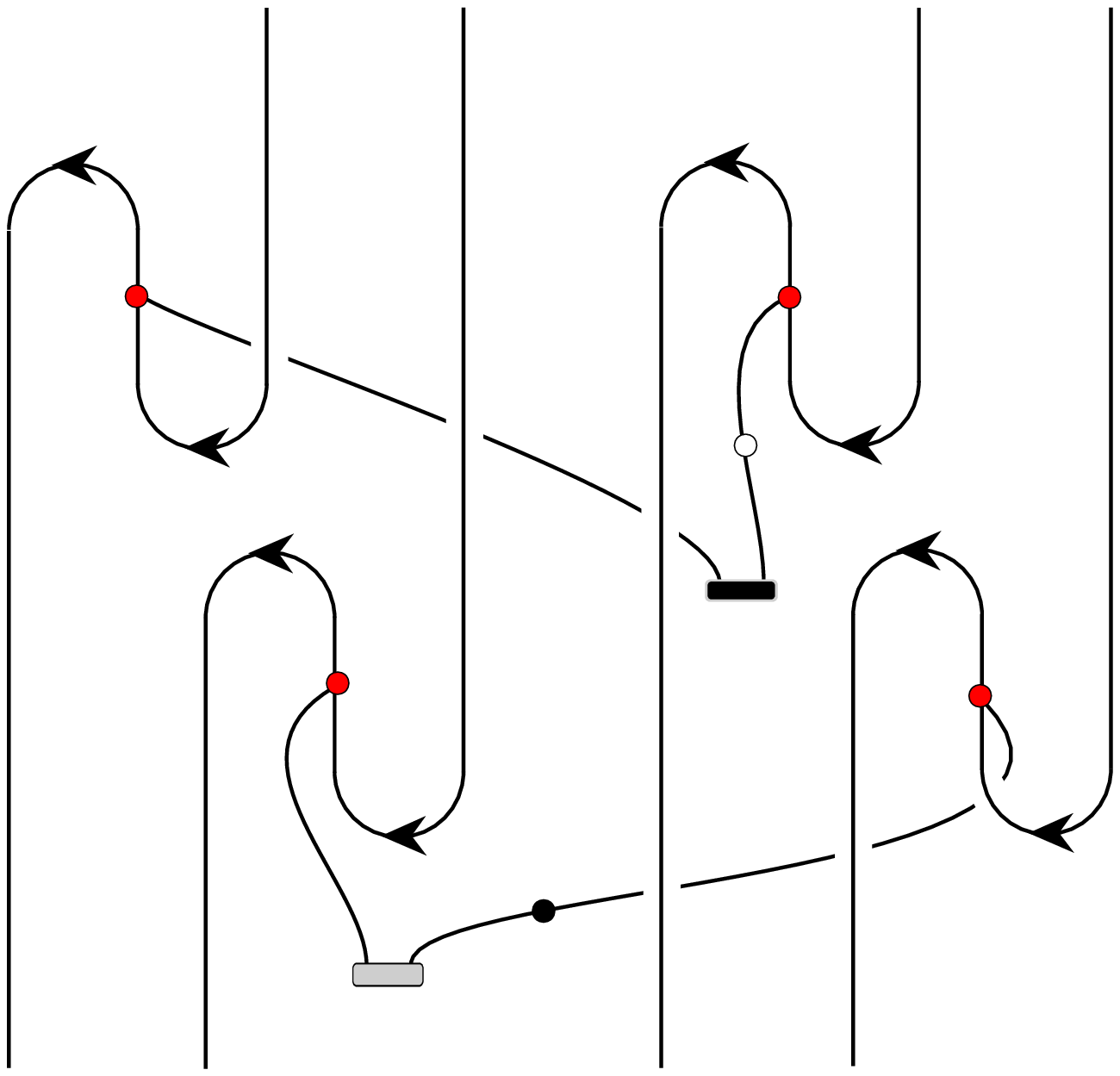}
  \put(-4.3,-9.2)  {\sse$ \Hss $}
  \put(20.7,-9.2)  {\sse$ \Hss $}
  \put(29.4,139)   {\sse$ \Hss $}
  \put(44.8,3.9)   {\sse$ Q $}
  \put(54.8,139)   {\sse$ \Hss $}
  \put(62.8,12.8)  {\sse$ \apoi $}
  \put(78.7,-9.2)  {\sse$ \Hss $}
  \put(86.6,53.1)  {\sse$ Q^{-1} $}
  \put(87.6,77.9)  {\sse$ \apo $}
  \put(102.9,-9.2) {\sse$ \Hss $}
  \put(112.2,139)  {\sse$ \Hss $}
  \put(137.7,139)  {\sse$ \Hss $}
  } }
The S- and T-transformations \erf{S-HK} and \erf{p9} involve the monodromy
matrix and cointegral, and the ribbon element $v$, respectively:
  \eqpic{S_KH-TKH} {410}{45} { \put(0,4){
  \put(0,39)    {$ \SK ~= $}
  \put(50,-6) { \Includepichtft{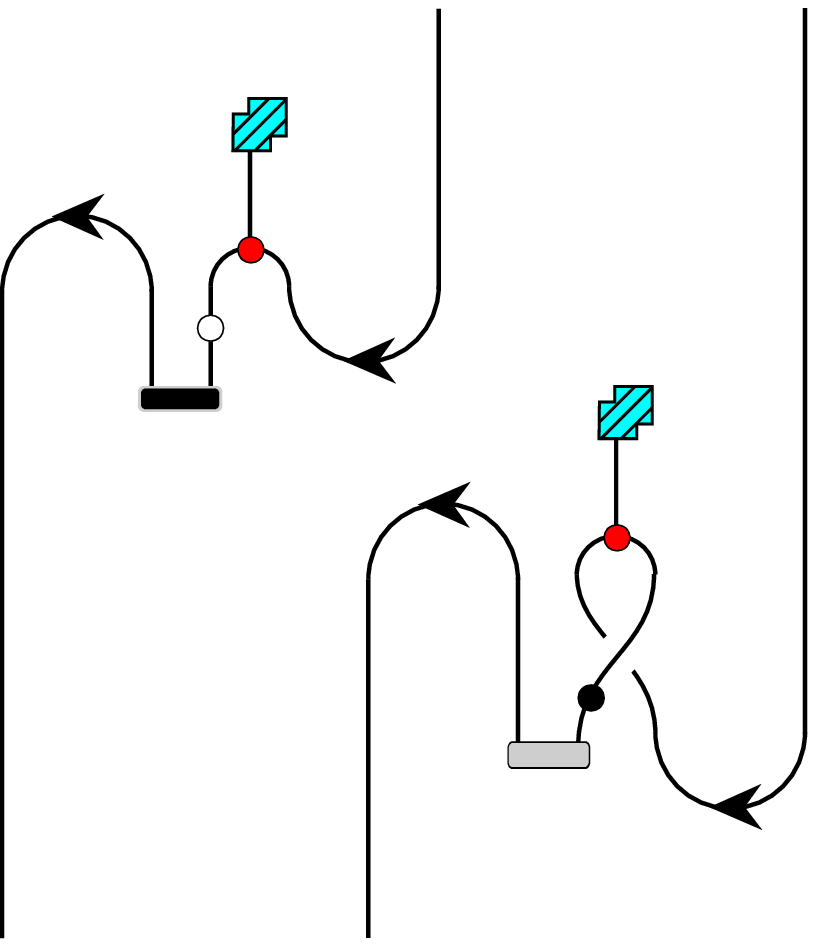}
  \put(-4.5,-9.2) {\sse$ \Hss $}
  \put(13,51.8)   {\sse$ Q^{-1} $}
  \put(36.5,-9.2) {\sse$ \Hss $}
  \put(32.5,87)   {\sse$ \lambda $}
  \put(44.6,106)  {\sse$ \Hss $}
  \put(57.4,12.1) {\sse$ Q $}
  \put(73.5,56)   {\sse$ \lambda $}
  \put(85.6,106)  {\sse$ \Hss $}
  }
  \put(180,39)    {and}
  \put(230,39)    {$ \TK ~= $}
  \put(278,0)  { \Includepichtft{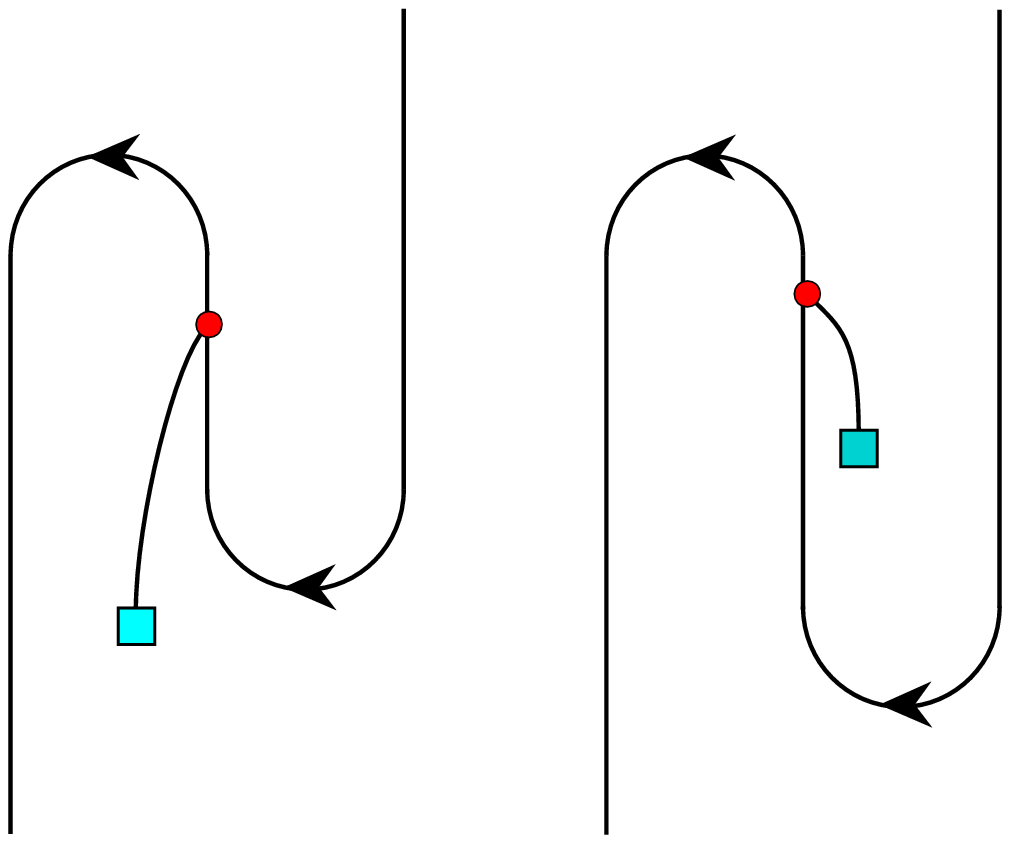}
  \put(-4.4,-9.2) {\sse$ \Hss $}
  \put(12,15)     {\sse$ v $}
  \put(39.6,94)   {\sse$ \Hss $}
  \put(61.2,-9.2) {\sse$ \Hss $}
  \put(89.7,34.5) {\sse$ v^{-1} $}
  \put(105.3,94)  {\sse$ \Hss $}
  } } }
The partial monodromy action \erf{def-rho} of $K$ on an $H$-bimodule $X$,
which is obtained by composing $\QB_X$ in \erf{Qq_X} with
$\eps_K \oti \id_X \eq \eta^\vee \oti \eta^\vee \oti\id_X$ (and which, as noted
in Remark \ref{rem:YD}, fits together with the natural $K$-coaction to a
Yetter-Drinfeld structure) is
  \eqpic{Lyubact_HKH} {120} {48} {
  \put(0,47)       {$ \rho^K_X=~ $}
    \put(50,0) { \Includepichtft{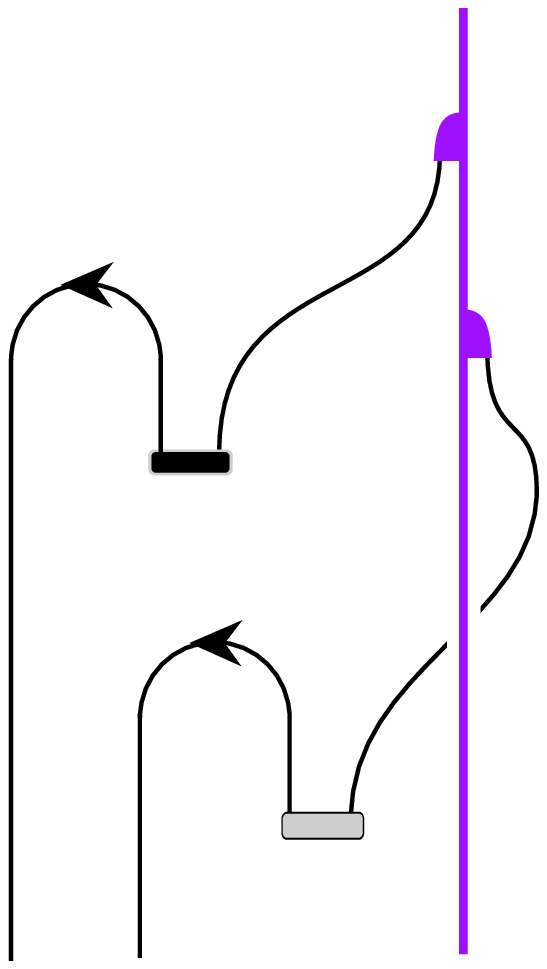}
  \put(-5.4,-8.5)  {\sse$ \Hss $}
  \put(10,-8.5)    {\sse$ \Hss $}
  \put(46.2,-8.5)  {\sse$ X $}
  \put(31,6.8)     {\sse$ Q $}
  \put(13,46.8)    {\sse$ Q^{-1} $}
  \put(54.7,68.8)  {\sse$ \ohr_X^{} $}
  \put(35.2,90.5)  {\sse$ \rho_X^{} $}
  \put(46.6,108.6) {\sse$ X $}
  } }
i.e.\ the natural $K$-action is nothing but the $H$-bimodule action composed
with variants of the Drinfeld map.
We will also need the inverse of the isomorphism $\SK$; it is given by
  \eqpic{S_KHinv} {155}{45} { \put(0,4){
  \put(0,42)       {$ \SK^{-1} ~= $}
    \put(50,-6) { \Includepichtft{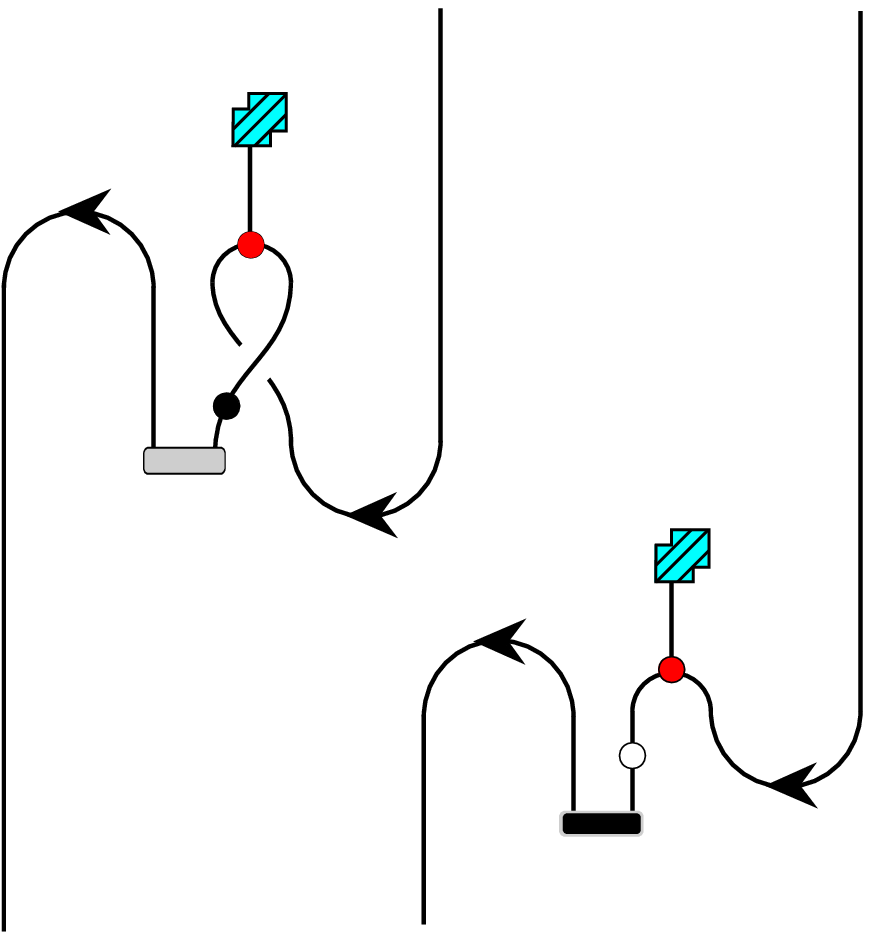}
  \put(-4.5,-8.1)  {\sse$ \Hss $}
  \put(16.4,43.8)  {\sse$ Q $}
  \put(41.5,-8.1)  {\sse$ \Hss $}
  \put(32.4,87)    {\sse$ \lambda $}
  \put(44.1,106)   {\sse$ \Hss $}
  \put(59.3,4.1)   {\sse$ Q^{-1} $}
  \put(78.7,39)    {\sse$ \lambda $}
  \put(90.6,106)   {\sse$ \Hss $}
  } } }

Next, for further use we note that, since the R-matrix intertwines the coproduct
and opposite coproduct of $H$, conjugating by the monodromy matrix preserves
the coproduct:
  \eqpic{removeQs} {205} {31} {
    \put(0,0) {\Includepichtft{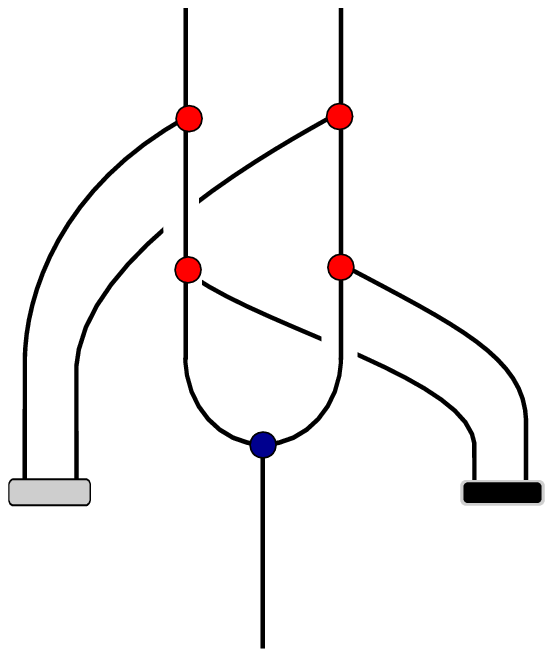}
  \put(23,-9.2)    {\sse$ H $}
  \put(1,8.2)      {\sse$ Q $}
  \put(47,8.6)     {\sse$ Q^{-1}$}
  \put(16.5,74)    {\sse$ H $}
  \put(33,74)      {\sse$ H $}
  }
  \put(85,33)      {$=~\Delta_H~=$}
    \put(159,0) {\Includepichtft{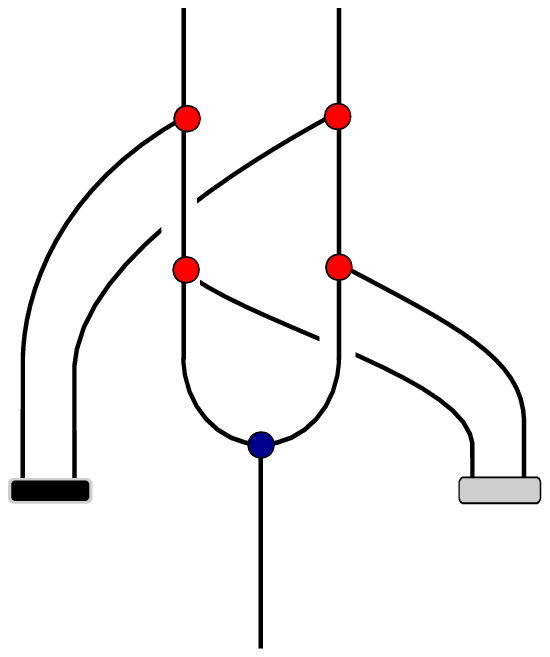}
  \put(23,-9.2)    {\sse$ H $}
  \put(-1.9,8.6)   {\sse$ Q^{-1}$}
  \put(50.7,8.6)   {\sse$ Q$}
  \put(16.5,74)    {\sse$ H $}
  \put(33,74)      {\sse$ H $}
  } }
We will also make use of the following result that can be formulated in the
general setting of ribbon categories:

\begin{lemma}\label{lem:2prod_mon}
For any commutative Frobenius algebra $A$ in a ribbon category \C\ the equalities
  \eqpic{2prod_mon} {250} {38} {\setulen90
    \put(0,0)  { \INcludepichtft{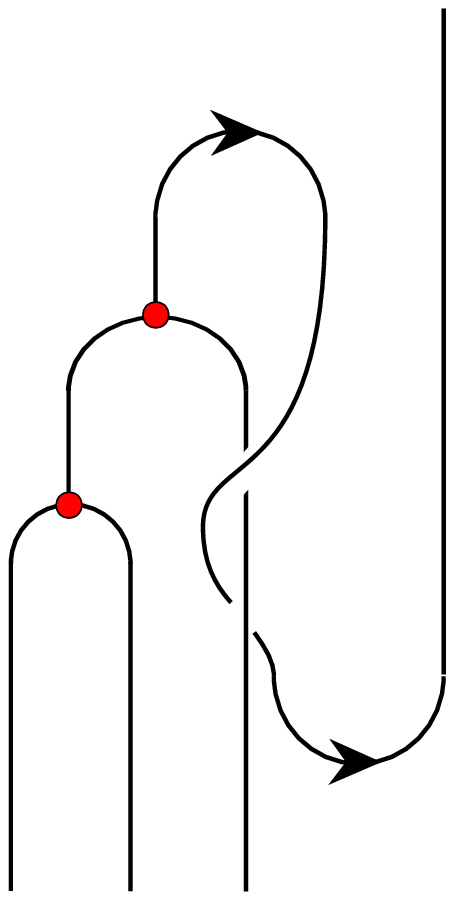}{342}
  \put(-4.3,-9.2)  {\sse$ A $}
  \put(8.3,-9.2)   {\sse$ A $}
  \put(20,47.6)    {\sse$ c $}
  \put(21.5,-9.2)  {\sse$ A $}
  \put(28,30.7)    {\sse$ c $}
  \put(44.5,101)   {\sse$ A $}
  }
  \put(82,43)      {$=$}
    \put(120,0) { \INcludepichtft{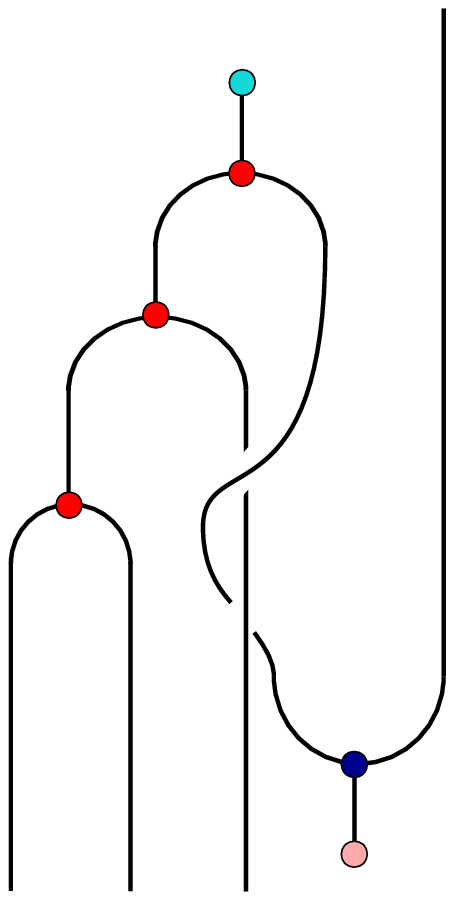}{342}
  \put(-4.3,-9.2)  {\sse$ A $}
  \put(8.3,-9.2)   {\sse$ A $}
  \put(20,47.6)    {\sse$ c $}
  \put(21.5,-9.2)  {\sse$ A $}
  \put(28,30.7)    {\sse$ c $}
  \put(44.2,101)   {\sse$ A $}
  }
  \put(203,43)     {$=$}
    \put(240,0) { \INcludepichtft{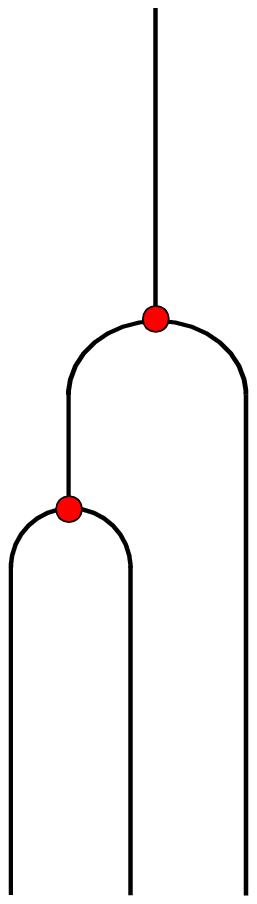}{342}
  \put(-4.3,-9.2)  {\sse$ A $}
  \put(8.3,-9.2)   {\sse$ A $}
  \put(12.8,101)   {\sse$ A $}
  \put(21.8,-9.2)  {\sse$ A $}
  }
  \put(297,80)     {\catpic}
  }
hold.
\end{lemma}

\begin{proof}
The first equality is a direct consequence of the Frobenius property of $A$.
The second equality follows by consecutively using associativity, commutativity,
and associativity combined with the Frobenius property.
\end{proof}

We are now in a position to establish various convenient identities that are
satisfied for any factorizable ribbon Hopf algebra $H$.
First we obtain various equalities involving the coproduct and monodromy matrix.

\begin{lemma}
The following identities hold:
  \eqpic{Q2Delta} {100} {38} {\setulen80
  \put(0,0)    {\includepichtft{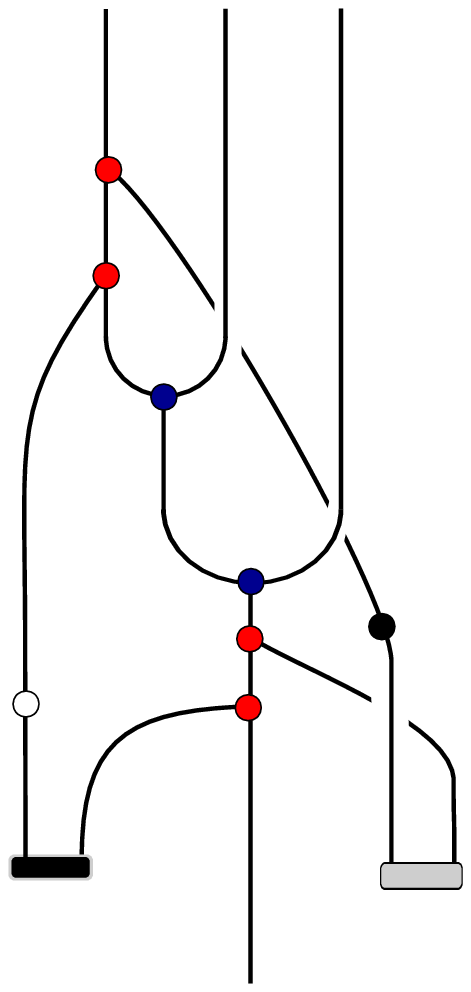}}
  \put(21.2,-9.5)  {\sse$ H $}
  \put(6.6,112)    {\sse$ H $}
  \put(19.7,112)   {\sse$ H $}
  \put(32.2,112)   {\sse$ H $}
  \put(-2.8,3.2)   {\sse$Q^{-1}$}
  \put(41.5,1.8)   {\sse$Q$}
  \put(78,44)      {$=$}
  \put(120,0) {
  \put(0,0)    {\includepichtft{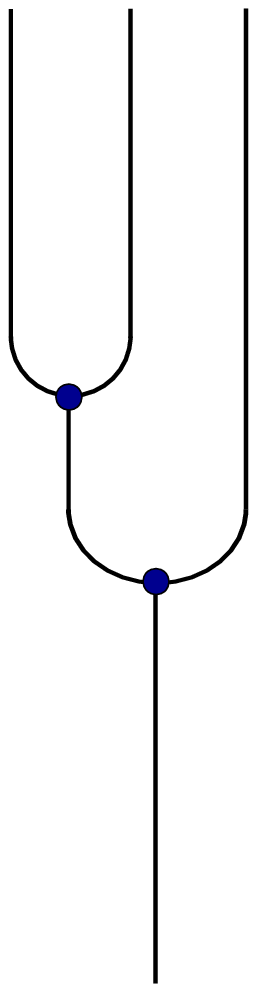}}
  \put(11.7,-9.2)  {\sse$ H $}
  \put(-3.6,112)   {\sse$ H $}
  \put( 9.5,112)   {\sse$ H $}
  \put(22.2,112)   {\sse$ H $}
  } }
and
  \eqpic{Q3Delta} {150} {55} {\setulen80
  \put(0,-9)  {
  \put(0,0)  {\includepichtft{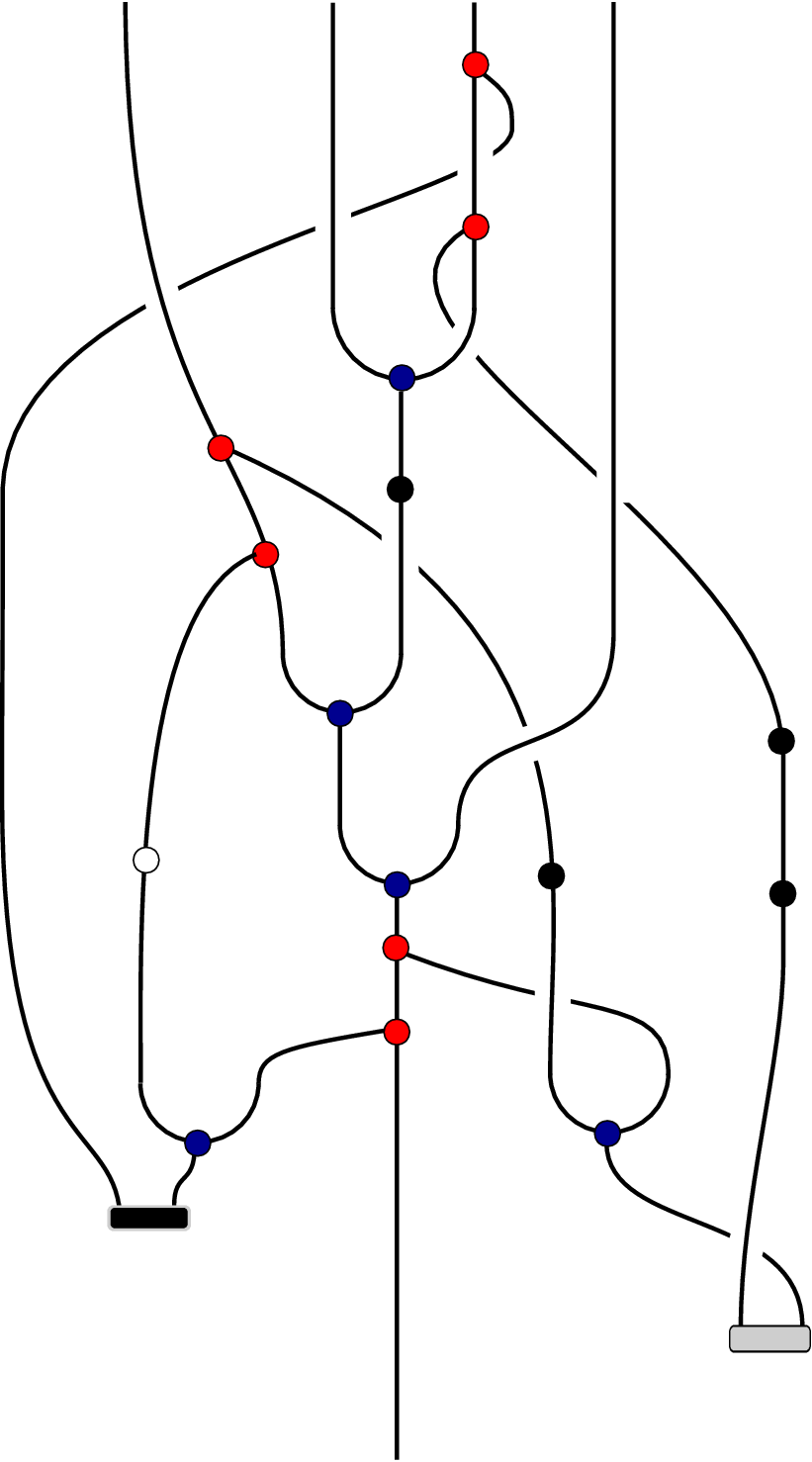}}
  \put(9.5,166.1)  {\sse$ H $}
  \put(33.1,166.1) {\sse$ H $}
  \put(39.2,-9.8)  {\sse$ H $}
  \put(48.2,166.1) {\sse$ H $}
  \put(63.7,166.1) {\sse$ H $}
  \put(9.2,16.4)   {\sse$Q^{-1}$}
  \put(80.3,2.5)   {\sse$Q$}
  \put(116,70)     {$=$}
  \put(146,0) {
  \put(0,0)  {\includepichtft{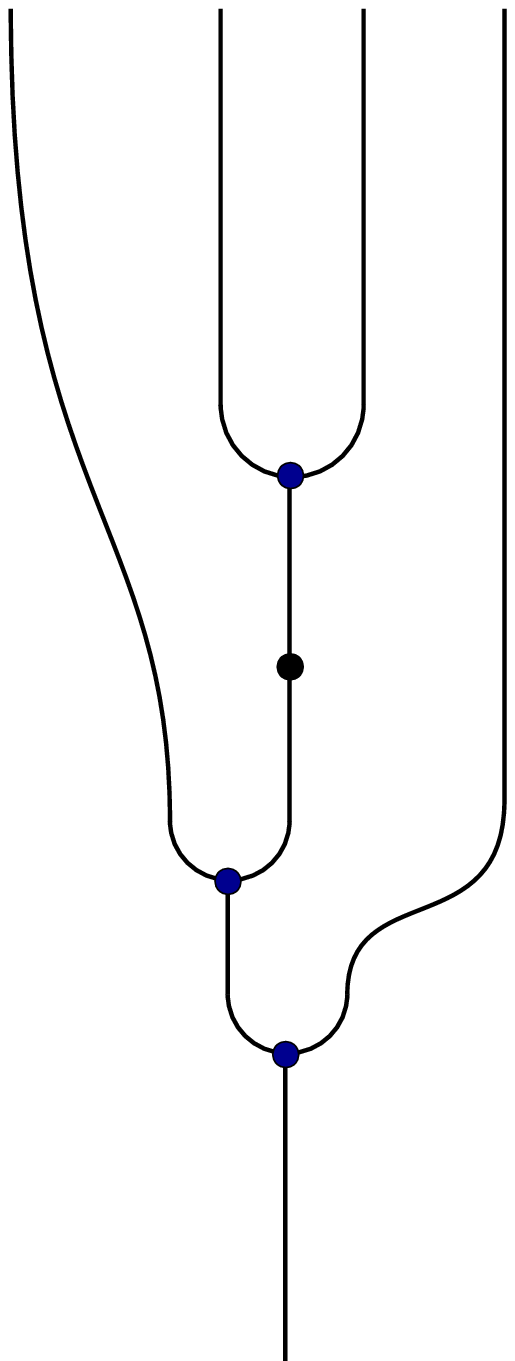}}
  \put(-4.1,166.1) {\sse$ H $}
  \put(18.9,166.1) {\sse$ H $}
  \put(25.4,-9.8)  {\sse$ H $}
  \put(35.1,166.1) {\sse$ H $}
  \put(50.7,166.1) {\sse$ H $}
  } } }
and
  \eqpic{Pic_QQ3} {370} {38} {\setulen80
  \put(0,0) { \includepichtft{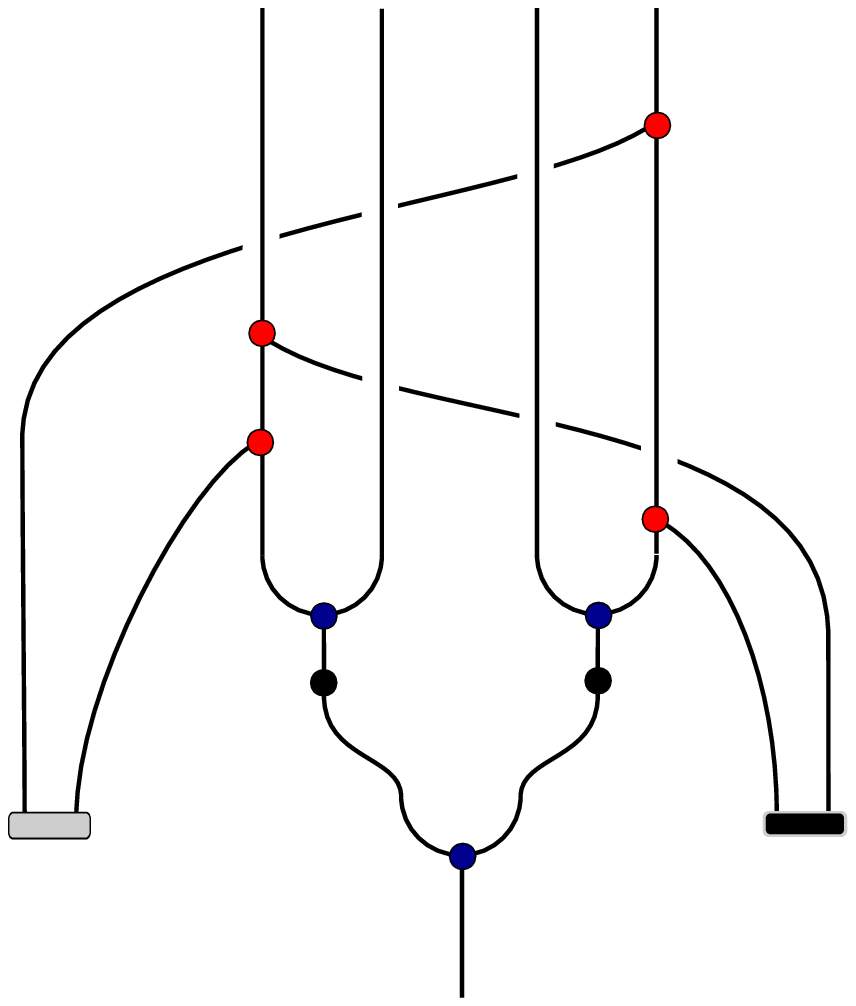}
  \put(24,114.1)   {\sse$ H $}
  \put(37.1,114.1) {\sse$ H $}
  \put(45,-11.2)   {\sse$ H $}
  \put(54,114.1)   {\sse$ H $}
  \put(67.5,114.1) {\sse$ H $}
  \put(0,8)        {\sse$Q$}
  \put(78.6,8.5)   {\sse$Q^{-1}$}
  }
  \put(116,52)     {$=$}
  \put(143,0) { \includepichtft{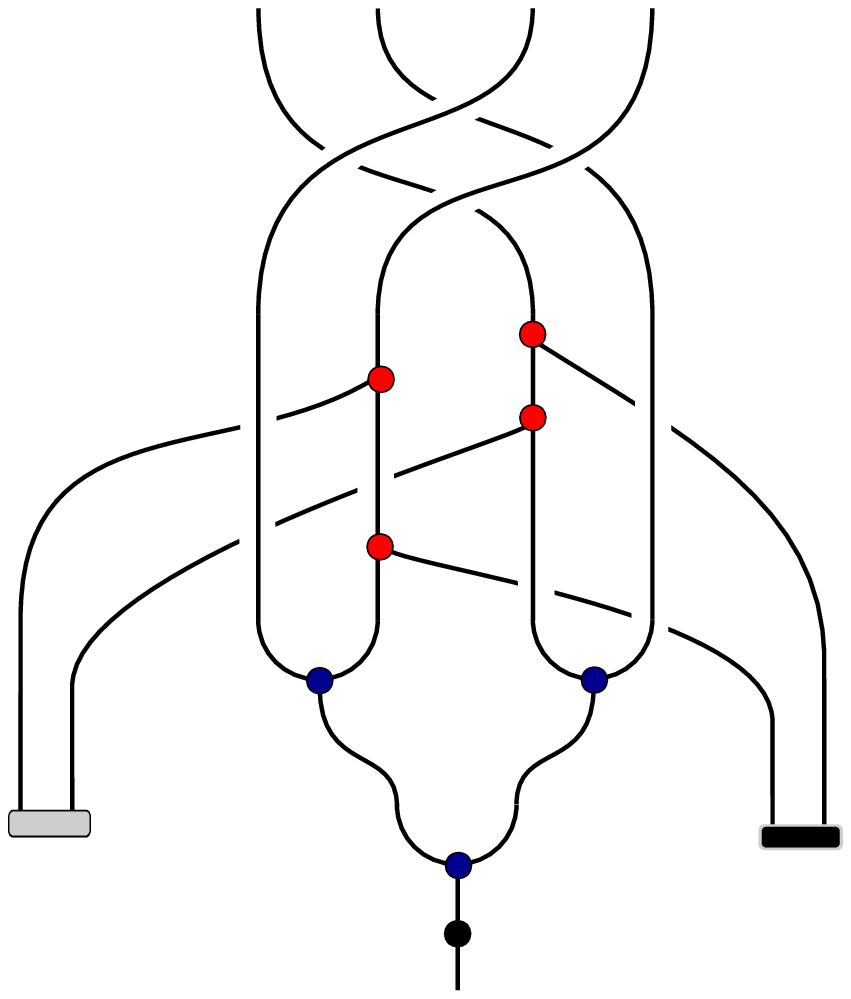}
  \put(23.3,114.1) {\sse$ H $}
  \put(36.7,114.1) {\sse$ H $}
  \put(44.2,-11.2) {\sse$ H $}
  \put(54,114.1)   {\sse$ H $}
  \put(67.5,114.1) {\sse$ H $}
  \put(0,8)        {\sse$Q$}
  \put(78.6,6.9)   {\sse$Q^{-1}$}
  }
  \put(264,52)     {$=$}
  \put(300,0) { \includepichtft{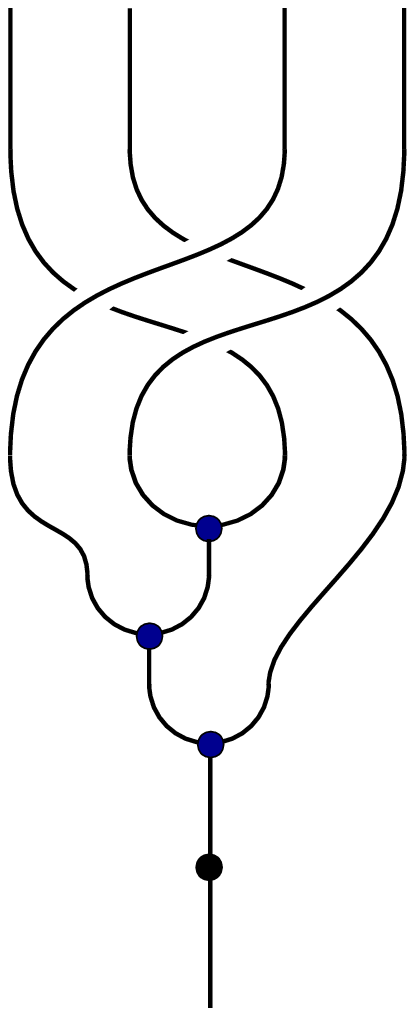}
  \put(-2.5,114.1) {\sse$ H $}
  \put(11.5,114.1) {\sse$ H $}
  \put(19,-11.2)   {\sse$ H $}
  \put(28,114.1)   {\sse$ H $}
  \put(40.7,114.1) {\sse$ H $}
  }
  \put(378,52)     {$=$}
  \put(409,0) { \includepichtft{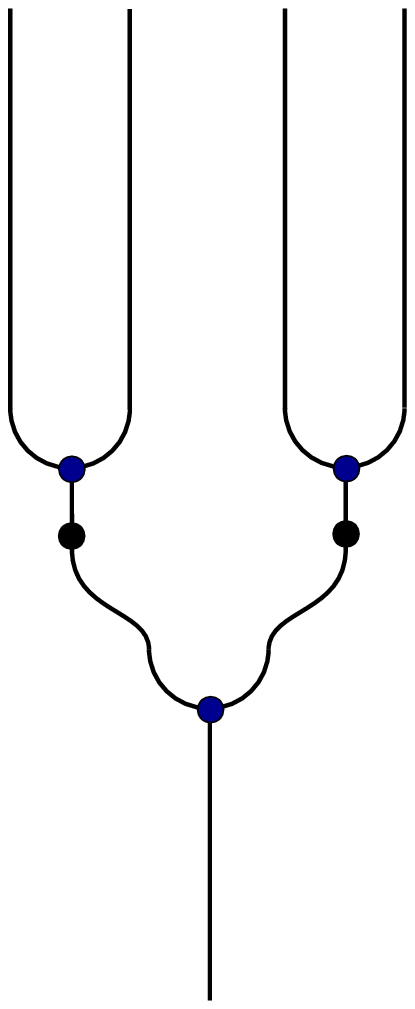}
  \put(-3.7,114.1) {\sse$ H $}
  \put(9.3,114.1)  {\sse$ H $}
  \put(17.5,-11.2) {\sse$ H $}
  \put(25.6,114.1) {\sse$ H $}
  \put(39.8,114.1) {\sse$ H $}
  } }
and
  \eqpic{inv_Qq3} {410} {47} { \put(0,2) { \setulen80
    \put(-2,0) { \includepichtft{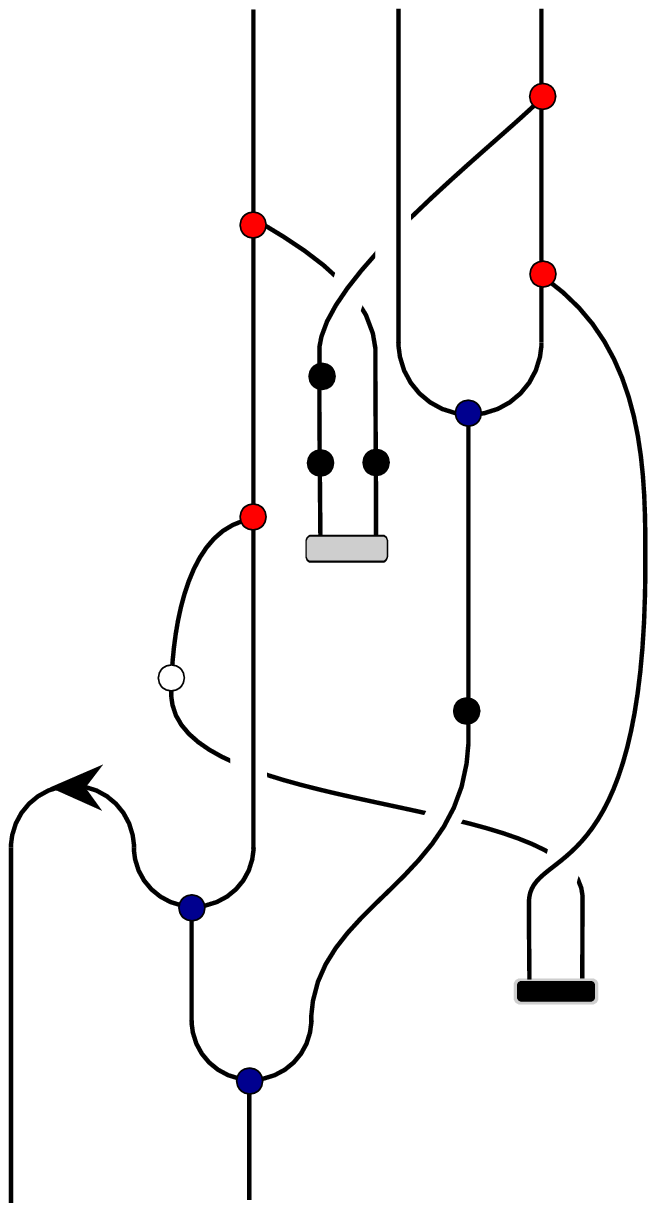}
  \put(-6.2,-11.2) {\sse$ \Hss $}
  \put(21.2,-11.2) {\sse$ H $}
  \put(22.8,135.2) {\sse$ H $}
  \put(32.6,60.9)  {\sse$Q$}
  \put(37.9,135.2) {\sse$ H $}
  \put(51.2,13.2)  {\sse$Q^{-1}$}
  \put(54.2,135.2) {\sse$ H $}
  }
  \put(103,57)     {$=$}
    \put(144,0) { \includepichtft{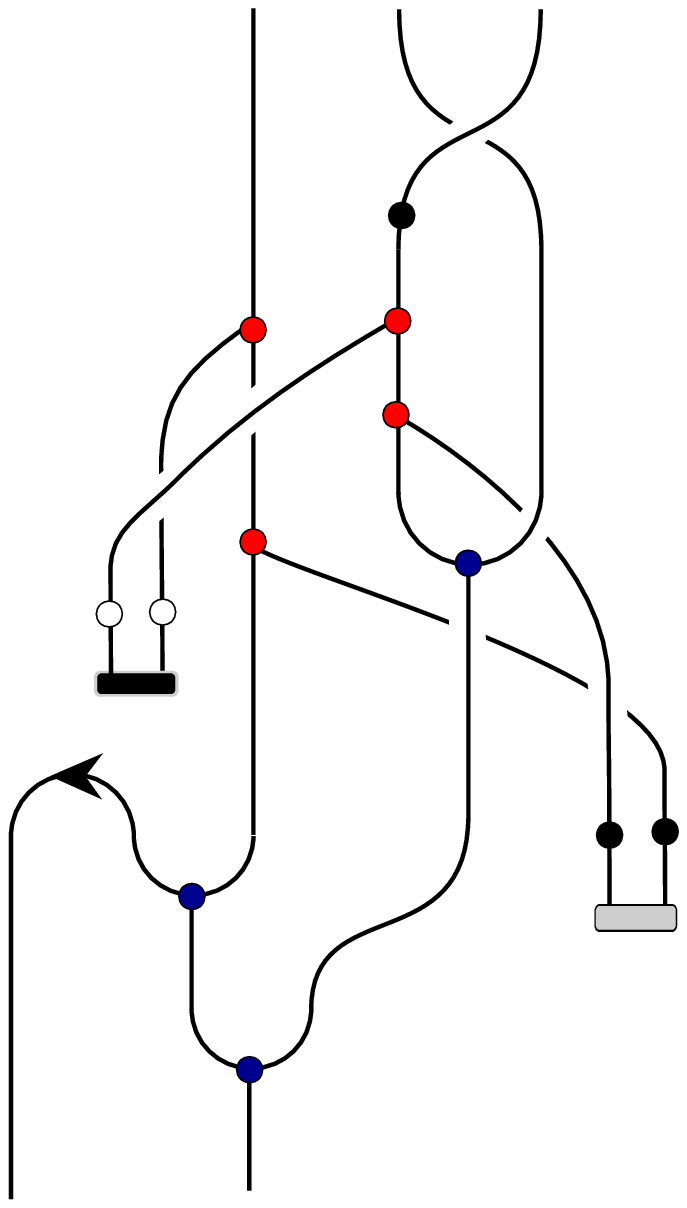}
  \put(-11.5,53)   {\sse$Q^{-1}$}
  \put(-6.5,-11.2) {\sse$ \Hss $}
  \put(21.2,-11.2) {\sse$ H $}
  \put(22.7,135.2) {\sse$ H $}
  \put(39.2,135.2) {\sse$ H $}
  \put(55.2,135.2) {\sse$ H $}
  \put(64.6,21.3)  {\sse$Q$}
  }
  \put(254,57)     {$=$}
    \put(292,0) { \includepichtft{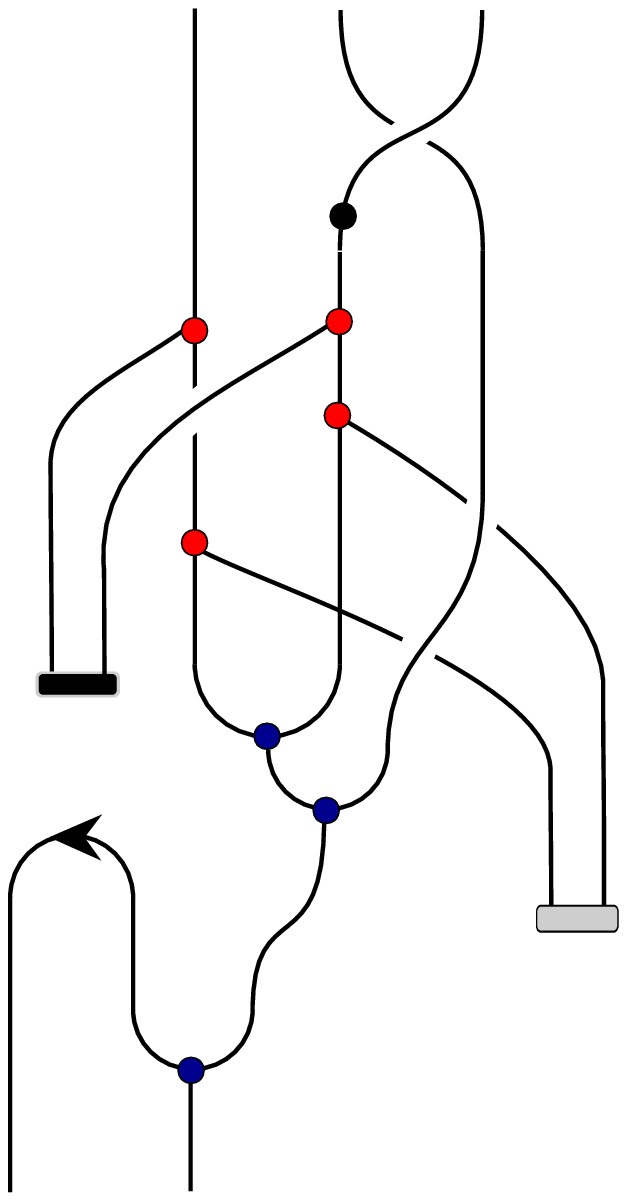}
  \put(-3,45.4)    {\sse$Q^{-1}$}
  \put(-6.2,-11.2) {\sse$ \Hss $}
  \put(15.2,-11.2) {\sse$ H $}
  \put(15.7,135.2) {\sse$ H $}
  \put(32.1,135.2) {\sse$ H $}
  \put(48.3,135.2) {\sse$ H $}
  \put(58.5,20.7)  {\sse$Q$}
  }
  \put(395,57)     {$=$}
    \put(431,0) { \includepichtft{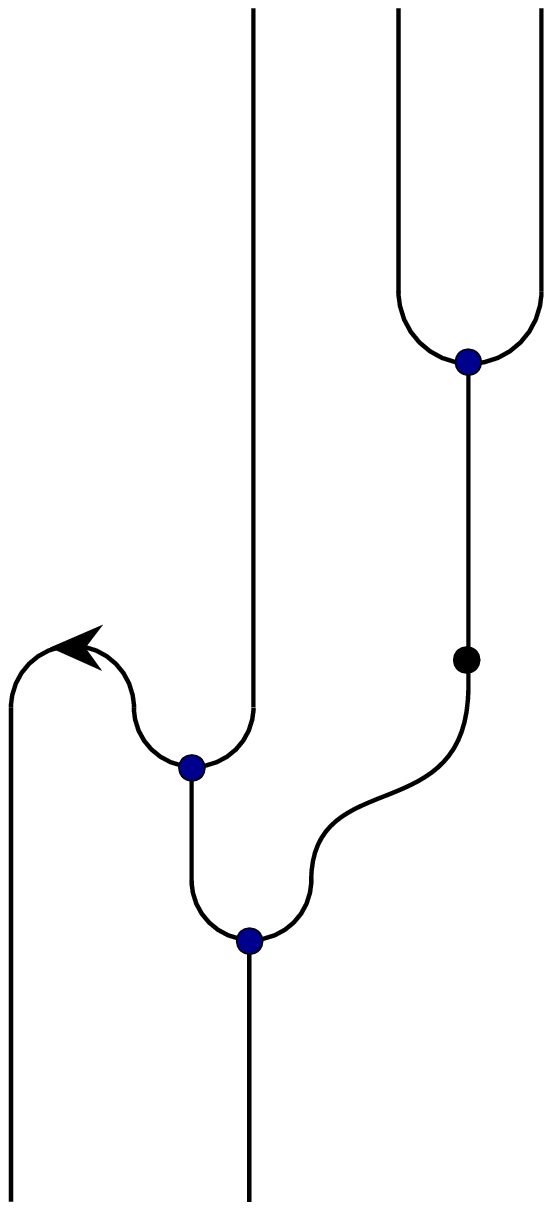}
  \put(-6.2,-11.2) {\sse$ \Hss $}
  \put(21.4,-11.2) {\sse$ H $}
  \put(21.6,135.2) {\sse$ H $}
  \put(37.8,135.2) {\sse$ H $}
  \put(54.3,135.2) {\sse$ H $}
  } } }
\end{lemma}

\begin{proof}
(i)\, Consider the equality between the left- and right-most expressions in
\erf{2prod_mon} for the case of our interest, i.e.\ with $A \eq F$ and with $c$
the braiding in \HBimod. This involves the left-dual actions $\rhoV^H$ and
$\ohrV^H$ of $H$ on the bimodule ${}_{}^\vee\! F$. The latter obey
  \eqpic{Fv_act} {160} {33} {
    \put(-4,0) {\Includepichtft{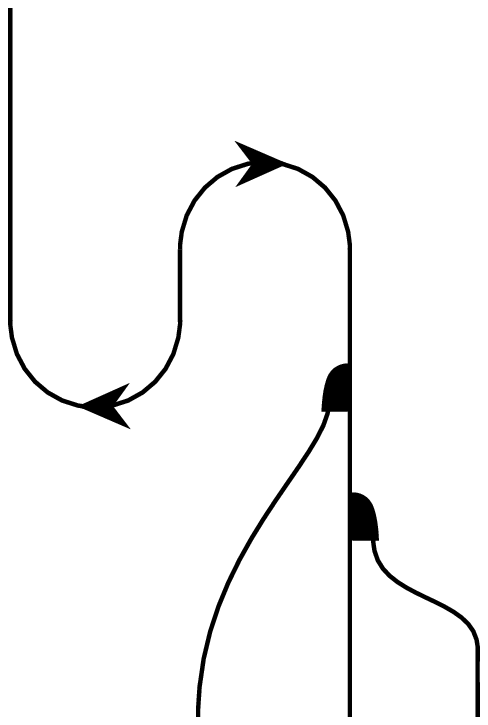}
  \put(-3,82)      {\sse$ H $}
  \put(16.5,-9.2)  {\sse$ H $}
  \put(31,-9.2)    {\sse$ ^\vee\!F $}
  \put(47.7,-9.2)  {\sse$ H$}
  \put(41.5,20.5)  {\sse$ \ohrV^H $}
  \put(39.2,35.5)  {\sse$ \rhoV^H $}
  }
  \put(73,34)      {$=$}
    \put(104,0) {\Includepichtft{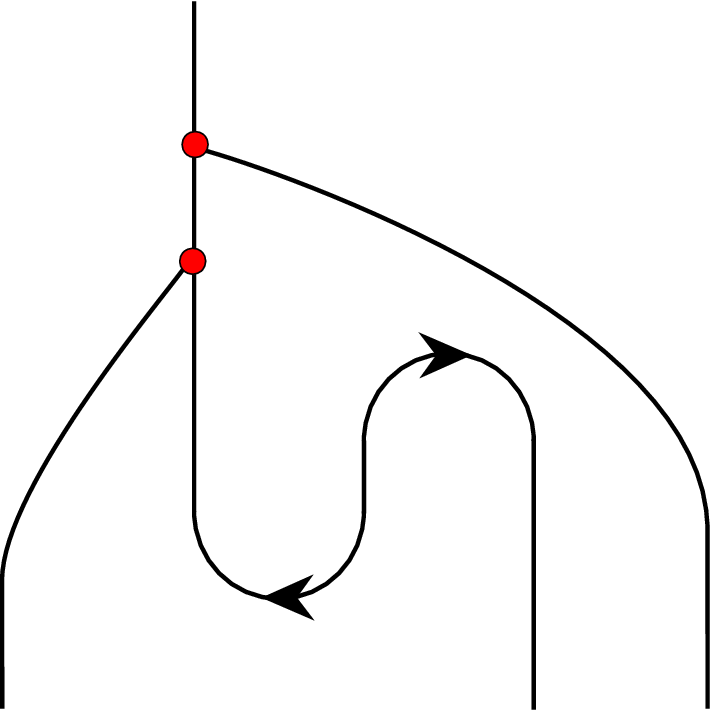}
  \put(-4.3,-9.2)   {\sse$ H $}
  \put(51,-9.2)    {\sse$ ^\vee\! F $}
  \put(73,-9.2)    {\sse$ H$}
  \put(17.5,82)    {\sse$ H $}
  } }
Implementing this relation and composing the resulting equality with suitable
duality morphisms yields \erf{Q2Delta}.
\\[2pt]
(ii)\,
On the left hand side of \erf{Q3Delta}, push the inverse antipode that
is located highest up in the picture to the top of the picture, i.e.\ through
a coproduct and two products, and push the two `lowest' products upwards by
invoking the connecting axiom of the bialgebra $H$. This results in the left
hand side of the following chain of equalities:
  \eqpic{Q3Delta_proof} {420} {60} { \put(-17,0) { \setulen80
    \put(0,0) { \includepichtft{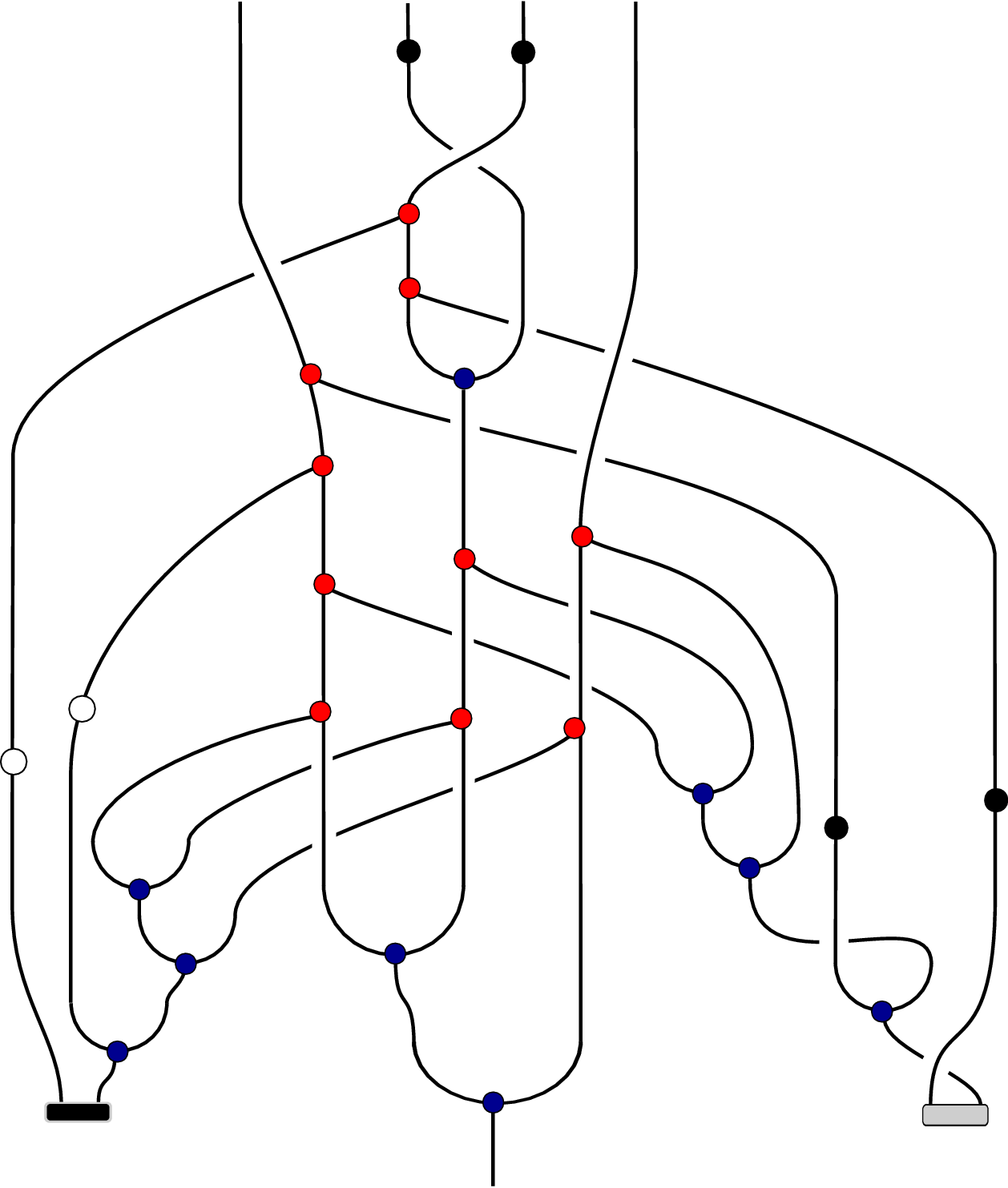}
  \put(1.1,0)      {\sse$ Q^{-1} $}
  \put(28.7,166.2) {\sse$ H $}
  \put(50.7,166.2) {\sse$ H $}
  \put(61.5,-11.2) {\sse$ H $}
  \put(66.7,166.2) {\sse$ H $}
  \put(82.7,166.2) {\sse$ H $}
  \put(125.5,-.8)  {\sse$ Q $}
  }
  \put(169,74) {$=$}
    \put(204,0) { \includepichtft{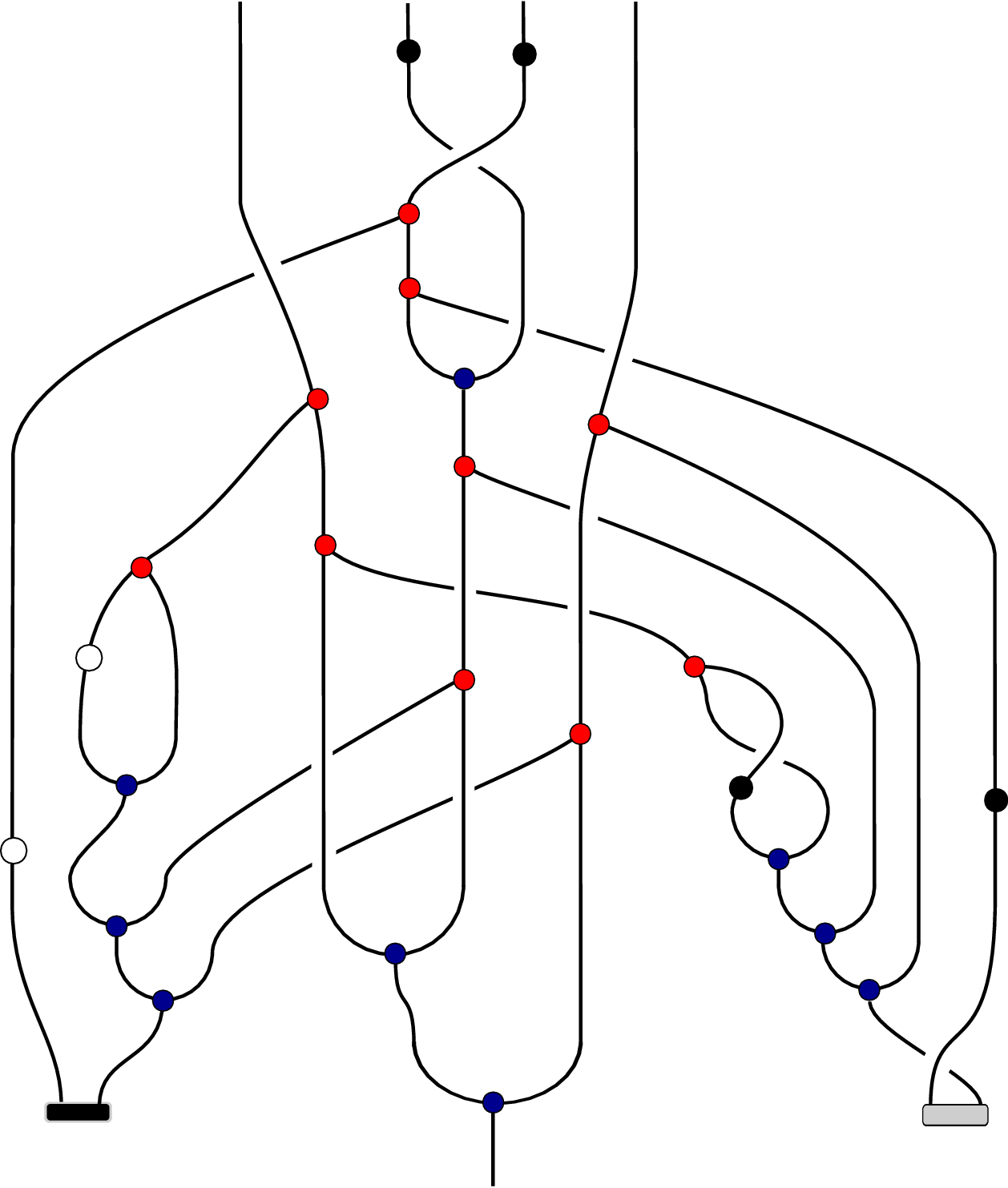}
  \put(1.5,.2)     {\sse$ Q^{-1} $}
  \put(26.7,166.2) {\sse$ H $}
  \put(50.7,166.2) {\sse$ H $}
  \put(61.8,-11.2) {\sse$ H $}
  \put(66.7,166.2) {\sse$ H $}
  \put(82.7,166.2) {\sse$ H $}
  \put(125.2,-.8)  {\sse$ Q $}
  }
  \put(373,74) {$=$}
    \put(409,0) { \includepichtft{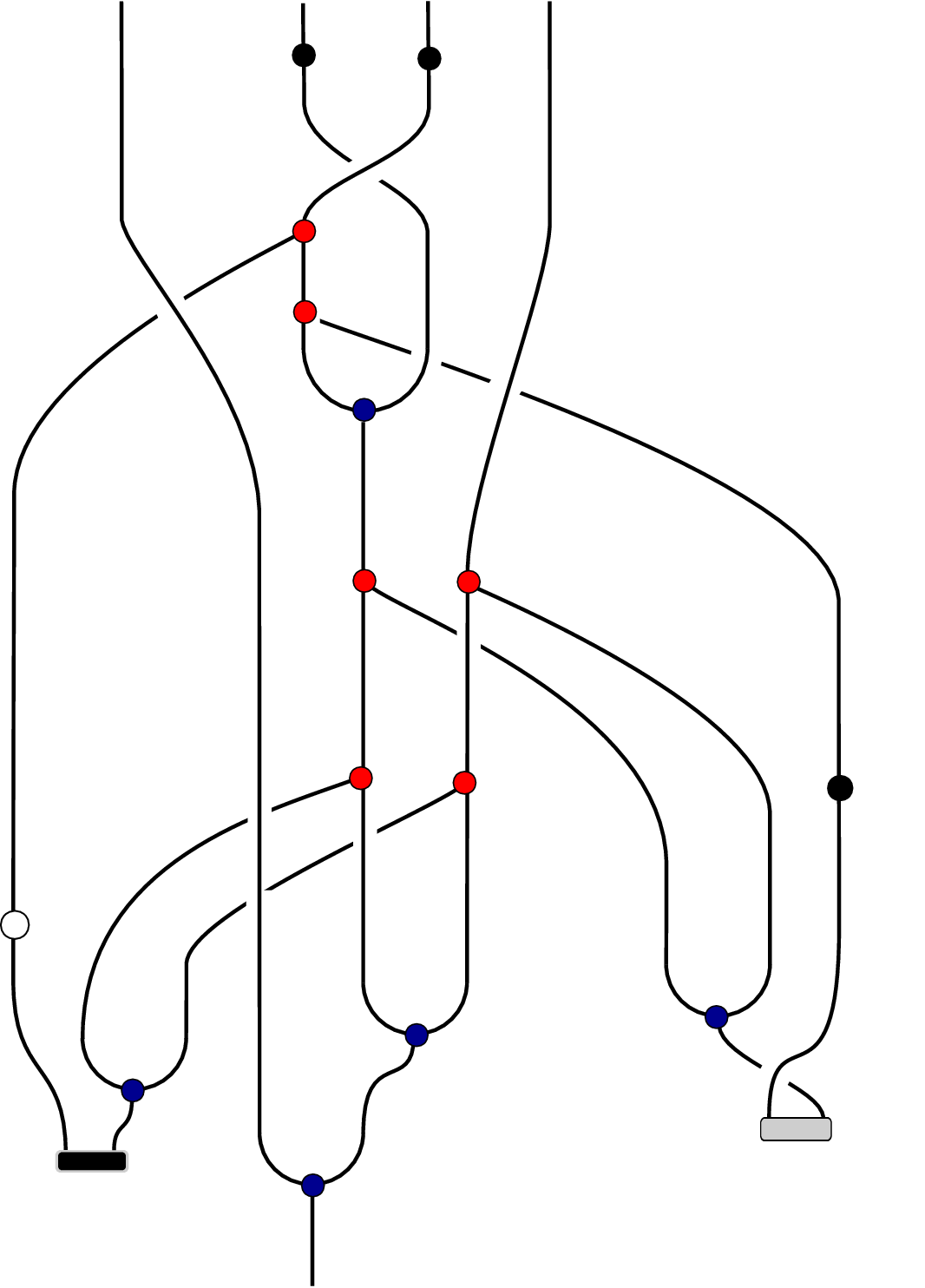}
  \put(2.1,6.1)    {\sse$ Q^{-1} $}
  \put(9.7,166.2)  {\sse$ H $}
  \put(32.9,166.2) {\sse$ H $}
  \put(33.2,-11.2) {\sse$ H $}
  \put(50.1,166.2) {\sse$ H $}
  \put(66.2,166.2) {\sse$ H $}
  \put(95.2,9.5)   {\sse$ Q $}
  } } }
The first of these equalities follows easily by a multiple use of associativity
and coassociativity, while the second one is obtained with the help of the
defining property of the antipode and of coassociativity.
Now the right hand side of \erf{Q3Delta_proof} equals the right hand side of
\erf{Q3Delta}, as is seen by first applying the connecting axiom to the two
coproducts that come directly after the monodromy matrices and then invoking
the identity \erf{Q2Delta}.
\\[2pt]
(iii)\,
The first equality in \erf{Pic_QQ3} is obtained by pushing the inverse antipodes down
through the coproduct (together with a deformation compatible with the braid
relations). The second of the equalities in \erf{Pic_QQ3} follows from
coassociativity combined with \erf{removeQs}, while the third equality is
immediate from coassociativity and properties of the inverse antipode.
\\[2pt]
(iv)\,
The first equality in \erf{inv_Qq3} follows by pushing $\apoi$ through the upper
coproduct and products. The second equality uses coassociativity and
$(\apo\oti\apo)\cir Q \eq \tauHH\cir Q$. Applying first \erf{removeQs} and then
undoing the manipulations performed in the first two equalities yields the third one.
\end{proof}

The next two identities involve the two-sided integral $\Lambda$ and the right
cointegral $\lambda$ of $H$.

\begin{lemma}\label{lem:reppastint}
We have
  \eqpic{reppastint} {100} {25} {
  \put(0,0) {\Includepichtft{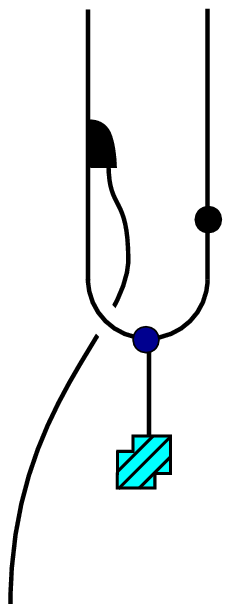}
  \put(-4,-8.5)    {\sse$ H $}
  \put(19.2,12.9)  {\sse$ \Lambda $}
  \put(-1.8,49.3)  {\sse$ \ohrad $}
  \put(6,68.5)     {\sse$ H $}
  \put(19,68.5)    {\sse$ H $}
  }
  \put(50,30)      {$=$}
  \put(85,0) {\Includepichtft{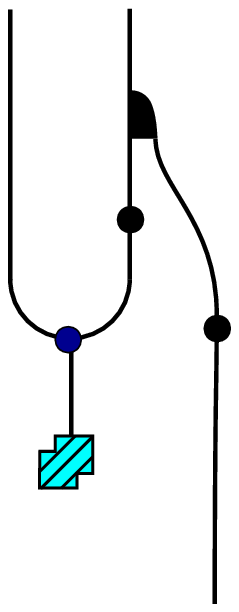}
  \put(18.5,-8.5)  {\sse$ H $}
  \put(-4.5,12.8)  {\sse$ \Lambda $}
  \put(17.6,52.8)  {\sse$ \ohrad $}
  \put(-3,68.5)    {\sse$ H $}
  \put(10,68.5)    {\sse$ H $}
  } }
where $\ohrad$ is the right-adjoint action of $H$ on itself
{\rm(}as defined in {\rm\erf{Corr12H})}.
\end{lemma}

\begin{proof}
This follows by combining the explicit form
of $\ohrad$ with the identities \cite[(2.18)]{fuSs3}
   \eqpic{Hopf_Frob_trick2} {290} {29} {\setulen80
     \put(0,0) {\includepichtft{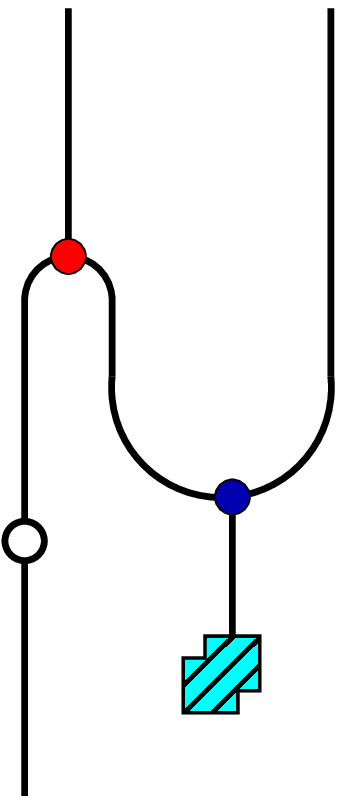} \put(2.2,0){
   \put(-4.2,-11.2){\sse$ H $}
   \put(2.2,92)    {\sse$ H $}
   \put(27.7,10)   {\sse$ \Lambda $}
   \put(31,92)     {\sse$ H $}
   } }
   \put(60,41)     {$ = $}
     \put(86,0) { \includepichtft{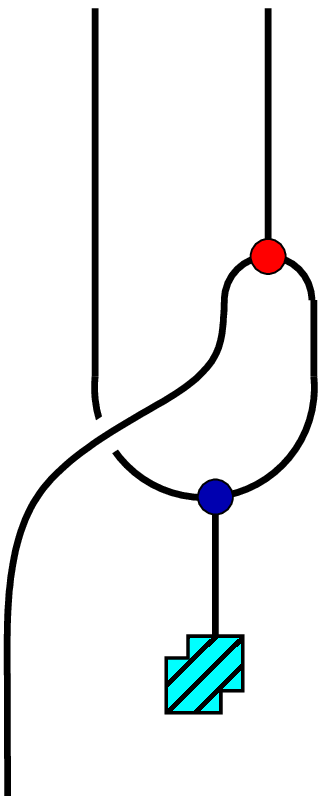}
   \put(-4.2,-11.2){\sse$ H $}
   \put(6.6,92)    {\sse$ H $}
   \put(26.6,92)   {\sse$ H $}
   \put(27.7,10)   {\sse$ \Lambda $}
   }
   \put(174,41)    {and}
     \put(240,0) { \includepichtft{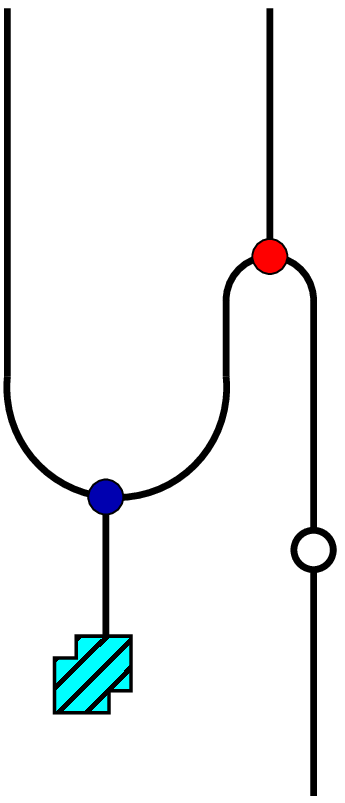}
   \put(-3.2,92)   {\sse$ H $}
   \put(15.7,10)   {\sse$ \Lambda $}
   \put(26.1,92)   {\sse$ H $}
   \put(29.8,-11.2){\sse$ H $}
   }
   \put(300,41)  {$ = $}
     \put(326,0) { \includepichtft{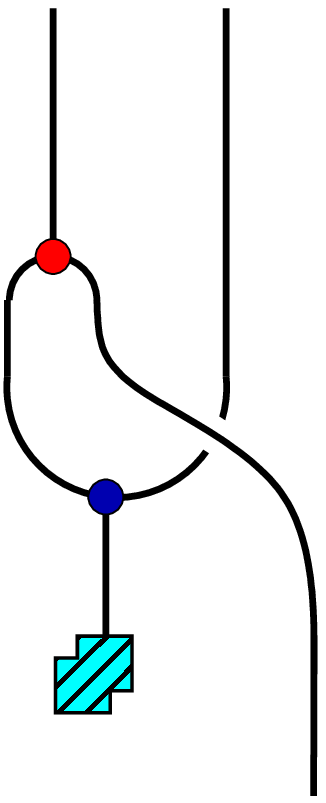}
   \put(1.2,92)    {\sse$ H $}
   \put(15.7,10)   {\sse$ \Lambda $}
   \put(21.7,92)   {\sse$ H $}
   \put(29.8,-11.2){\sse$ H $}
   } }
(which, in turn, are obtained by combining various defining properties of the Hopf
algebra structure of $H$ with unimodularity).
\end{proof}

\begin{lemma}\label{lem:Sinv-7}
We have the chain
  \eqpic{Sinv-7} {410} {47} {
    \put(0,0) {\Includepichtft{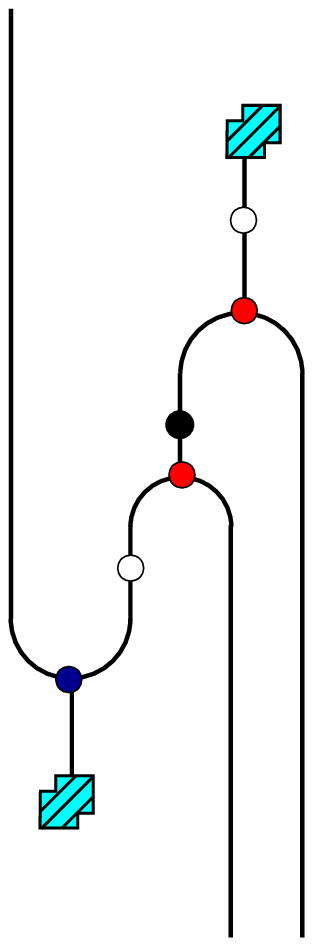}
  \put(18,-8)      {\sse$ \Hss $}
  \put(30,-8)      {\sse$ \Hss $}
  \put(11,12)      {\sse$ \Lambda $}
  \put(31.5,86)    {\sse$ \lambda $}
  \put(-3,106)     {\sse$ \Hss $}
  }
  \put(58,47)      {$ = $}
    \put(95,0)  {\Includepichtft{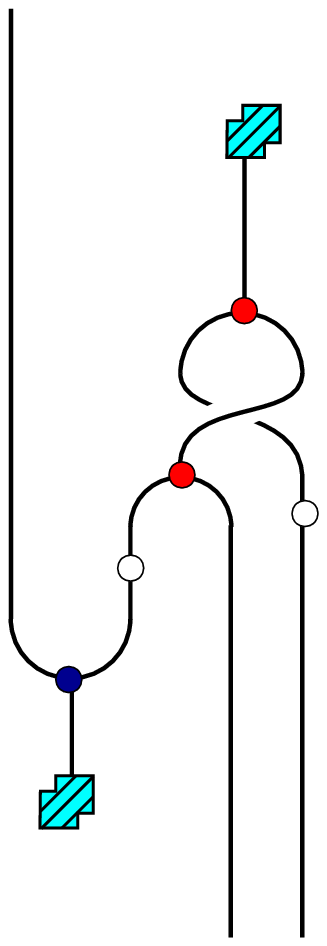}
  \put(18,-9.2)    {\sse$ \Hss $}
  \put(30,-9.2)    {\sse$ \Hss $}
  \put(11,12)      {\sse$ \Lambda $}
  \put(31.5,86)    {\sse$ \lambda $}
  \put(-3,106)     {\sse$ \Hss $}
  }
  \put(153,47)     {$=$}
    \put(190,0) {\Includepichtft{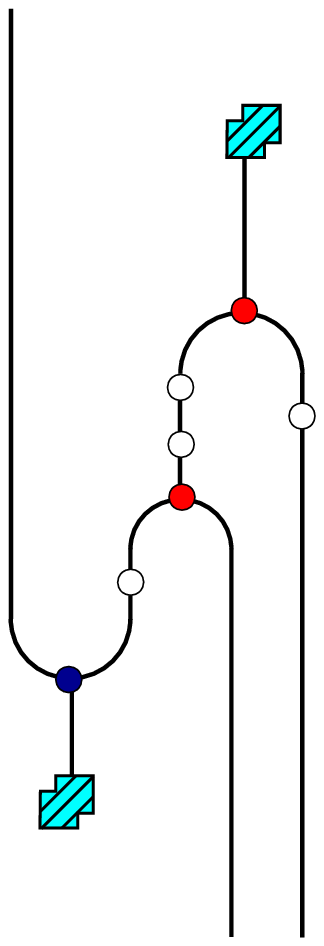}
  \put(18,-9.2)    {\sse$ \Hss $}
  \put(30,-9.2)    {\sse$ \Hss $}
  \put(11,12)      {\sse$ \Lambda $}
  \put(31.5,86)    {\sse$ \lambda $}
  \put(-3,106)     {\sse$ \Hss $}
  }
  \put(249,47)     {$=$}
    \put(282,0) {\Includepichtft{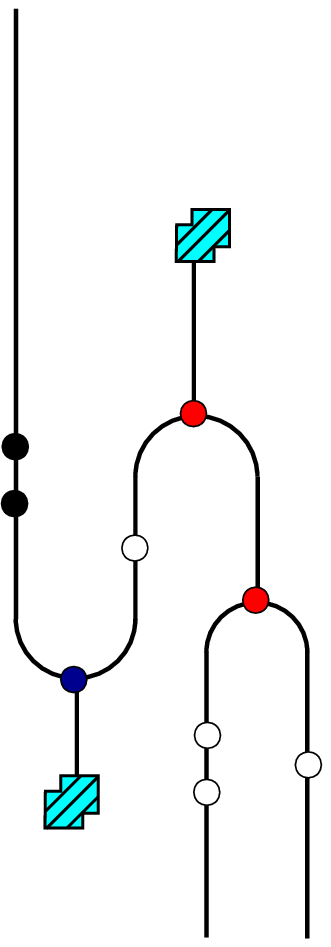}
  \put(17.5,-9.2)  {\sse$ \Hss $}
  \put(30,-9.2)    {\sse$ \Hss $}
  \put(11.4,11)    {\sse$ \Lambda $}
  \put(26.5,74)    {\sse$ \lambda $}
  \put(-1.4,106)   {\sse$ \Hss $}
  }
  \put(339,47)     {$ = $}
    \put(375,16) {\Includepichtft{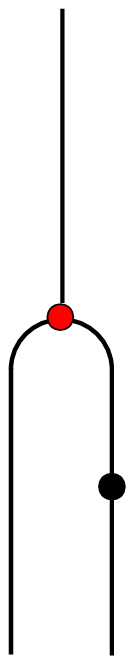}
  \put(-5,-9.2)    {\sse$ \Hss $}
  \put(8,-8)       {\sse$ \Hss $}
  \put(3,75)       {\sse$ \Hss $}
  } }
of equalities.
\end{lemma}

\begin{proof}
The second equality uses the fact that, since $H$ is unimodular, the cointegral
satisfies \cite[Thm.\,3]{radf12} $\lambda \cir m \eq \lambda \cir m \cir \tauHH
\cir (\id_H\oti\apo^2)$. The third follows with the help of $\apo\cir\Lambda \eq
\Lambda$. The important ingredient in the last equality is the invertibility of
the Frobenius map, which is equivalent to the  statement that $H$ has a natural
structure of a Frobenius algebra (see e.g.\ \cite[Eq.\,(8)]{coWe6}).
\end{proof}

Finally we have the following result involving somewhat more complicated
expressions, which arise in connection with the action \erf{LyubactC3} of the
generators $t_{j,k}$ of the mapping class group.

\begin{lemma}\label{lem:Hpastloops}
For any integer $p \,{\ge}\, 1$ we have
  \eqpic{Hpastloops} {320} {101} {\setulen80
      \put(0,0) {\includepichtft{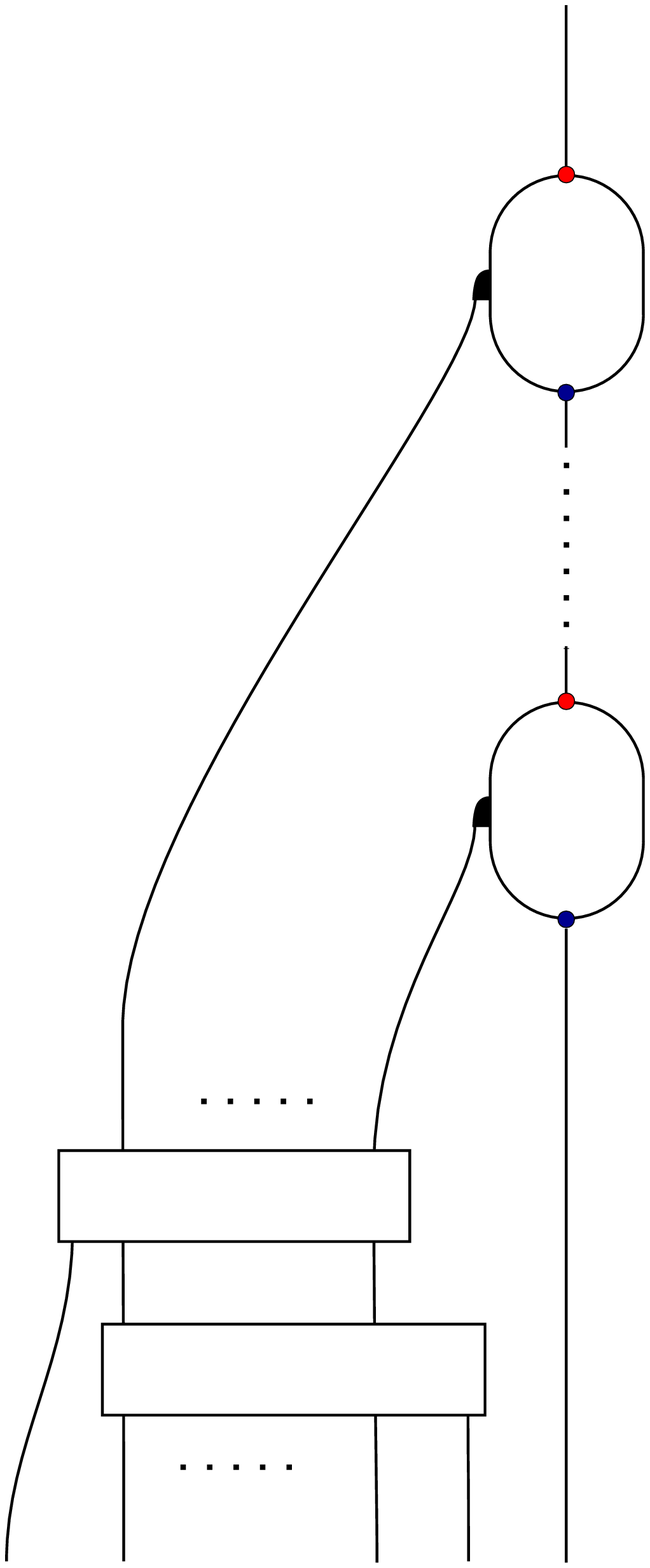}
  \put(-5,-11.2)   {\sse$ H $}
  \put(17.2,-11.2) {\sse$ K $}
  \put(60.2,-11.2) {\sse$ K $}
  \put(66.8,128.2) {\sse$ \rho^K_F $}
  \put(66.8,218.9) {\sse$ \rho^K_F $}
  \put(76.2,-11.2) {\sse$ H $}
  \put(92.2,-11.2) {\sse$ F $}
  \put(92.9,273.3) {\sse$ F $}
  \put(29,63.8)    {\sse$ \rho_{\HK^{\!\otimes p}}^{} $}
  \put(39,33.8)    {\sse$ \ohr_{\HK^{\!\otimes p}}^{} $}
  }
  \put(152,115)    {$=$}
  \put(200,0) {\includepichtft{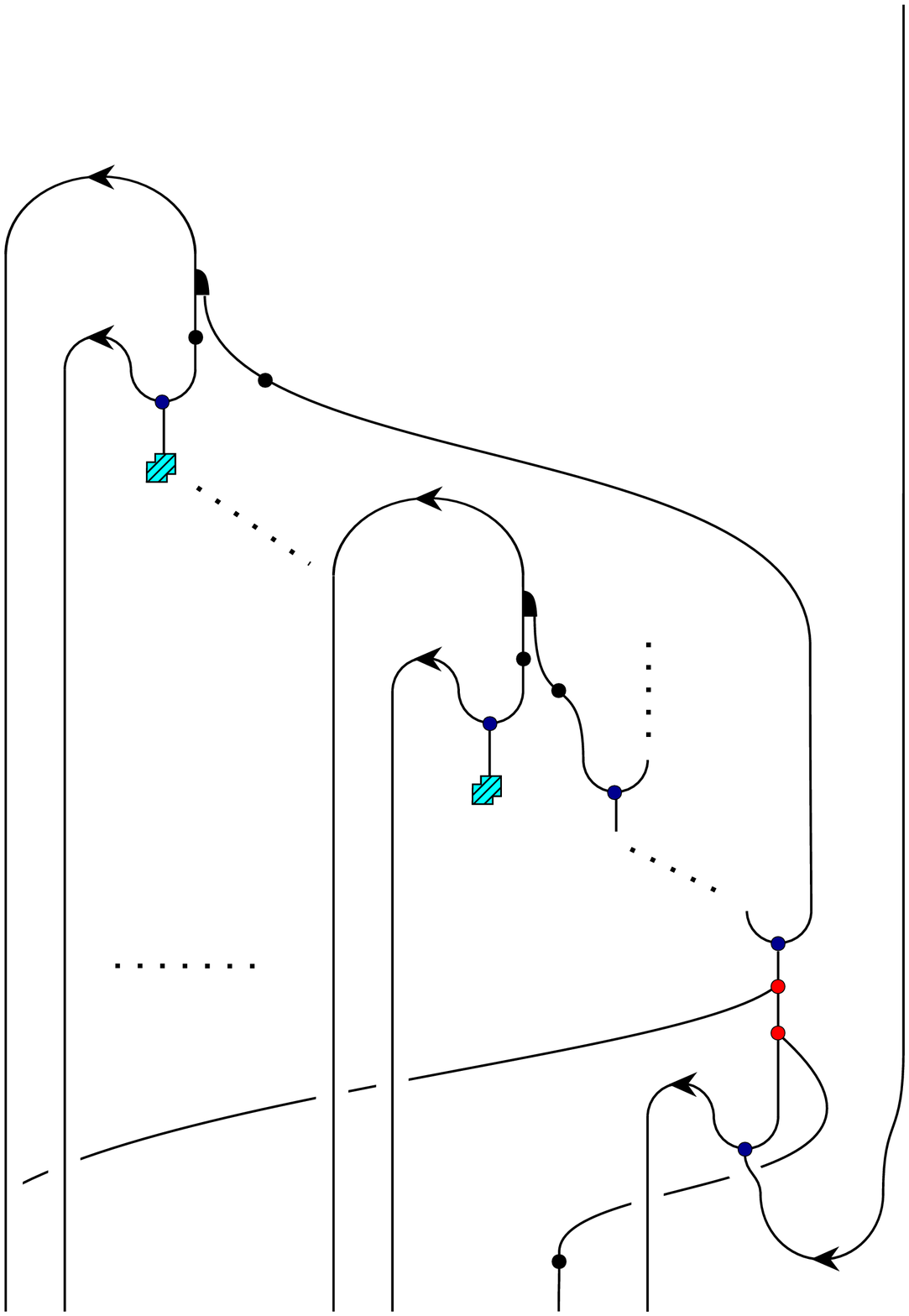}
  \put(-5,-11.2)   {\sse$ H $}
  \put(8.2,-11.2)  {\sse$ \Hss $}
  \put(20.2,-11.2) {\sse$ \Hss $}
  \put(76.2,-11.2) {\sse$ \Hss $}
  \put(88.2,-11.2) {\sse$ \Hss $}
  \put(121.2,-11.2){\sse$H$}
  \put(140.2,-11.2){\sse$ \Hss $}
  \put(192.2,272.5){\sse$ \Hss $}
  \put(34,171.5)   {\sse$ \Lambda $}
  \put(100,103.5)  {\sse$ \Lambda $}
  \put(57,208)     {\sse$ \ohrad $}
  \put(124,142)    {\sse$ \ohrad $}
  } }
where $\rho^{}_{K^{\otimes p}_{}}$ and $\ohr^{}_{K^{\otimes p}_{}}$ are
the left and right actions of the factorizable ribbon Hopf algebra $H$ on
the $H$-bimodule $K^{\otimes p}_{}$, respectively.
\end{lemma}

\begin{proof}
Recalling the expression \erf{CorrgnH} for $\Corr g11$, writing out the actions
$\rho^{}_{K^{\otimes p}_{}}$ and $\ohr^{}_{K^{\otimes p}_{}}$
and invoking Lemma \ref{lem:reppastint},
the left hand side of \erf{Hpastloops} becomes
  \eqpic{Hpastloops_1} {320} {103} {\setulen80
    \put(0,0) {\includepichtft{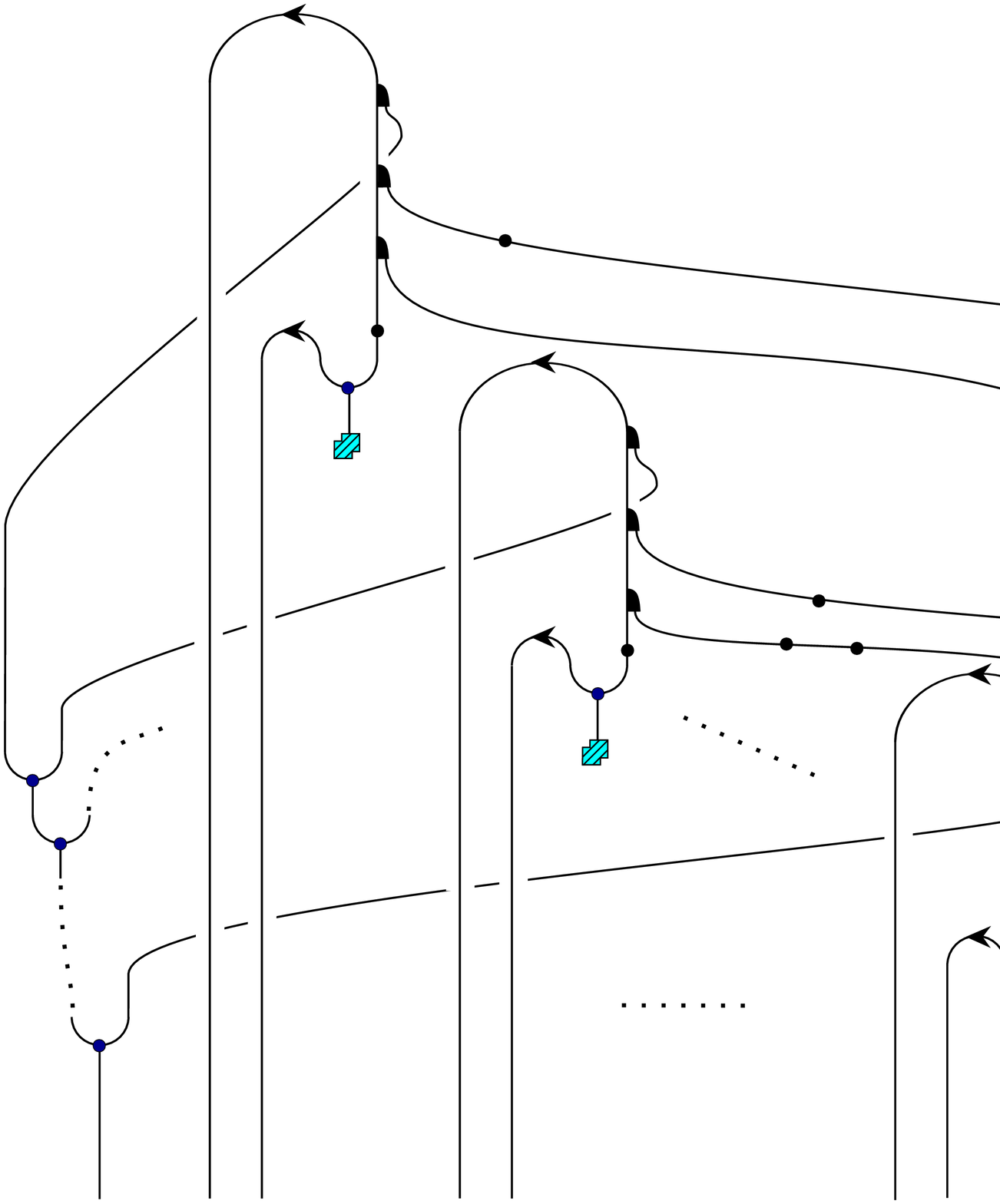}
  \put(16.4,-11.2) {\sse$ H $}
  \put(40.8,-11.2) {\sse$ \Hss $}
  \put(55.2,-11.2) {\sse$ \Hss $}
  \put(98.2,-11.2) {\sse$ \Hss $}
  \put(112.5,-11.2){\sse$ \Hss $}
  \put(148,-11.2)  {\sse$ \cdots\cdots $}
  \put(199.2,-11.2){\sse$ \Hss $}
  \put(214.2,-11.2){\sse$ \Hss $}
  \put(310.9,-11.2){\sse$ H $}
  \put(327.1,-11.2){\sse$ \Hss $}
  \put(396.1,289)  {\sse$ \Hss $}
  \put(84,171.5)   {\sse$ \Lambda$}
  \put(141,100.5)  {\sse$ \Lambda$}
  \put(241.5,31.1) {\sse$ \Lambda$}
  \put(92.2,217)   {\sse$\ohrad$}
  \put(92.2,234)   {\sse$\ohrad$}
  \put(92.2,254)   {\sse$\ohrad$}
  \put(149.7,137.5){\sse$\ohrad$}
  \put(149.7,155.8){\sse$\ohrad$}
  \put(149.7,175.3){\sse$\ohrad$}
  \put(249.7,63)   {\sse$\ohrad$}
  \put(249.7,82)   {\sse$\ohrad$}
  \put(249.7,102.8){\sse$\ohrad$}
  } }
Invoking the representation property of the right-adjoint action $\ohrad$
and the anti-(co)algebra
morphism property of the antipode several times, this morphism can be rewritten as
  \eqpic{Hpastloops_2} {320} {109} {\setulen80
    \put(0,0) {\includepichtft{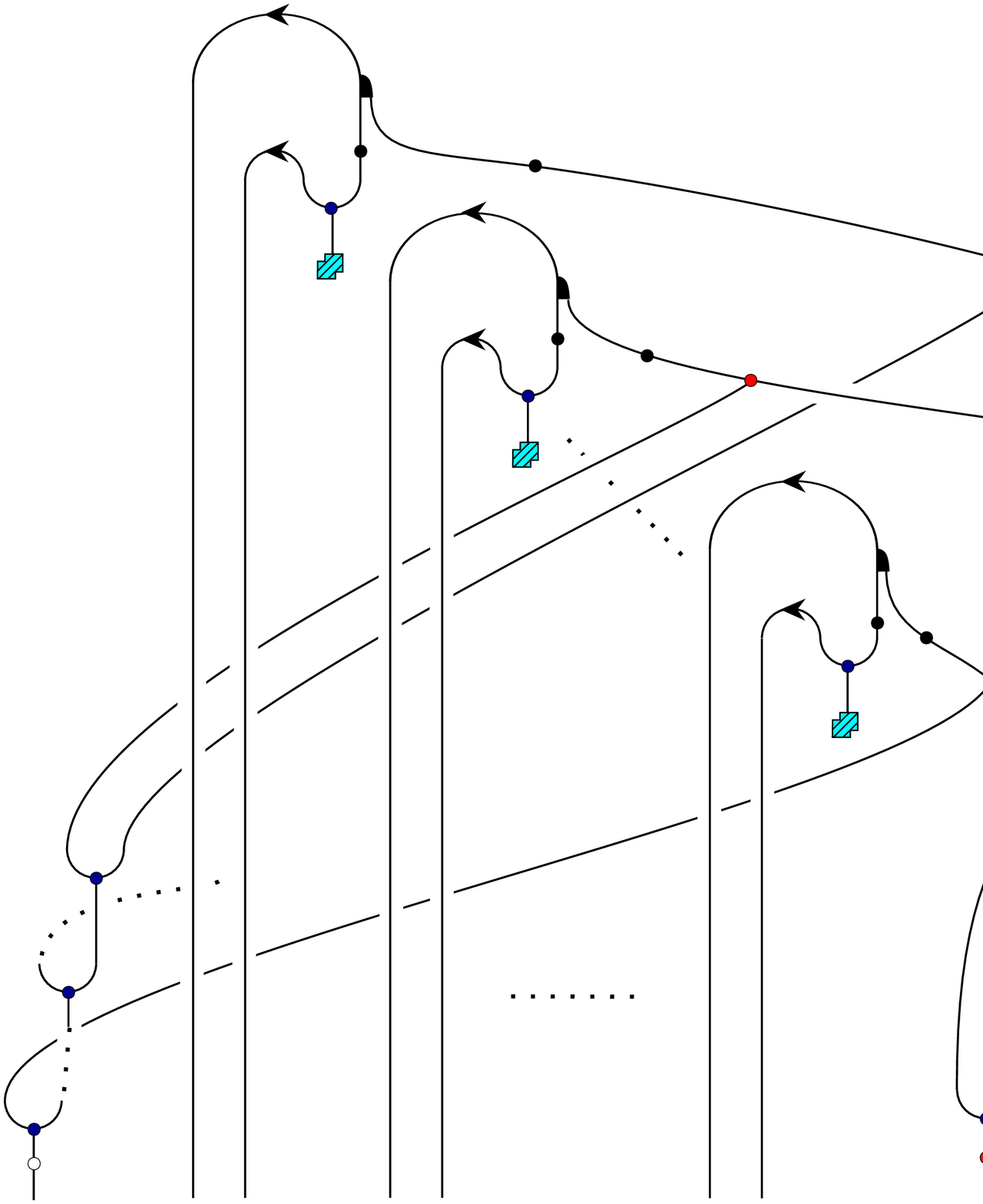}
  \put(2.2,-11.2)  {\sse$ H $}
  \put(36.8,-11.2) {\sse$ \Hss $}
  \put(51.8,-11.2) {\sse$ \Hss $}
  \put(69,-11.2)   {\sse$\cdots $}
  \put(83.5,-11.2) {\sse$ \Hss $}
  \put(97.5,-11.2) {\sse$ \Hss $}
  \put(157.2,-11.2){\sse$ \Hss $}
  \put(171.2,-11.2){\sse$ \Hss $}
  \put(222.2,-11.2){\sse$ H $}
  \put(323.5,-11.2){\sse$ \Hss $}
  \put(400.4,295)  {\sse$ \Hss $}
  \put(63.1,213.1) {\sse$ \Lambda$}
  \put(108.8,169.7){\sse$ \Lambda$}
  \put(182.1,107.1){\sse$ \Lambda$}
  } }
We now observe that in \erf{Hpastloops_2} there appear two copies of the
$(p{-}1)$-fold coproduct of $H$. Using that these are bimodule morphisms from
$H$ to $H^{\otimes p}_{}$, we conclude that \erf{Hpastloops_2}, and thus the
left hand side of \erf{Hpastloops}, equals the right hand side of
\erf{Hpastloops}.
\end{proof}


\section{Proof of the main theorem}\label{sec:proofmain}

We are now in a position to establish invariance
of the correlators $\Cor gn$ under the action of the mapping class group \Map gn.
In the subsequent lemmas we treat separately the
various types of generators from the presentation given in Section \ref{ssec:mpg}.

We start with the generators affording S- and T-transformations of the
one-holed torus.

\begin{lemma}\label{lem:101S,101T}
We have
  \be
  \Corr 110 \circ \SK = \Corr 110  \qquand \Corr 110 \circ \TK = \Corr 110
  \ee
for $\SK$ and $\TK$ as given in {\rm \erf{S_KH-TKH}}.
\end{lemma}

\begin{proof}
This has already been shown in Lemmas 5.8 and 5.9 of \cite{fuSs3}, but we
find it instructive to repeat the main steps here.
\\[2pt]
(i)\, Composing $\SK^{-1}$ as given in \erf{S_KHinv} with the first expression
for $\Corr 110$ in \erf{PF_1} results in the first equality in
  \Eqpic{Sinv-6} {380} {65} {
  \put(-31,65)  {$ \Corr 110\circ\SK^{-1} ~= $}
     \put(78,0) {\Includepichtft{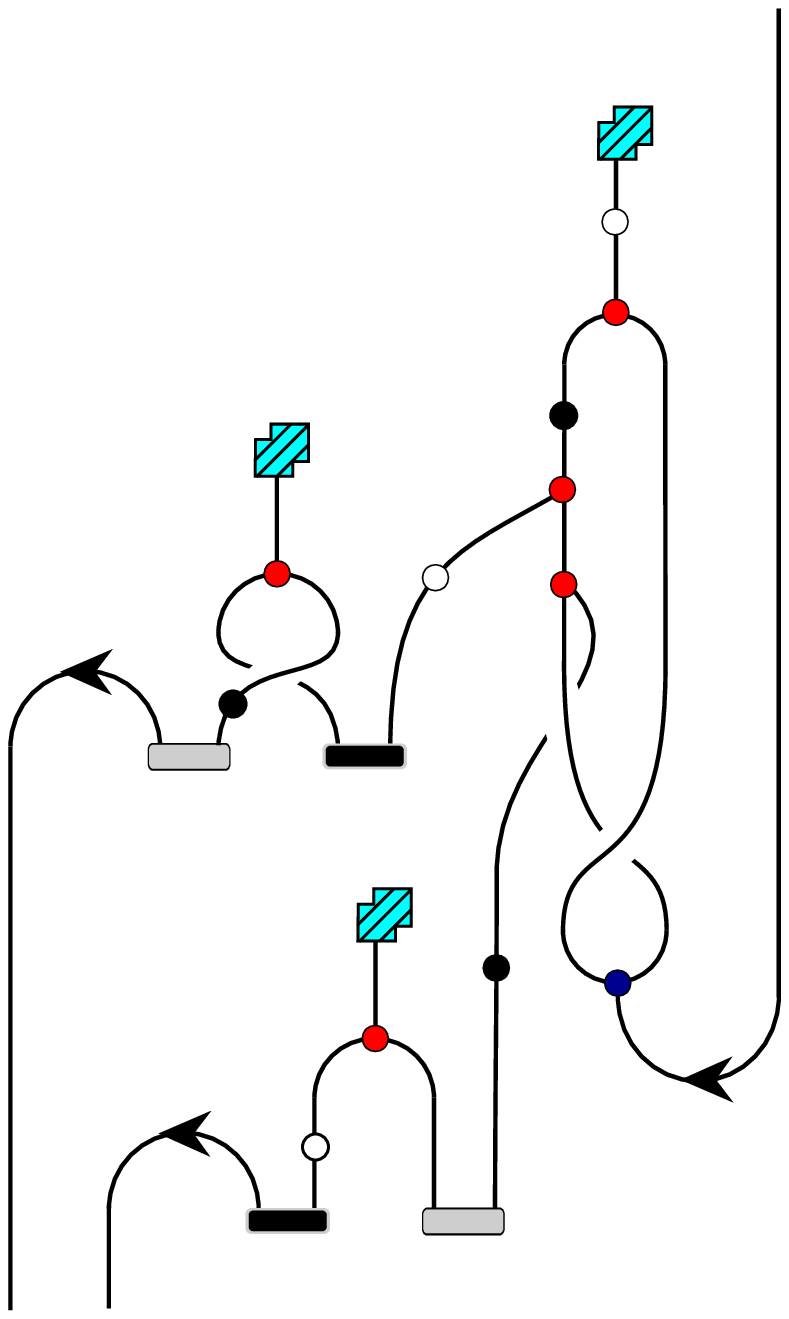}
  \put(-5.4,-9.2) {\sse$ \Hss $}
  \put(7,-9.2)    {\sse$ \Hss $}
  \put(16.3,52.7) {\sse$ Q $}
  \put(31.9,53.1) {\sse$ Q^{-1} $}
  \put(23.4,1.9)  {\sse$ Q^{-1} $}
  \put(46.6,1.7)  {\sse$ Q $}
  \put(32.3,39)   {\sse$ \lambda $}
  \put(34.5,91)   {\sse$ \lambda $}
  \put(72,126)    {\sse$ \lambda $}
  \put(81,147)    {\sse$ \Hss $}
  }
  \put(189,65)  {$=$}
      \put(224,0) {\Includepichtft{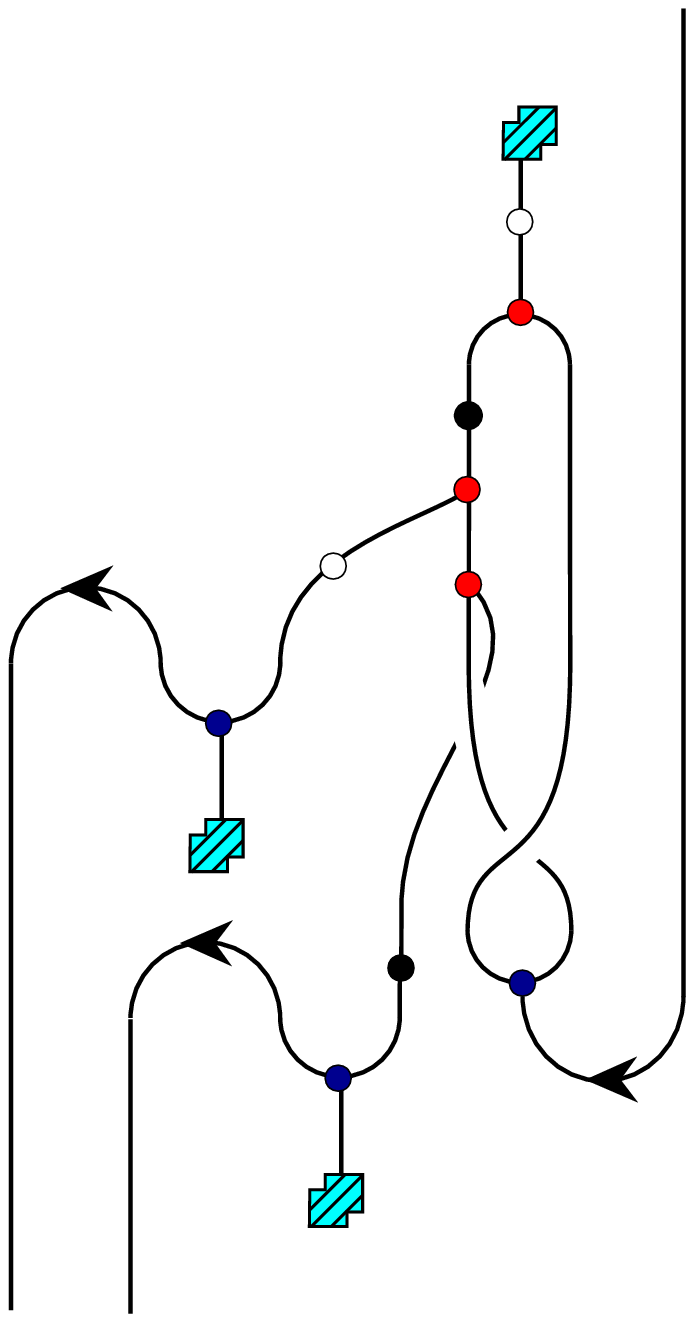}
  \put(-5,-9.2)   {\sse$ \Hss $}
  \put(9,-9.2)    {\sse$ \Hss $}
  \put(26.5,48)   {\sse$ \Lambda $}
  \put(40,9)      {\sse$ \Lambda $}
  \put(60.4,126)  {\sse$ \lambda $}
  \put(70.5,147)  {\sse$ \Hss $}
  }
  \put(329,65)  {$=$}
      \put(359,0) {\Includepichtft{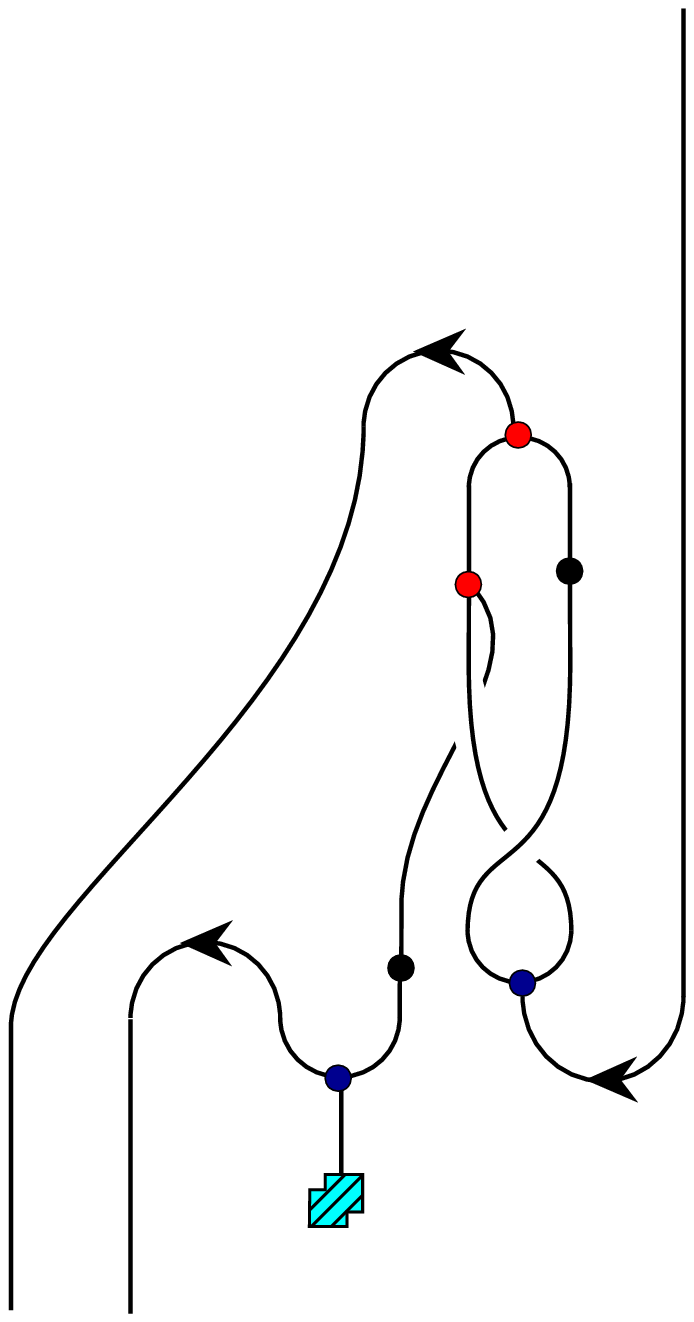}
  \put(-5,-9.2)   {\sse$ \Hss $}
  \put(9,-9.2)    {\sse$ \Hss $}
  \put(40,9)      {\sse$ \Lambda $}
  \put(70.5,147)  {\sse$ \Hss $}
  } }
Here in the second step we get rid of the two monodromy matrices by implementing
the relations \erf{fQS_Psi} between the integral, cointegral and monodromy
matrix, combined with the identity $\lambda \cir m \eq \lambda \cir m \cir \tauHH
\cir (\id_H\oti\apo^2)$, while the third follows by Lemma \ref{lem:Sinv-7}.
Finally, after pushing the `upper' inverse antipode through the coproduct, the
so obtained final expression in \erf{Sinv-6} can be recognized as the right-most
one in formula \erf{PF_1}. This shows $\Corr 110 \cir \SK^{-1} \eq \Corr 110$.
\\[4pt]
(ii)\, Composing $\TK$ as given in \erf{S_KH-TKH} with the last expression for
$\Corr 110$ in \erf{PF_1} yields
  \eqpic{T-inv} {200} {47} {
  \put(11,48)     {$ \Corr 110 \circ \TK ~= $}
  \put(118,0) {\Includepichtft{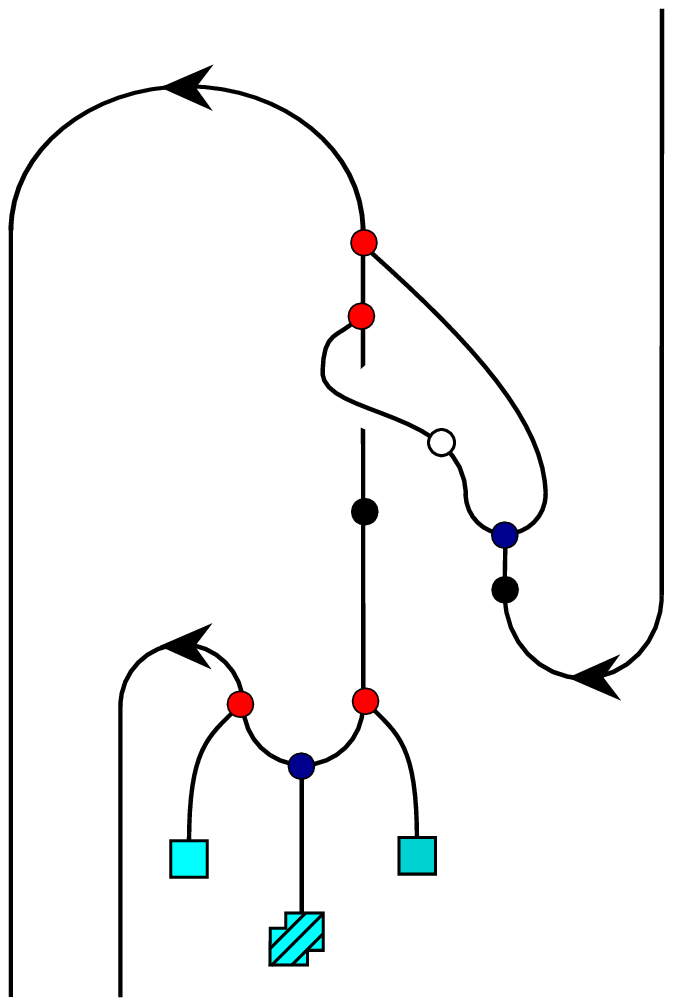}
  \put(-5,-9.2)   {\sse$ \Hss $}
  \put(8.5,-9.2)  {\sse$ \Hss $}
  \put(23.5,13.5) {\sse$ v $}
  \put(49,13.5)   {\sse$ v^{-1} $}
  \put(36,2.5)    {\sse$ \Lambda $}
  \put(69.2,112)  {\sse$ \Hss $}
  } }
After applying \eqref{Hopf_Frob_trick2} and using that $\apo \cir v \eq v$,
the central elements $v$ and $v^{-1}$
appearing here cancel out. Thus we arrive at $\Corr 110 \cir \TK \eq \Corr 110$.
\end{proof}

In view of the relation \erf{gpqfromg11} between $\Corr g11$ and
$\Corr gpq$ for arbitrary number of incoming and outgoing insertions, Lemma
\ref{lem:101S,101T} immediately generalizes as follows.

\begin{prop}\label{prop:triple}
For any triple of integers $g\,{\ge}\,1$ and $p,q\,{\ge}\,0$ we have
  \be
  \bearl
  \Corr gpq \circ
  \big( \id_K^{\otimes m} \oti \SK^{} \oti\id_K^{\otimes g-m-1} \oti \id_F^{\otimes q} \big)
  = \Corr gpq  \qquad{\rm and}
  \\{}\\[-5pt]
  \Corr gpq \circ
  \big( \id_K^{\otimes m} \oti \TK^{} \oti\id_K^{\otimes g-m-1} \oti \id_F^{\otimes q} \big)
  = \Corr gpq
  \eear
  \ee
for all $m \iN \{0,1,...\,,g{-}1 \}$.
\end{prop}

Next we show that the correlators do not change when applying a twist to
any tensor power of the handle Hopf algebra $K$:

\begin{lemma}\label{lem:inv_twist}
For any integer $g\,{\ge}\,1$ we have
  \be
  \Corr g11 \circ \big( \id_K^{\otimes r} \oti \theta^{}_{K^{\otimes s}_{}} \oti
  \id_K^{\otimes t} \oti \id_F \big) = \Corr g11
  \labl{0.40}
for all triples $r,s,t$ of non-negative integers with $r\,{+}\,s\,{+}\,t \eq g$.
\end{lemma}

\begin{proof}
Consider first the case $s \eq g$. In view of the formula \erf{deftwist} for
the twist in \HBimod, it follows immediately from Lemma \ref{lem:Hpastloops}
that the left hand side of \erf{0.40} differs from the right hand side only by
multiplications from the left with $\apo \cir v$ and from the right with $\apoi
\cir v^{-1}$, both acting on the same $H$-line. Since $v$ is left invariant
by the antipode and is central, these two modifications cancel each other.
\\
The argument for $s \,{<}\, g$ is completely analogous.
\end{proof}

The next observations will allow us to establish invariance under
the action of Dehn twists around the cycles $a_k$.

\begin{lemma}
We have
  \be
  \Corr 211 \circ (\QQ \oti \id_F) = \Corr 211
  \labl{QQ_inv}
with $\QQ$ as introduced in {\rm \erf{QHH}}.
\end{lemma}

\begin{proof}
According to \eqref{CorrgnH} and \eqref{QQ_Haa} we have
  \eqpic{Pic_QQ1} {320} {71} {\setulen80
  \put(-53,88)     {$ \Corr 211 \circ (\QQ \oti\id_F) ~\equiv $}
    \put(130,0) {\includepichtft{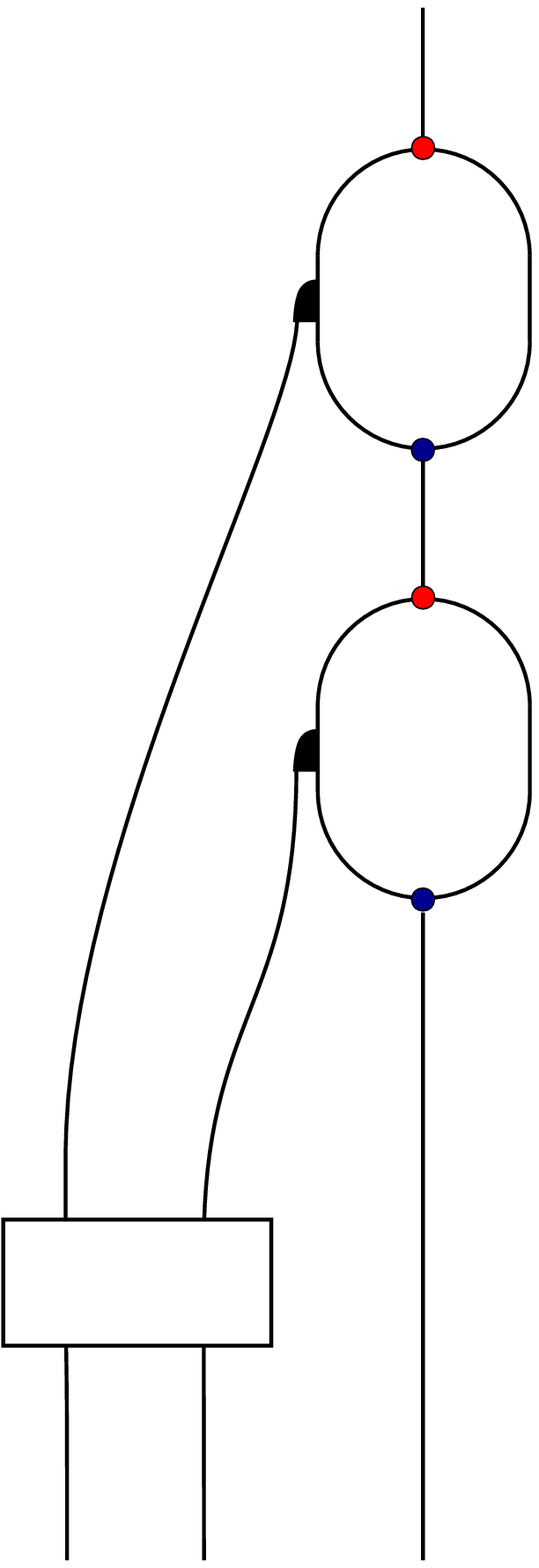}
  \put(2,-11.2)    {\sse$ K $}
  \put(8.8,33)     {\sse$ \QQ $}
  \put(20,-11.2)   {\sse$ K $}
  \put(47.6,-11.2) {\sse$ F $}
  \put(48.8,197)   {\sse$ F $}
  \put(-7,130)     {\begin{turn}{90}{\catpicH}\end{turn}}
  }
  \put(226,88)    {$=$}
    \put(272,0) {\includepichtft{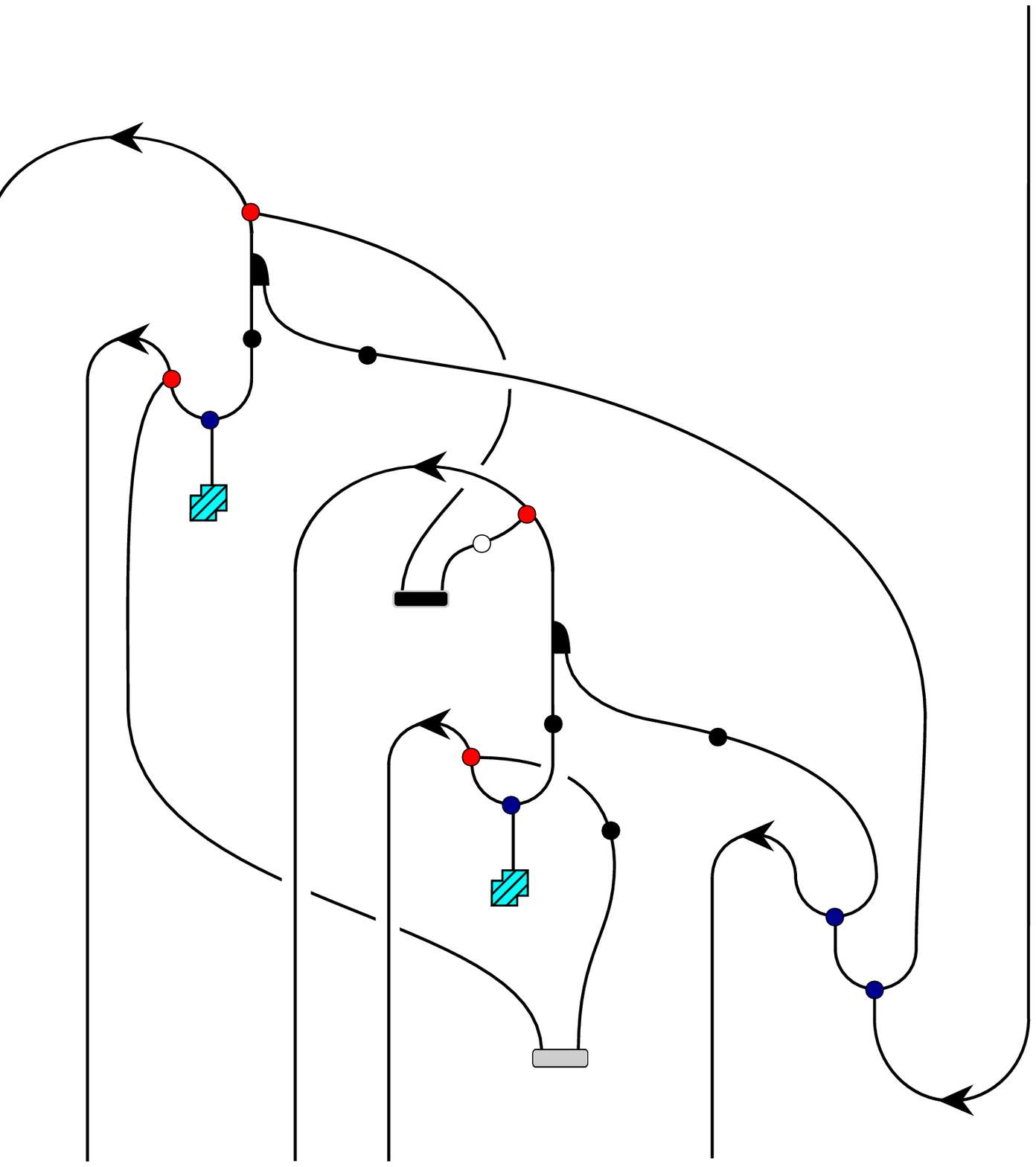}
  \put(-8.8,-11.2) {\sse$ \Hss $}
  \put(11,-11.2)   {\sse$ \Hss $}
  \put(43.7,-11.2) {\sse$ \Hss $}
  \put(60,-11.2)   {\sse$ \Hss $}
  \put(111.2,-11.2){\sse$ \Hss $}
  \put(39.5,104.1) {\sse$ \Lambda $}
  \put(88.5,41.1)  {\sse$ \Lambda $}
  \put(88,7.5)     {\sse$ Q$}
  \put(62,81.6)    {\sse$ Q^{-1}$}
  \put(96.4,82.9)  {\sse$\ohrad$}
  \put(46.7,143.8) {\sse$\ohrad$}
  \put(164,193)    {\sse$ \Hss $}
  } }
Using the identities \erf{Hopf_Frob_trick2} and inserting the explicit form of
the right coadjoint action $\ohrad$, this can be rewritten as
  \Eqpic{Pic_QQ2} {154} {81} { \put(0,5) { \setulen90
  \put(-166,88)   {$ \Corr 211 \circ (\QQ \oti\id_F) ~= $}
    \put(-5,0) {\INcludepichtft{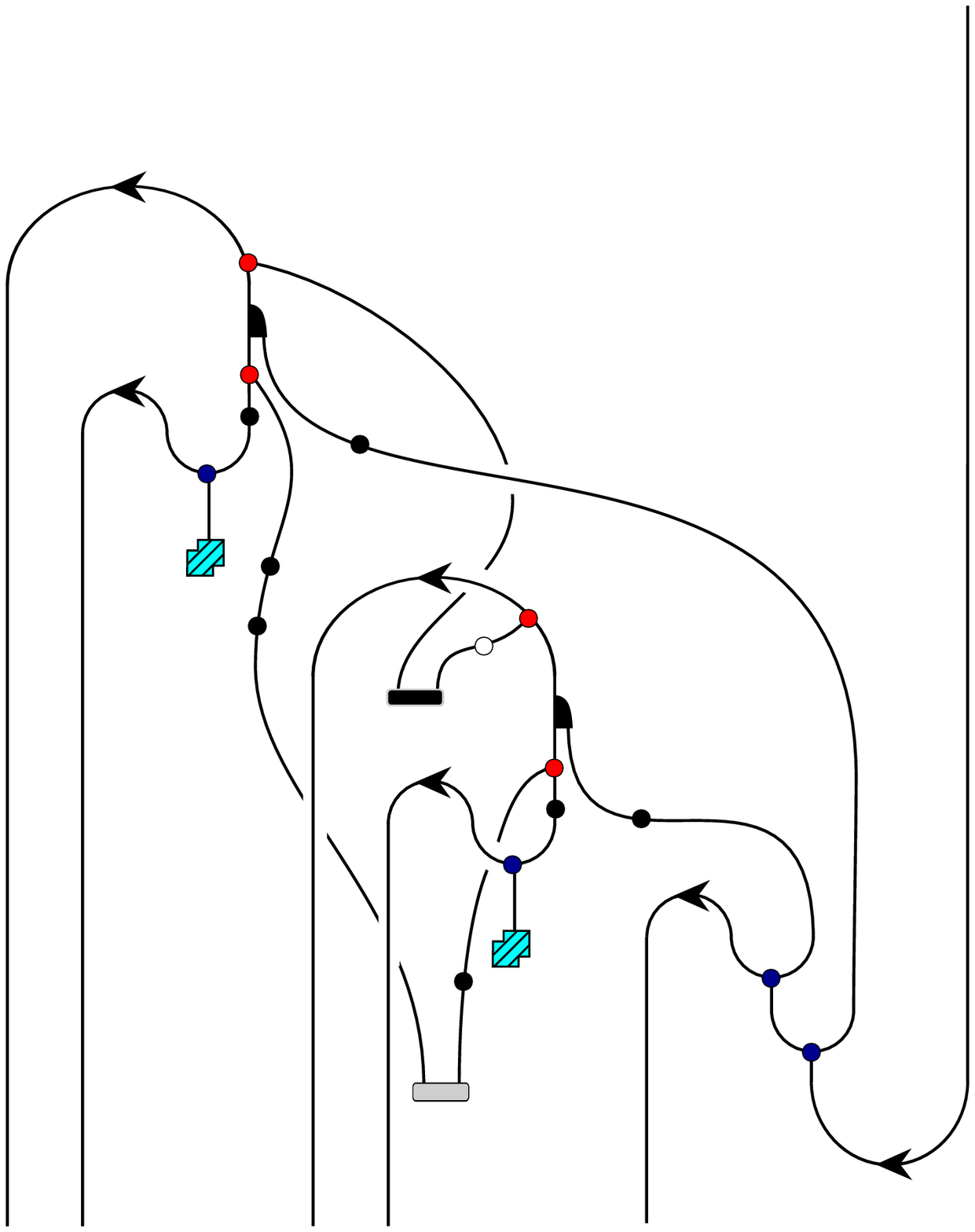}{342}
  \put(-4,-10.2)   {\sse$ \Hss $}
  \put(8,-10.2)    {\sse$ \Hss $}
  \put(44.5,-10.2) {\sse$ \Hss $}
  \put(56,-10.2)   {\sse$ \Hss $}
  \put(98.2,-10.2) {\sse$ \Hss $}
  \put(21,102.5)   {\sse$ \Lambda $}
  \put(85,41.5)    {\sse$ \Lambda $}
  \put(74,14.5)    {\sse$ Q $}
  \put(57,76.7)    {\sse$ Q^{-1} $}
  \put(92.7,80.9)  {\sse$ \ohrad $}
  \put(42.7,143.5) {\sse$ \ohrad $}
  \put(151,199)    {\sse$ \Hss $}
  }
  \put(175,88)     {$=$}
    \put(212,0) {\INcludepichtft{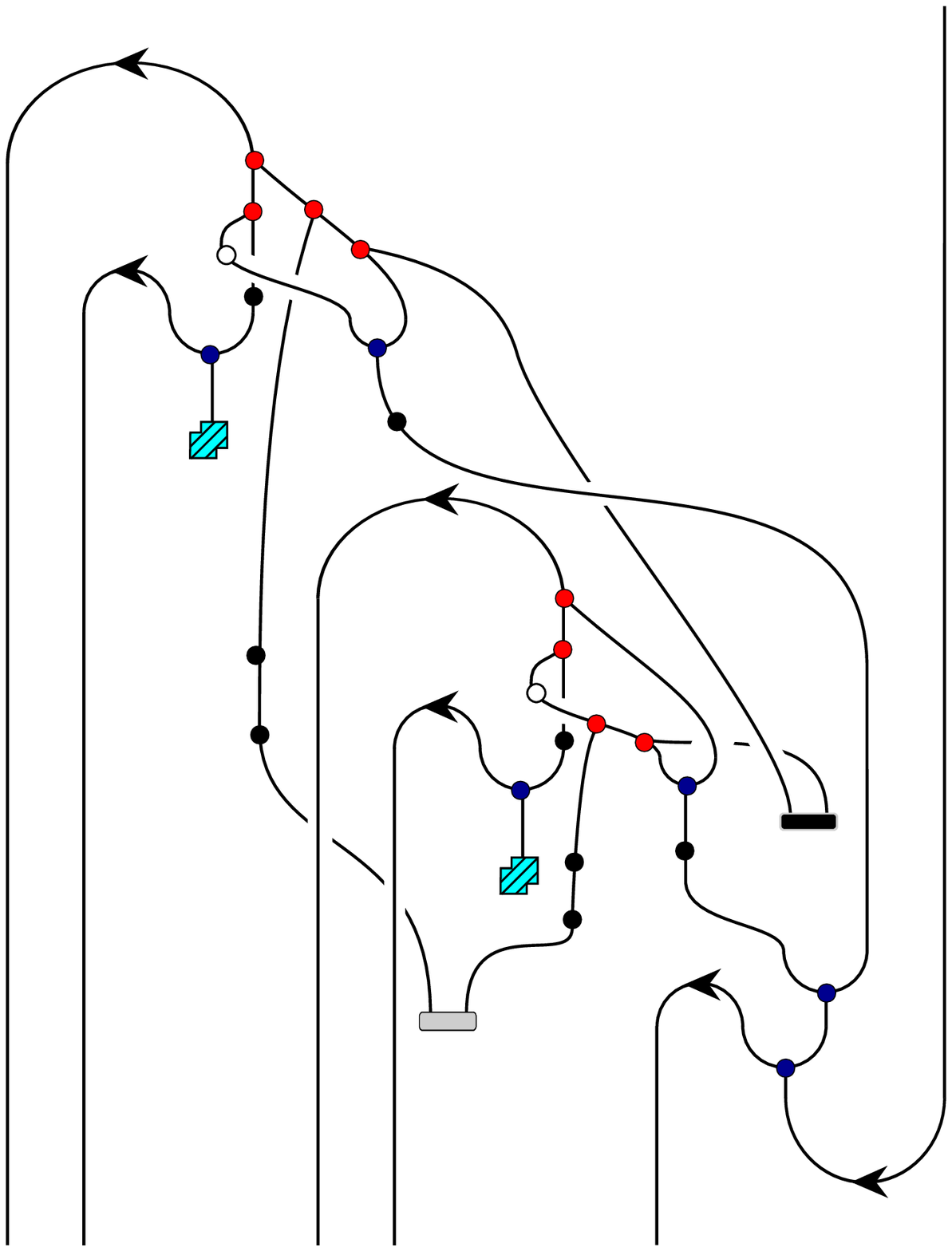}{342}
  \put(-4,-10.2)   {\sse$ \Hss $}
  \put(8,-10.2)    {\sse$ \Hss $}
  \put(44.5,-10.2) {\sse$ \Hss $}
  \put(56,-10.2)   {\sse$ \Hss $}
  \put(98.2,-10.2) {\sse$ \Hss $}
  \put(20,123.3)   {\sse$ \Lambda $}
  \put(68.8,54.5)  {\sse$ \Lambda $}
  \put(65.7,26.5)  {\sse$ Q $}
  \put(116.5,57.4) {\sse$ Q^{-1} $}
  \put(144,199)    {\sse$ \Hss $}
  } } }
Observing that $(\apo^2\oti\apo^2)\cir Q \eq Q$ and invoking the identity
\erf{Pic_QQ3}, comparison with \eqref{CorrgnH} establishes \erf{QQ_inv}.
\end{proof}

Again this result easily generalizes:

\begin{prop}\label{prop:tripleQQ}
For any triple of integers $g\,{\ge}\,2$ and $p,q\,{\ge}\,0$ we have
  \be
  \Corr gpq \circ \big( \id_K^{\otimes m} \oti
  \QQ \oti\id_K^{\otimes g-m-2} \oti \id_F^{\otimes q} \big)
  = \Corr gpq
  \ee
for all $m \iN \{0,1,...\,,g{-}2 \}$.
\end{prop}

Now we present relations that will help us to show invariance under
the action of the generators $t_{j,k}$ of the mapping class group.

\begin{lemma}\label{lem:removeQq_t}
For any integer $m \,{\ge2}\,$ we have
  \be
  (\id_F\oti\tilde b_F)\circ(\Corr m20\oti\id_{^\vee\!F})
  \circ \QB_{K^{\otimes m-1}\otimes{}^\vee\!F}
  = (\id_F\oti\tilde b_F)\circ(\Corr m20\oti\id_{^\vee\!F})
  \labl{eq:removeQq_t}
in $\HomHH(K^{\otimes m}\oti{}^\vee\!F,F)$. That is, graphically:
  \eqpic{removeQq_t} {290} {81} {
    \put(0,0) {
    \put(5,0)   {\includepichtft{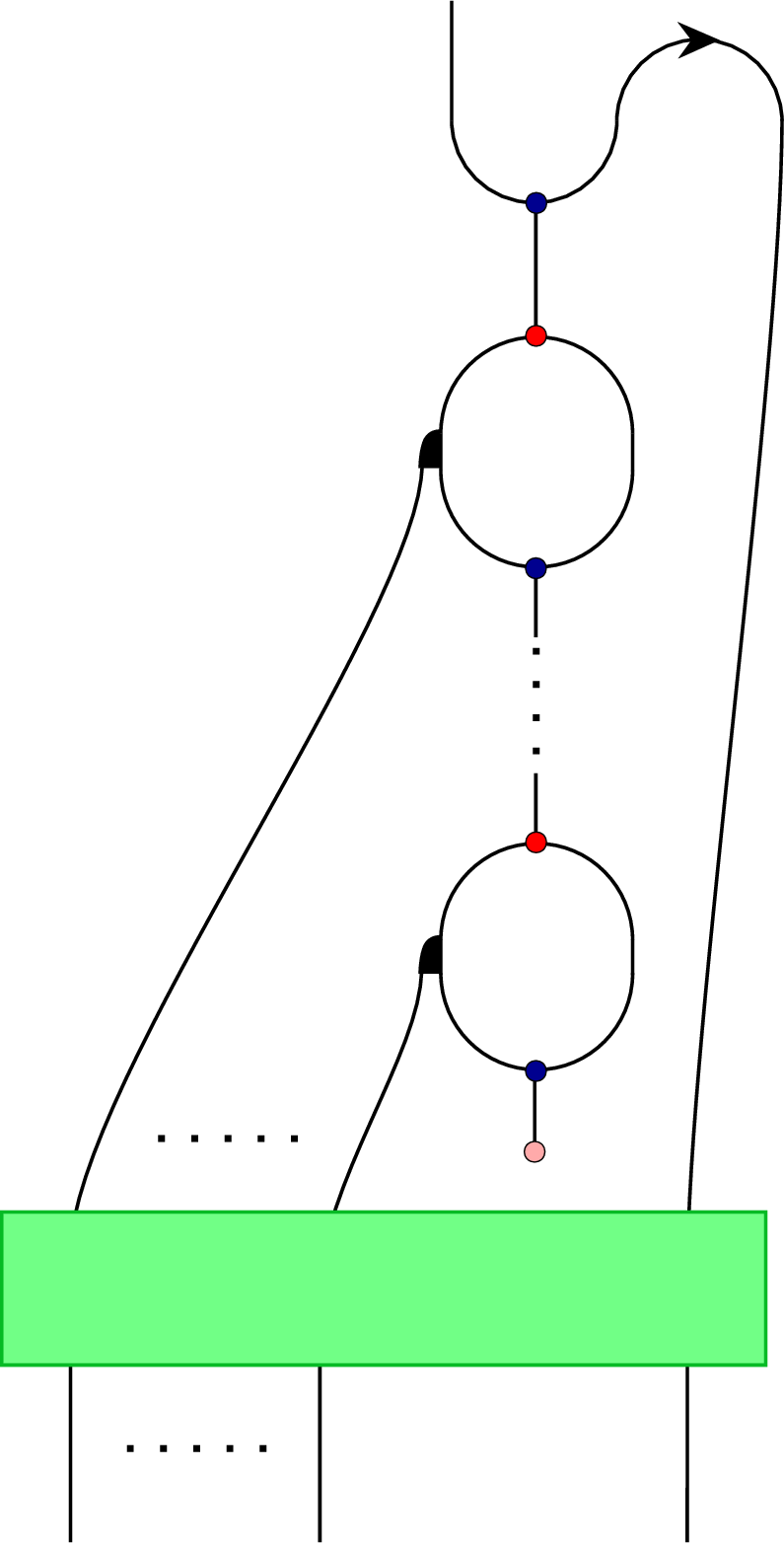}}
  \put(23.3,26.3)  {\sse$ \QB_{\HK^{\otimes m-1}_{}\otimes{}^\vee\!F} $}
  \put(8,-9.2)     {\sse$ K $}
  \put(21,-9.2)    {\sse$ \cdots $}
  \put(36,-9.2)    {\sse$ K $}
  \put(75,-9.2)    {\sse$ ^\vee\!F $}
  \put(52.6,174)   {\sse$ F $}
  }
  \put(126,82)     {$=$}
  \put(150,0) {
    \put(13,0)  {\includepichtft{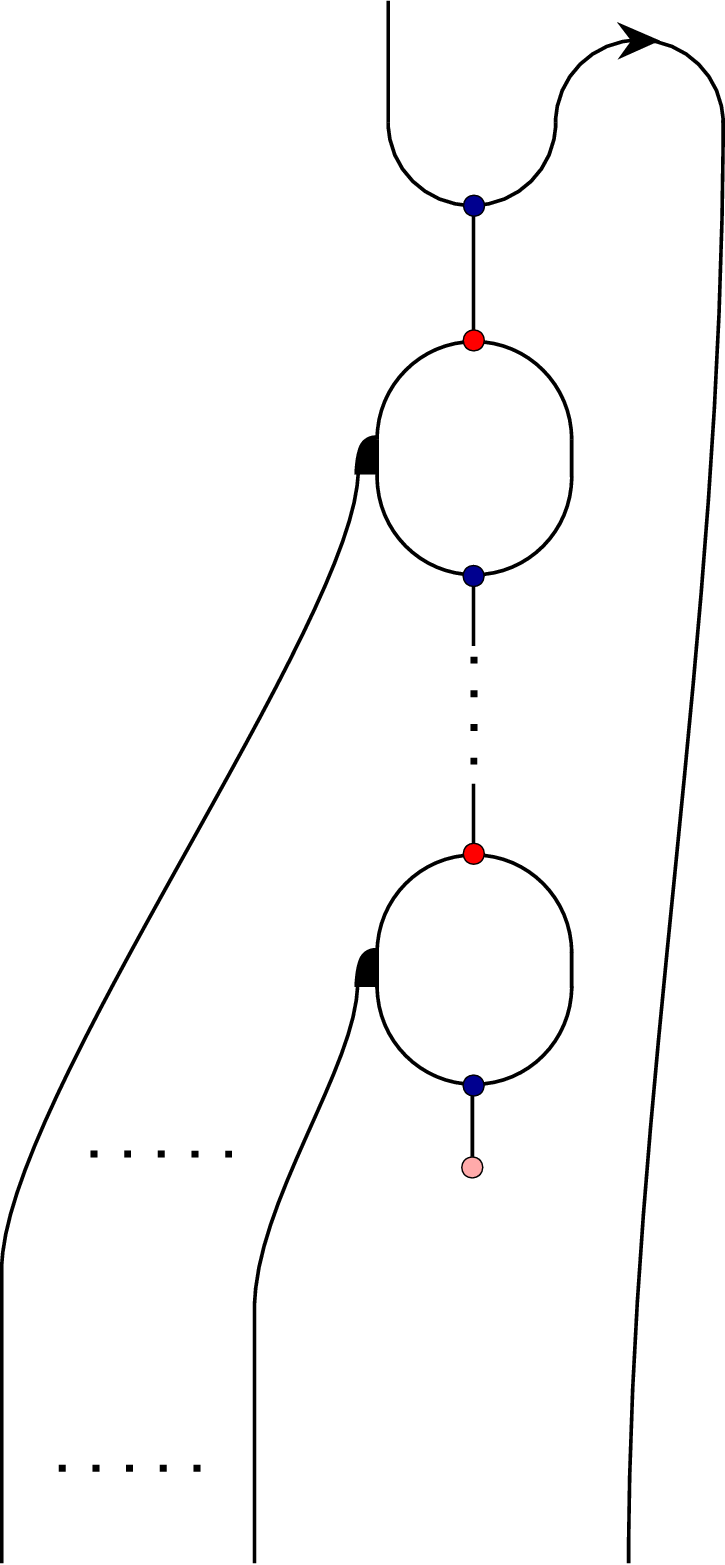}}
  \put(8,-9.2)     {\sse$ K $}
  \put(21,-9.2)    {\sse$ \cdots $}
  \put(36.4,-9.2)  {\sse$ K $}
  \put(77,-9.2)    {\sse$ ^\vee\!F $}
  \put(52.7,174)   {\sse$ F $}
  }
  \put(265,154)    {\catpicH}
  }
\end{lemma}

\begin{proof}
Denote the left hand side of \erf{eq:removeQq_t} by $\Phi$. We first
insert the structural morphisms for $F$ from \erf{pic-Hb-Frobalgebra} and
the expressions \erf{CorrgnH} for $\Corr m11$ and \erf{Qq_X} for $\QB_X$ and
invoke Lemma \ref{lem:Hpastloops}, and in a second step we insert the
explicit expression for $\ohrad$ and use the identities \erf{Hopf_Frob_trick2}
and \erf{Fv_act}. This yields
  \Eqpic{inv_tjk_1} {420} {116} { \put(0,17){
  \put(-15,94)     {$ \Phi ~= $}
    \put(37,0) {\includepichtft{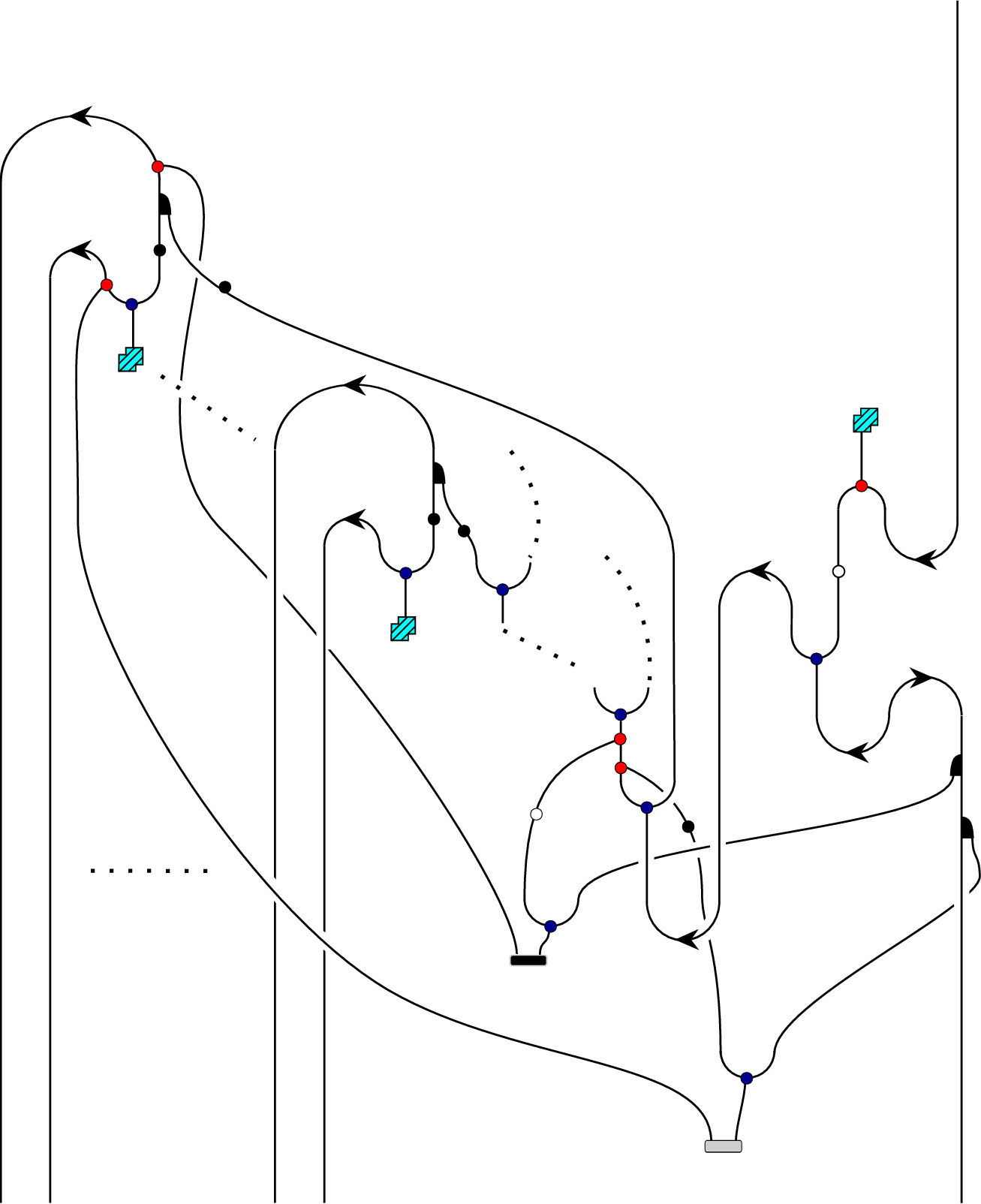}\setulen80
  \put(-7,-11.2)   {\sse$ \Hss $}
  \put(8.2,-11.2)  {\sse$ \Hss $}
  \put(60.5,-11.2) {\sse$ \Hss $}
  \put(74.5,-11.2) {\sse$ \Hss $}
  \put(28.2,-11.2) {\sse$ \cdots\cdots$}
  \put(229.5,-11.2){\sse$ H $}
  \put(229,297)    {\sse$ \Hss $}
  \put(30,195.5)   {\sse$ \Lambda$}
  \put(100,132)    {\sse$ \Lambda$}
  \put(216,188)    {\sse$ \lambda$}
  \put(173,3.5)    {\sse$ Q$}
  \put(120,49,2)   {\sse$ Q^{-1}$}
  }
  \put(253,94)     {$=$}
    \put(293,0) {\includepichtft{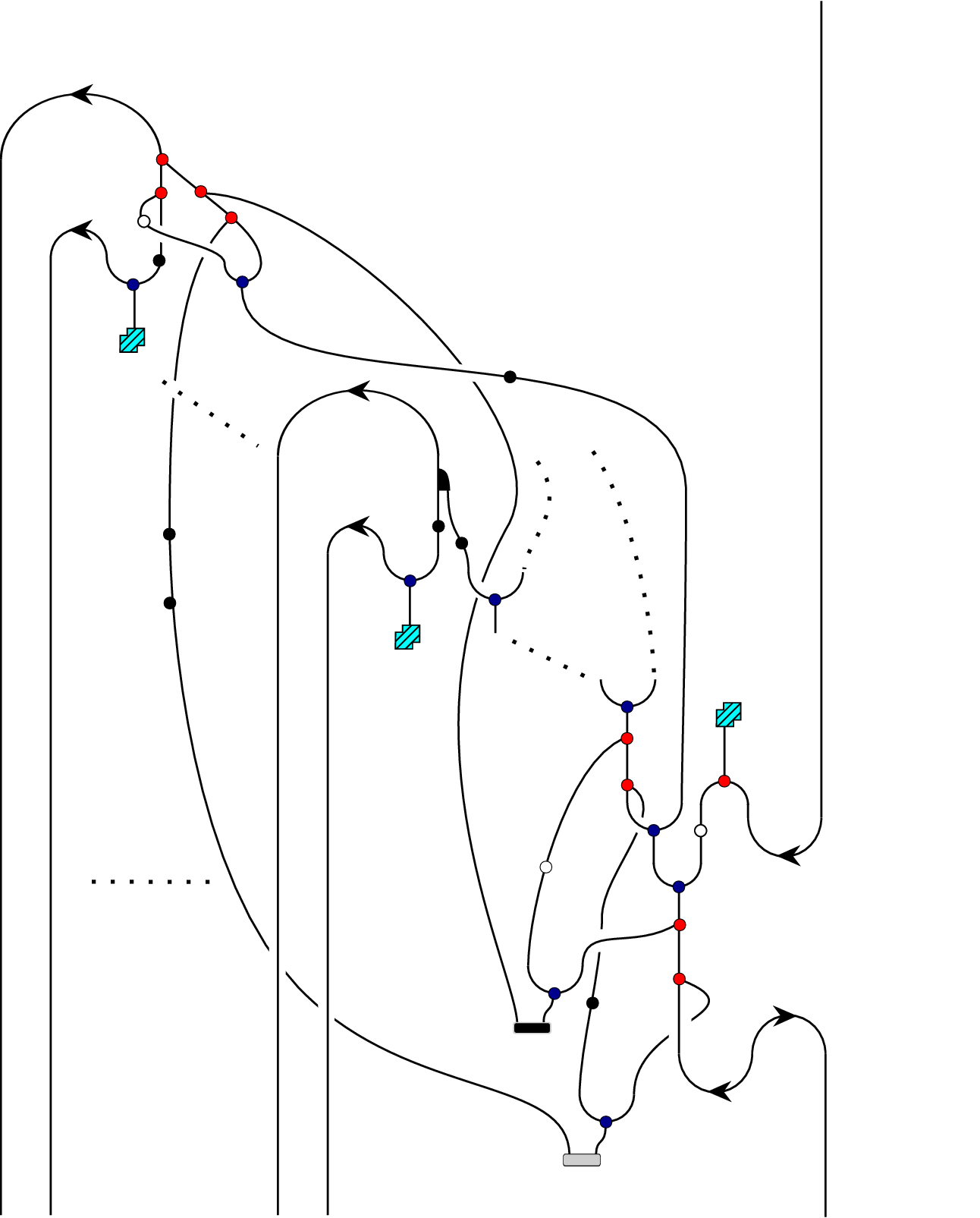}\setulen80
  \put(-7,-11.2)   {\sse$ \Hss $}
  \put(8.2,-11.2)  {\sse$ \Hss $}
  \put(60.5,-11.2) {\sse$ \Hss $}
  \put(74.5,-11.2) {\sse$ \Hss $}
  \put(28.2,-11.2) {\sse$ \cdots\cdots$}
  \put(194.5,-11.2){\sse$ H $}
  \put(193,297)    {\sse$ \Hss $}
  \put(19.8,207.5) {\sse$ \Lambda$}
  \put(100,132)    {\sse$ \Lambda$}
  \put(179.5,117.2){\sse$ \lambda$}
  \put(136.4,4)    {\sse$ Q$}
  \put(118.1,36.1) {\sse$ Q^{-1}$}
  } } }
Application of the identity \erf{Q3Delta} to the right hand side of
\erf{inv_tjk_1} results in the right hand side of \erf{eq:removeQq_t},
\end{proof}

Composition of \erf{removeQq_t} with $\id^{}_{K^{\otimes m}_{}} \oti
{}^\vee\!\eps_F$ gives

\begin{cor}\label{cor:BKKm-1}
For any integer $g \,{\ge2}\,$ we have
  \be
  \Corr g10 \circ \QB_{K^{\otimes g-1}_{}} = \Corr g10 \,.
  \ee
\end{cor}

\medskip

We can combine the previous results to omit not only the twist of any tensor
power $K^{\otimes m}$ of $K$, but also the one of
$K^{\otimes m}\oti{}^{\vee\!}\!F$:

\begin{lemma}\label{lem:twist_S_remove}
For any integer $m\,{>}\,0$ the equalities
  \eqpic{twist_S_remove} {420} {100} {\setulen80
    \put(0,0) {\includepichtft{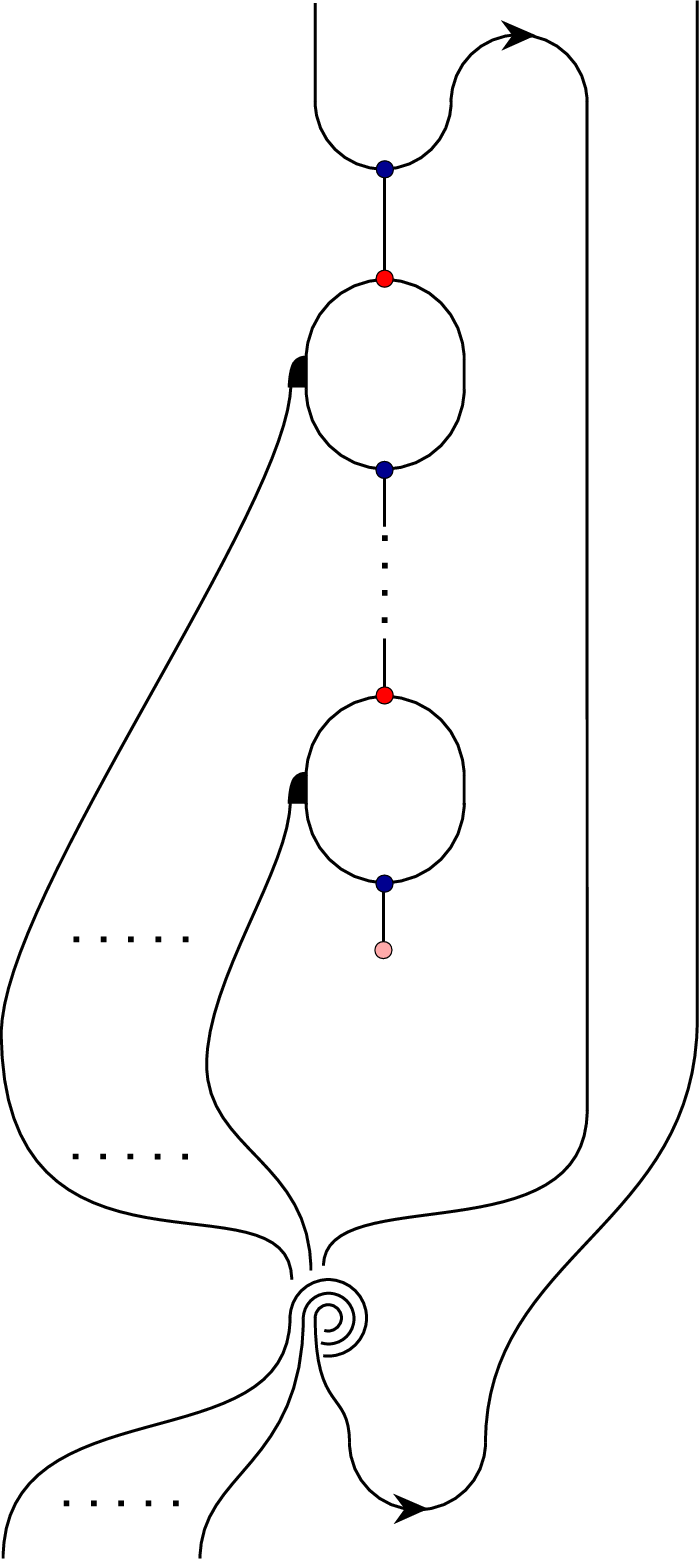}
  \put(-5,-11.2)   {\sse$ K $}
  \put(11,-11.2)   {\sse$ \cdots$}
  \put(28,-11.2)   {\sse$ K $}
  \put(50,265)     {\sse$ F $}
  \put(114,265)    {\sse$ F $}
  \put(1.5,39)     {\sse$ \theta_{\!K^{\otimes m}_{}\otimes{}^{\vee\!}\!F} $}
  }
  \put(153,130)    {$=$}
    \put(180,0) {\includepichtft{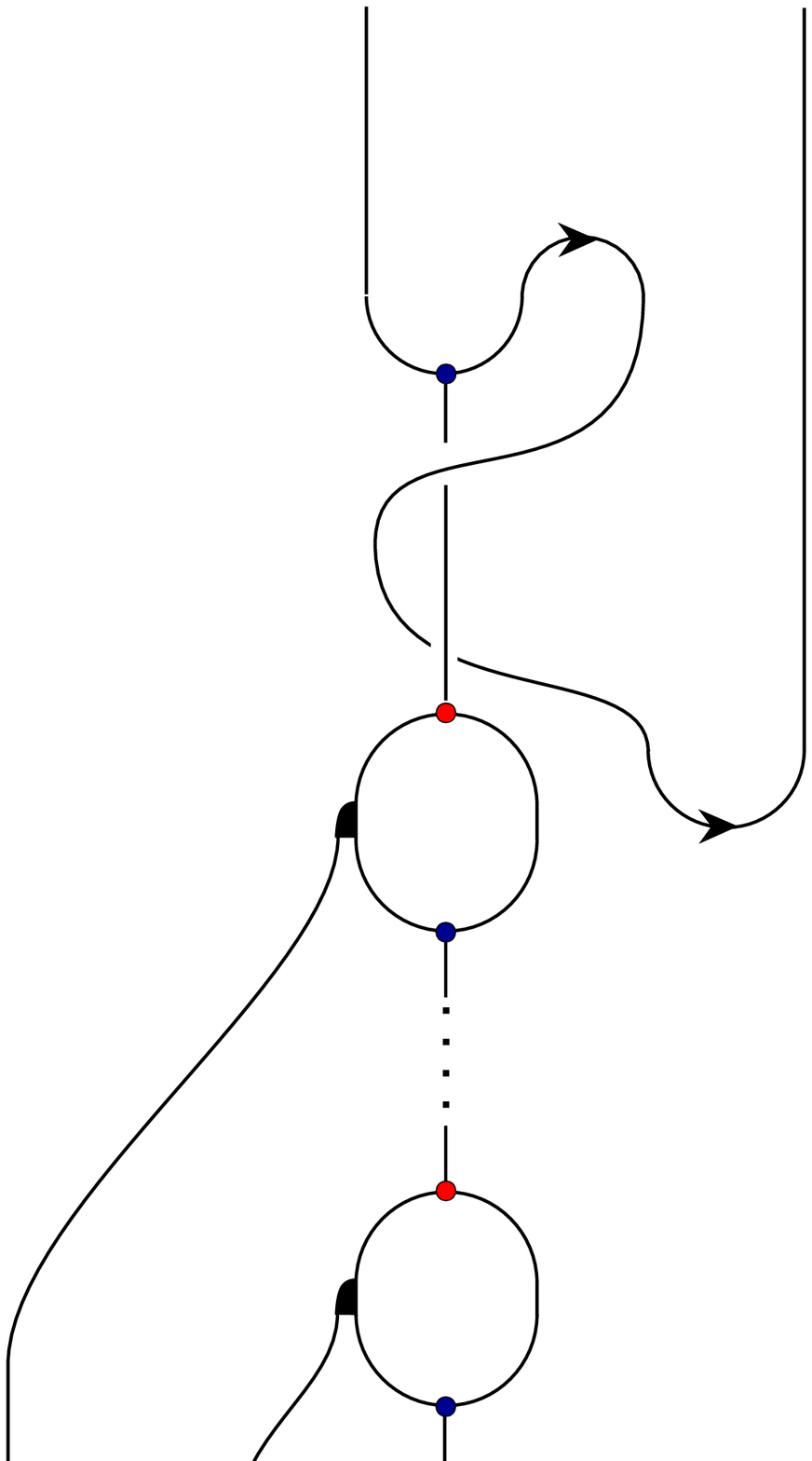}
  \put(-5,-11.2)   {\sse$ K $}
  \put(11,-11.2)   {\sse$ \cdots$}
  \put(28,-11.2)   {\sse$ K $}
  \put(50,265)     {\sse$ F $}
  \put(114,265)    {\sse$ F $}
  }
  \put(326,130)    {$=$}
    \put(357,0) {\includepichtft{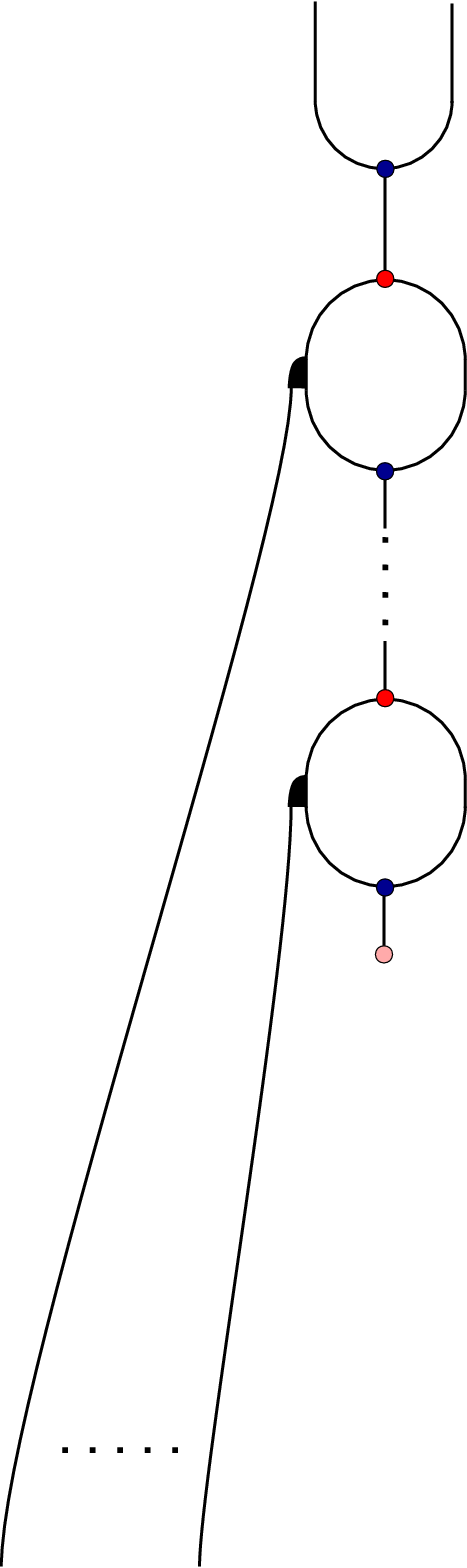}
  \put(-5,-11.2)   {\sse$ K $}
  \put(11,-11.2)   {\sse$ \cdots$}
  \put(28,-11.2)   {\sse$ K $}
  \put(50,265)     {\sse$ F $}
  \put(73,265)     {\sse$ F $}
  \put(114,233)    {\catpicH}
  } }
hold in $\HomHH(K^{\otimes m}_{},F\oti F)$.
\end{lemma}

\begin{proof}
Using the compatibility between braiding and twist, as well as that $F$ has
trivial twist and that according to Lemma \ref{lem:inv_twist} the twist of
$K^{\otimes m}_{}$ can be omitted, we can replace $\theta_{\!K^{\otimes m}_{}
\otimes{}^\vee\!F}$ by the monodromy between $K^{\otimes m}_{}$ and ${}^\vee\!F$.
Using naturality of the braiding and thus of monodromy, we can
push this monodromy through the $F$-loops and thus arrive
at the middle picture. The second equality follows because $F$ is commutative
Frobenius, in the same way as in the proof of Lemma \ref{lem:2prod_mon}.
\end{proof}

\begin{lemma}\label{lem:t_act_Cor}
For any pair of integers $g,n\,{>}\,0$ we have
  \eqpic{t_act_Cor} {380} {172} {
  \put(36,170)     {$ \Cor gn ~= $}
  \put(105,10) {\includepichtft{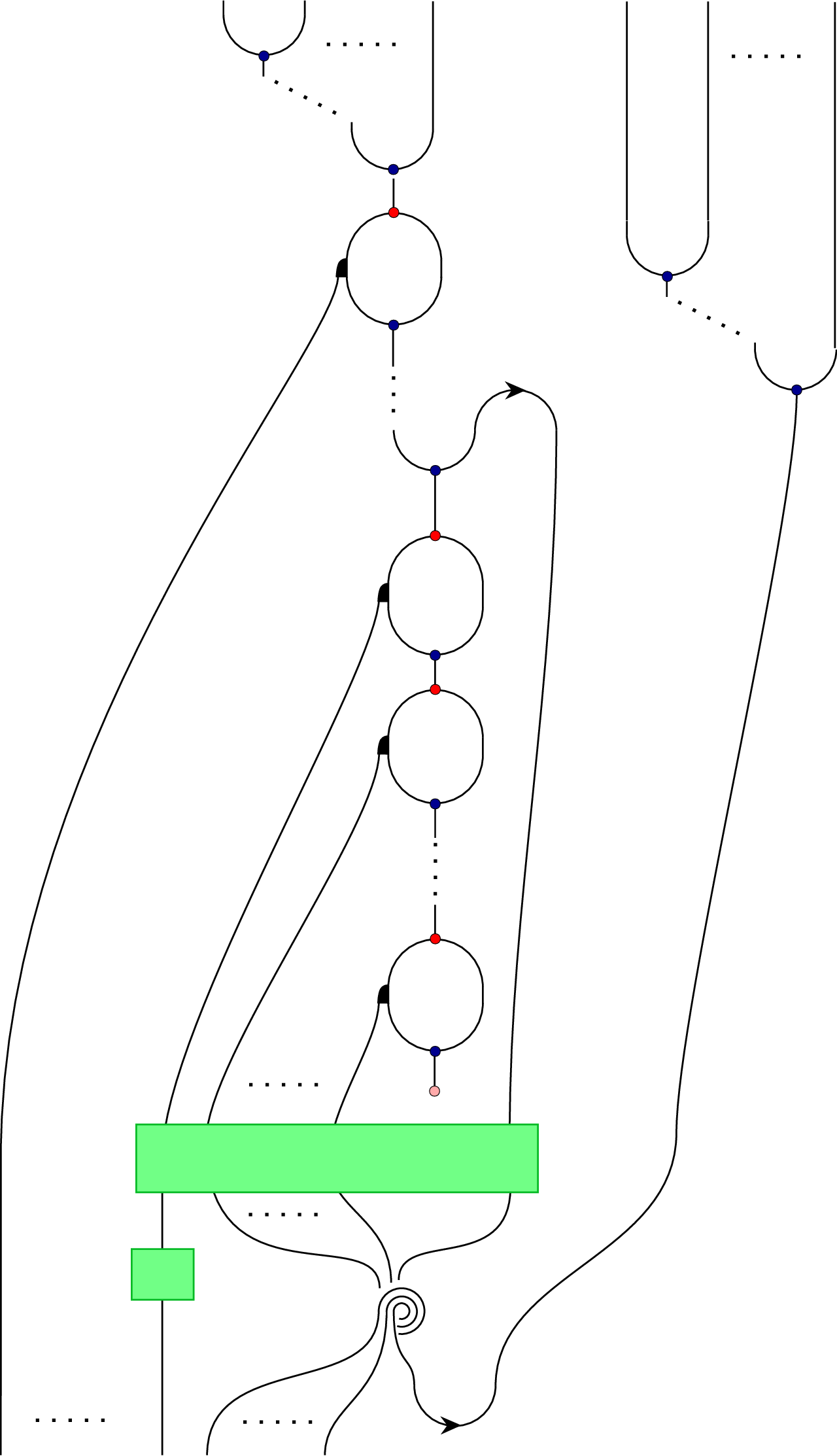}
  \put(49.3,339)   {\sse$ \overbrace{\hspace*{6em}}^{j~ \rm factors} $}
  \put(139.3,339)  {\sse$ \overbrace{\hspace*{6em}}^{n-j~ \rm factors} $}
  \put(51,64.7)    {\sse$ \QB_{K^{\otimes m-1}\otimes{}^\vee\!\BF}$}
  \put(95,30)      {\sse\begin{turn}{40} $ \theta_{\!K^{\otimes m-1}_{}
                    \otimes{}^{\vee\!}\!F} $\end{turn}}
  \put(31,38)      {\sse$\TK$}
  \put(-4,-8.5)    {\sse$ K $}
  \put(13,-8.5)    {\sse$\cdots$}
  \put(31,-8.5)    {\sse$ K $}
  \put(51,-8.5)    {\sse$ K $}
  \put(62,-8.5)    {\sse$\cdots$}
  \put(77,-8.5)    {\sse$ K $}
  \put(48,329)     {\sse$ F $}
  \put(66,329)     {\sse$ F $}
  \put(78,329)     {\sse$\cdots$}
  \put(95,329)     {\sse$ F $}
  \put(138,329)    {\sse$ F $}
  \put(156,329)    {\sse$ F $}
  \put(168,329)    {\sse$\cdots$}
  \put(185,329)    {\sse$ F $}
  \put(215,307)    {\catpicH}
  \put(-9.7,-12)   {\sse$ \underbrace{\hspace*{5em}}_{g-m+1~\rm factors~} $}
  \put(45,-12)     {\sse$ \underbrace{\hspace*{3.7em}}_{~~m-1~ \rm factors} $}
  } }
for all $m\eq 1,2,...\,,g$
\end{lemma}

\begin{proof}
As compared to the expression \erf{Sk_morph} for $\Cor gn$, on the right
hand side three additional pieces are present:
the endomorphism $\TK$ applied to one copy of $K$,
a twist endomorphism of $K^{\otimes m}_{}\oti{}^{\vee\!}\!F$, and a
partial monodromy $\QB_{K^{\otimes m-1}\otimes{}^\vee\!\BF}$. Now the latter
acts trivially by Lemma \ref{lem:removeQq_t}. After omitting this part, we
can invoke Proposition \ref{prop:triple} and Lemma \ref{lem:twist_S_remove}
(combined with the Frobenius property of $F$) to omit $\TK$ and
$\theta_{\!K^{\otimes m}_{}\otimes{}^{\vee}\!F}$, respectively, as well.
\end{proof}

As a final piece of information we give the graphical description of the
morphisms $\LAtjk(f)$ \erf{LyubactC3} that were introduced in formula
\erf{LyubactC3} for $f\iN \HomC(K^{\otimes g},X_1\oti\cdots\oti X_n)$:
  \eqpic{t_act} {290} {141} {
  \put(0,139)      {$ \LAtjk(f) ~= $}
    \put(81,0) {\Includepichtft{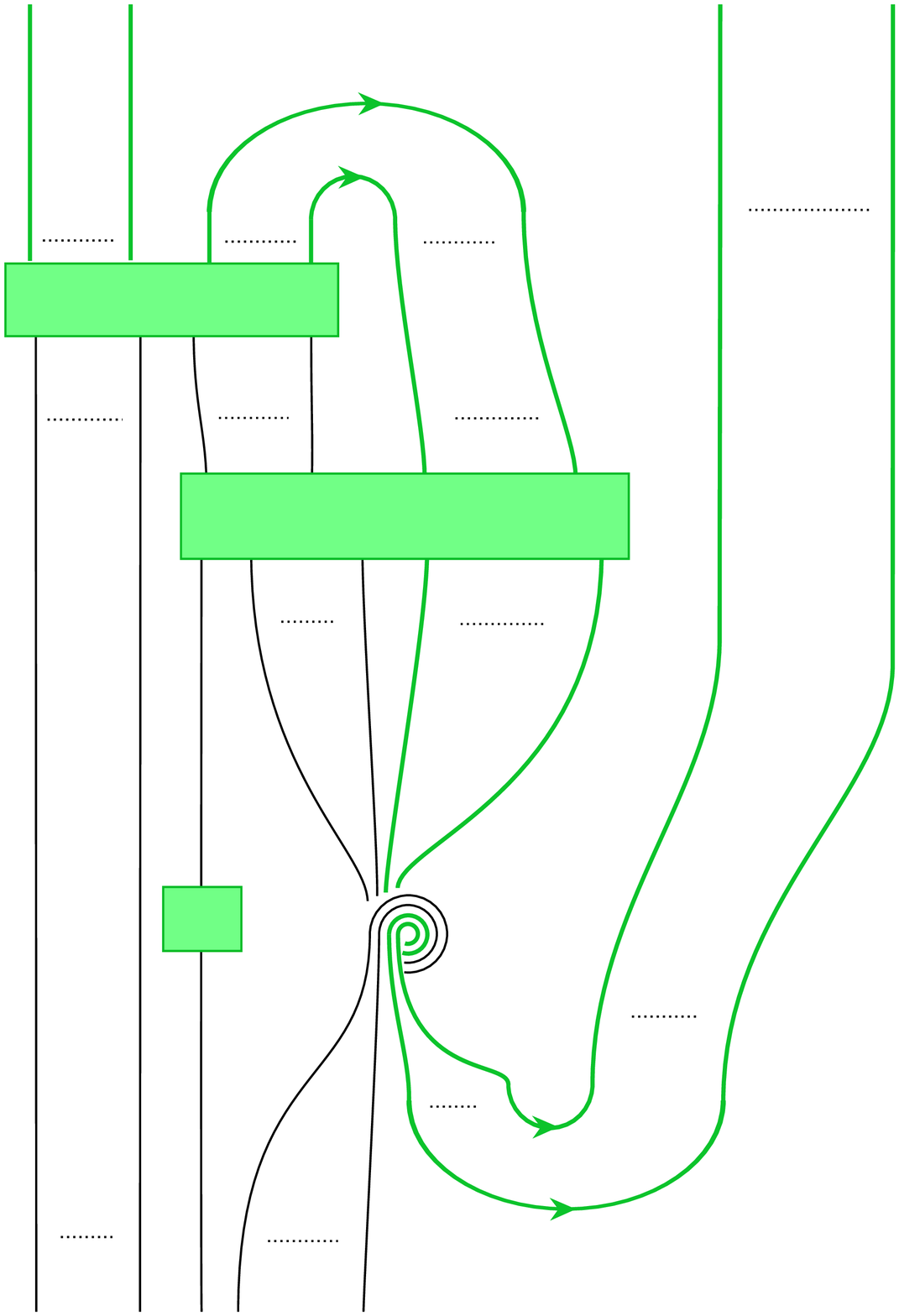}
  \put(33,222.7)   {$ f $}
  \put(42.3,176.2) {\sse$\QB_{K^{\otimes k-1}\otimes
                    {}^\vee\!X_{n}\otimes\cdots\otimes{}^\vee\!X_{j+1}}$}
  \put(38.5,85.5)  {\sse$\TK $}
  \put(1.2,-9.2)   {\sse$ K $}
  \put(12,-9.2)    {\sse$\cdots$}
  \put(25,-9.2)    {\sse$ K $}
  \put(37.8,-9.2)  {\sse$ K $}
  \put(46.9,-9.2)  {\sse$ K $}
  \put(59.8,-9.2)  {\sse$\cdots$}
  \put(75,-9.2)    {\sse$ K $}
  \put(-.2,296)    {\sse$ X_1$}
  \put(12,296)     {\sse$ \cdots$}
  \put(23.4,296)   {\sse$ X_j$}
  \put(147,296)    {\sse$ X_{j+1}$}
  \put(173,296)    {\sse$ \cdots$}
  \put(190,296)    {\sse$ X_{n}$}
  } }

\medskip

We have now all ingredients at our hands that are needed to finish the proof
of Theorem \ref{thm:main}.

\begin{proof}~\\[3pt]
Invoking the presentation of the mapping class group $\Mapgn$ described in
Section \ref{ssec:mpg}, invariance of the correlators $\Cor gn$ under the
action $\pi_{g:n}^{F^{\otimes n}}$ of $\Mapgn$ amounts to invariance under
$\pi_{g:n}^{F^{\otimes n}}(\gamma)$ for $\gamma\iN\Mapgn$ any of the
generators listed in Proposition \ref{Lyubact_prop}.
\\[4pt]
(i)\, $\gamma \eq \Ri$ ($i \eq 1,2,...\,,n$):
Invariance follows directly from the fact that $F$ is has trivial twist.
\\[4pt]
(ii)\, $\gamma \eq \omega_i$ ($i \eq 1,2,...\,,n{-}1$):
Invariance follows directly from the fact that $F$ is cocommutative.
\\[4pt]
(iii)\, $\gamma \eq S_k$ or $b_k$ or $d_k$ ($k\eq 1,2,...\,,g$):
In view of the explicit expressions \erf{LyubactC2} for $\LASk$, $\LAbk$
and $\LAdk$, invariance is implied by Proposition \ref{prop:triple}.
\\[4pt]
(iv)\, $\gamma \eq a_k$ ($k\eq 2,3,,...\,,g$):
By using Proposition \ref{prop:triple} twice we have
  \be
  \Cor gn \circ
  \big( \id_K^{\otimes m} \oti \TK^{} \oti \TK^{} \oti\id_K^{\otimes g-m-2} \big)
  = \Cor gn
  \ee
for all $m \eq 0,1,...\,,g{-}2$. Thus in view of the explicit expression
\erf{LyubactC2} for $\LAak$, Proposition \ref{prop:tripleQQ}
establishes the invariance.
\\[4pt]
(v)\, $\gamma \eq e_k$ ($k\eq 2,3,,...\,,g$):
Combining Proposition \ref{prop:triple} and  Lemma \ref{lem:inv_twist} we obtain
the equality $\Cor g1 \cir (\id_K^{\otimes g-k}\oti \TK \oti
\theta^{}_{K^{\otimes k-1}_{}}) \eq \Cor g1$
which, in turn, together with Corollary \ref{cor:BKKm-1} yields
  \be
  \Cor g1 \circ \big[ (\id_K^{\otimes g-k}\oti \TK \oti \theta_{K^{\otimes k-1}})
  \circ \QB_{K^{\otimes k-1}} \big] = \Cor g1 \,.
  \ee
This obviously generalizes to any number of $n \,{\ge}\, 0$ insertions and
thus in view of the explicit expression
\erf{LyubactC2} for $\LAek$ establishes invariance.
\\[4pt]
(vi)\, $\gamma \eq t_{j,k}$ ($j\eq 1,2,...\,,n{-}1$ and $k \eq 1,2,...\,,g$):
We first note that composing the dinatural family of the coend $K$ with the
partial monodromy $\QB_{}$ results in an ordinary monodromy. This implies that
  \be
  \QB_Y \circ (\id_K \oti f) = (\id_K \oti f) \circ \QB_X
  \ee
for any morphism $f\iN\Hom(X,Y)$. With the help of this relation (as well as
the coassociativity and Frobenius property of $F$) one can in particular push
partial monodromies through coproducts $\Delta_F$, and by functoriality of the
twist the same can be done with twist endomorphisms. As a consequence, the
morphism obtained by acting according to formula \erf{LyubactC3} and picture
\erf{t_act} with $t_{j,k}$ on $\Cor g n$ can be rewritten as the one on the
right hand side of \erf{t_act_Cor}. Thus invariance under $t_{j,k}$ reduces to
the assertion of Lemma \ref{lem:t_act_Cor}.
\end{proof}


\section{Invariants from ribbon automorphisms}\label{sec:omega}

We finally extend our result from $F$ to similar $H$-bimodules for which the
action of the Hopf algebra is twisted by a suitable automorphism.

\begin{defi}
A \emph{ribbon} automorphism of a ribbon Hopf algebra $H$ is a Hopf algebra
automorphism $\omega$ of $H$ that preserves the ribbon element and the
$R$-matrix of $H$,
  \be
  \omega\circ v= v \qquad\text{and}\qquad\text (\omega\oti\omega)\circ R=R\,.
  \ee
\end{defi}

For any Hopf algebra automorphism $\omega$ of a Hopf algebra $H$ we denote
by $\Fomega$ the bimodule obtained from $F$ by twisting the right $H$-action
by $\omega$, i.e.
  \be
  \Fomega:=(\Hs,\,\rho_F,\,\ohr_F\cir(\id_{\Hs}\oti\omega))\,.
  \labl{F_omega}
(An isomorphic bimodule is obtained when twisting instead the left $H$-action
by $\omega^{-1}$.) In Section 6 of \cite{fuSs3} the following is shown:

\begin{lemma}
Let $H$ be a factorizable ribbon Hopf algebra and $\omega$ a ribbon
automorphism of $H$.
\\[3pt]
{\rm (i)}\, The $H$-bimodule $\Fomega$ together with the
dinatural family of morphisms
  \be
  \imath^{\Fomega_{}}_X := (\omega^{-1})^* \circ \iHb_X
  \ee
is the coend of the functor from $\HMod\op{\times}\, \HMod$ to \HBimod\ that acts
on objects by assigning to a pair $\big((U,\rho_U), (V,\rho_V)\big)$ of left
$H$-modules the vector space $U^\vee{\otimes_\ko}\,V$ endowed with left $H$-ac\-tion
$[(\rho_U)_\vee^{}\cir(\omega^{-1}\oti\id_{U^*_{}})] \oti \id_V$ and right
$H$-action $\id_{U^*_{}} \oti [\rho_V \cir \tau_{V,H}\cir (\id_V\oti\apoi)]$.
\\[3pt]
{\rm (ii)}\, The linear maps defined in
\eqref{pic-Hb-Frobalgebra} equip the object $\Fomega$ in \HBimod\ with the
structure of a commutative symmetric Frobenius algebra, with trivial twist,
in \HBimod. Furthermore, $\Fomega$ is special iff $H$ is semisimple.
\end{lemma}

To proceed we note the following identity:

\begin{lemma}
For any factorizable ribbon Hopf algebra $H$ the relation
  \be
  f_{Q^{-1}} \big( \lambda \cir m \cir (v \oti \id_H) \big)
  = (\lambda \cir v)\, v^{-1}
  \labl{v-v-inv}
involving the ribbon element $v$, the inverse of the monodromy matrix $Q$ and
the cointegral $\lambda$ holds.
\end{lemma}

\begin{proof}
Just use the fact (see formula \erf{def-ribbon}) that $(v \oti v) \,{\cdot}\,
Q^{-1} \eq \Delta \cir v$ and afterwards the defining property of the cointegral.
\end{proof}

As a direct consequence we have

\begin{lemma}\label{lambda-omega}
Every ribbon automorphism $\omega$ of $H$ preserves the integral and
cointegral of $H$, i.e.
  \be
  \lambda \cir \omega = \lambda  \qquand   \omega \cir \Lambda = \Lambda \,.
  \labl{lo=l,oL=L}
\end{lemma}

\begin{proof}
Consider the equality obtained by composing \erf{v-v-inv} with
$\omega$. Using on the left hand side of this equality
that $\omega$ is an algebra morphism and that it preserves $v$
(and thus $v^{-1}$) as well as $Q^{-1}$, one arrives at an equality that
differs from \erf{v-v-inv} only by replacing $\lambda$ on the left hand side
by $\lambda \cir \omega^{-1}$. Using further that the morphism $f_{Q^{-1}}$ as
well as the element $v$ of $H$ are invertible, the first of the equalities
\erf{lo=l,oL=L} follows.
\\
Further note that $\omega\cir\Lambda$ is again a non-zero
integral and is thus proportional to $\Lambda$. Since
$\lambda\cir\omega\cir\Lambda \eq \lambda\cir\Lambda \iN \ko$
is non-zero, this implies the second equality in \erf{lo=l,oL=L}.
\end{proof}

We can now generalize the morphisms $\Cor gn$ defined in \eqref{Sk_morph} by
simply replacing every occurrence of the Frobenius algebra $F$ with $\Fomega$.
We denote the so obtained morphisms by $\Corw gn$.

\begin{prop}\label{prop:Corw}
For every ribbon automorphism $\omega$ of $H$ we have
  \be
  \Corw gn = \Cor gn \circ \big(\id_{\Hs}\oti (\omega^{-1})^*\big)^{\otimes g}_{}
  \labl{eq:Corw}
as linear maps, for all pairs of integers $g,n \,{\ge}\,0$.
\end{prop}

\begin{proof}
Inserting the $H$-bimodule structure \erf{F_omega} of $\Fomega$ into the
expression \erf{Lyubact_HKH} for the action of $K$ we have
  \eqpic{Lyubact_HKH_Fomega} {125} {44} {
  \put(0,45)       {$ \rho^K_{\Fomega} =~ $}
    \put(50,0) { \Includepichtft{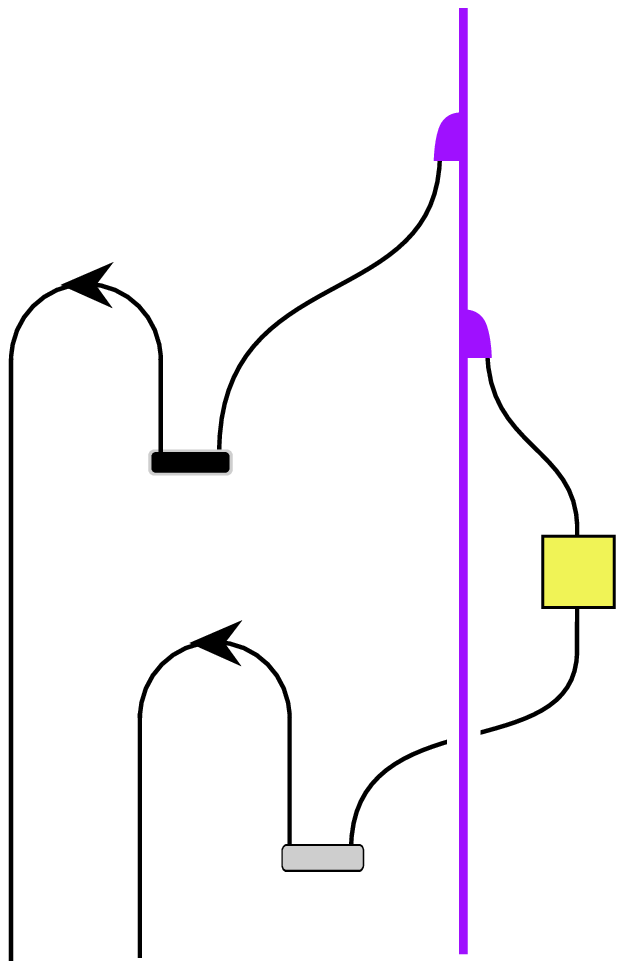}
  \put(-4.4,-7.7)  {\sse$ \Hss $}
  \put(10,-7.7)    {\sse$ \Hss $}
  \put(44.7,-7.7)  {\sse$ \Hss $}
  \put(31,2.8)     {\sse$ Q $}
  \put(13,46.8)    {\sse$ Q^{-1} $}
  \put(54.7,68.8)  {\sse$ \ohr_F^{} $}
  \put(35.2,90.5)  {\sse$ \rho_F^{} $}
  \put(46.5,108.3) {\sse$ \Hss $}
  \put(69.3,41.3)  {\sse$\omega$}
  } }
Using $(\id_H \oti \omega) \cir Q \eq (\omega^{-1} \oti \id_H) \cir Q$, it
follows immediately that $\rho^K_{\Fomega} \eq \rho^K_F \cir (\id_{\Hs}
\oti (\omega^{-1})^*)$ which, in turn, implies \erf{eq:Corw}.
\end{proof}

Next we note the following $\omega$-twisted version of Lemma \ref{lem:Hpastloops}:

\begin{lemma}\label{lem:Hpastloopsw}
Denoting, as in picture {\rm \erf{Hpastloops}}, by $\rho^{}_{K^{\otimes p}_{}}$
and $\ohr^{}_{K^{\otimes p}_{}}$ the left and right $H$-actions on
$K^{\otimes p}_{}$, we have
  \eqpic{Hpastloopsw} {360} {107} {\setulen80
    \put(0,0) {\includepichtft{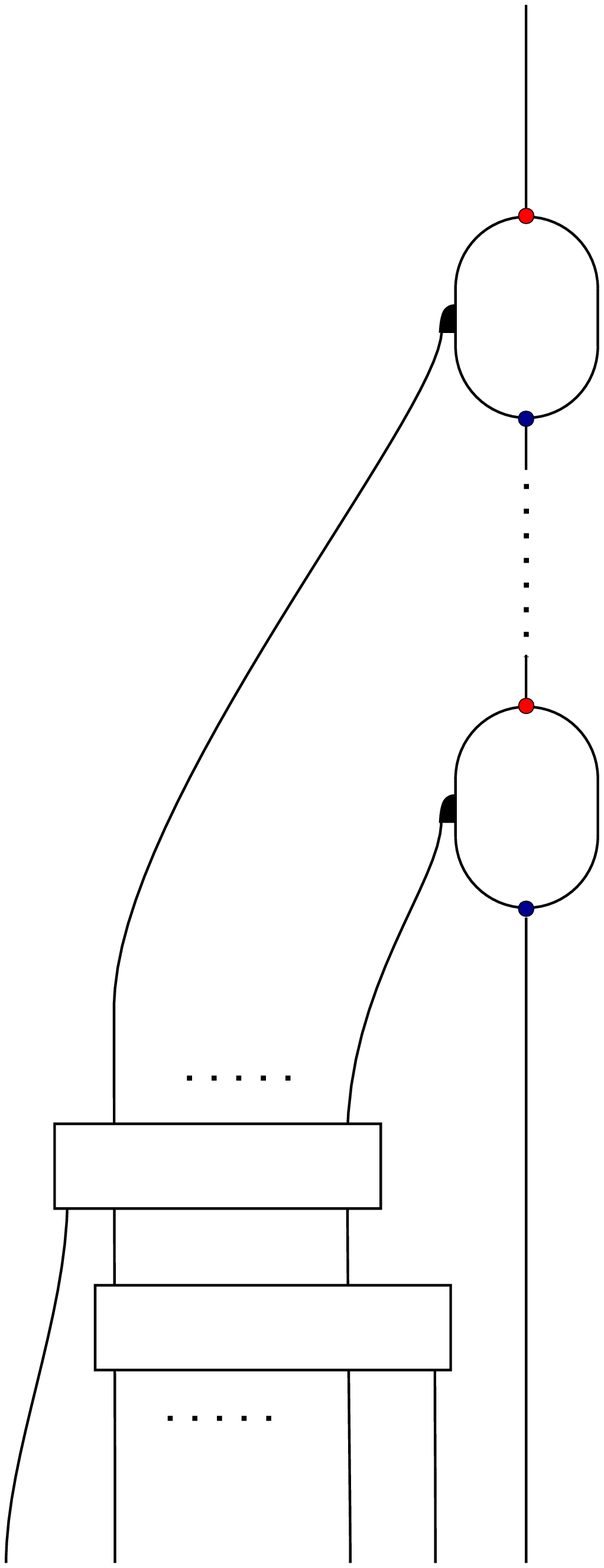}
  \put(-5,-10.2)   {\sse$ H $}
  \put(16.1,-10.2) {\sse$ K $}
  \put(59.4,-10.2) {\sse$ K $}
  \put(76.2,-10.2) {\sse$ H $}
  \put(92.2,-10.2) {\sse$ \Fomega $}
  \put(92.9,293.3) {\sse$ \Fomega $}
  \put(28,73.2)    {\sse$ \rho_{\HK^{\!\otimes p}}^{} $}
  \put(39,43.2)    {\sse$ \ohr_{\HK^{\!\otimes p}}^{} $}
  }
  \put(152,129)    {$=$}
    \put(230,0) { \put(-37.5,0)    {\includepichtft{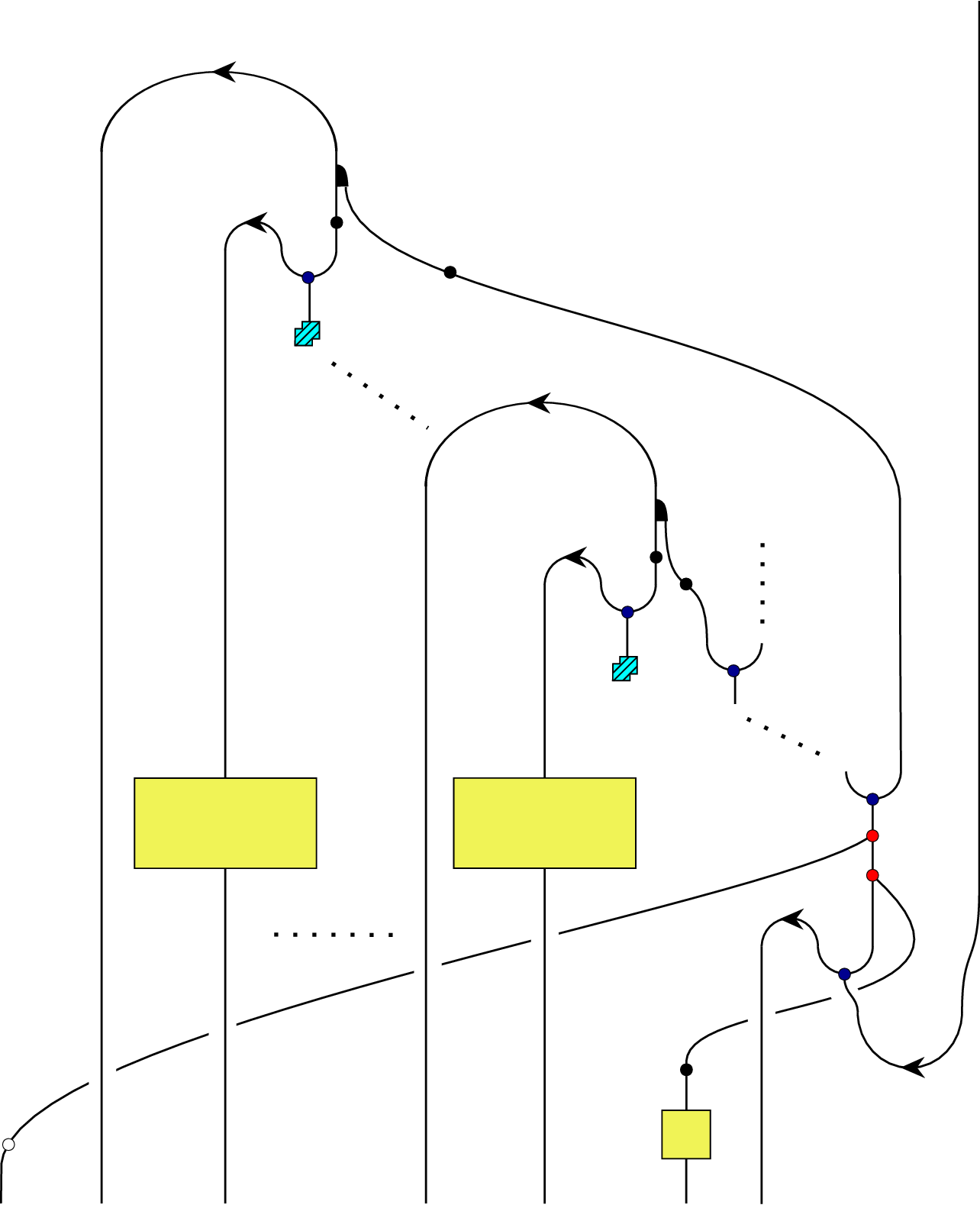}}
  \put(-42.5,-10.2)   {\sse$ H $}
  \put(-19,-10.2)  {\sse$ \Hss $}
  \put(10.2,-10.2) {\sse$ \Hss $}
  \put(58.2,-10.2) {\sse$ \Hss $}
  \put(87.4,-10.2) {\sse$ \Hss $}
  \put(121.2,-10.2){\sse$ H $}
  \put(139.2,-10.2){\sse$ \Hss $}
  \put(192.2,292.5){\sse$ \Hss $}
  \put(23.6,205.5) {\sse$ \Lambda $}
  \put(99.6,125.1) {\sse$ \Lambda $}
  \put(47.9,244.5) {\sse$ \ohrad $}
  \put(124.5,165)  {\sse$ \ohrad $}
  \put(122.2,13.6) {$\sse\omega$}
  \put(71.5,87.5)  {$ {(\omega^{-1})}^* $}
  \put(-5,87.5)    {$ {(\omega^{-1})}^* $}
  } }
\end{lemma}

\begin{proof}
This follows by the same line of arguments as in Lemma \ref{lem:Hpastloops},
combined with the identity $\ohrad \cir (\omega^{-1}\oti\id_{H})
\eq \omega^{-1}\cir\ohrad \cir (\id_H\oti\omega)$.
\end{proof}

\begin{thm}
Let $H$ be a finite-dimensional factorizable ribbon Hopf algebra and $\omega$
a ribbon automorphism of $H$. Then for any pair of integers $g,n \,{\ge}\,0$
the morphism $\Corw gn$ is invariant under the action
$\pi_{g:n}^{(\Fomega)^{\otimes n}_{}}$ of the mapping class group $\Mapgn$.
\end{thm}

\begin{proof}
Just like in the case $\omega \eq \id_\Hs$, invariance under the action of the
generators $\omega_i$ and $\Ri$ is an immediate consequence of the fact that
$\Fomega$ is cocommutative and has trivial twist.
\\[3pt]
Next consider the generators $\sk,\ak,\bk\,\dk$ and $\ek$. Proposition
\ref{prop:Corw} reduces invariance to the statement that the morphism
$\pi_{g:n}^{(\Fomega)^{\otimes n}_{}\!}(\gamma)$ commutes with
$(\id_{\Hs}\oti(\omega^{-1})^*)^{\otimes g}$ for $\gamma \eq S_k, a_m, b_m, d_m$
or $e_m$. In particular, for the case of $\gamma\eq S_1$ and
$g\eq 1$, the following chain of equalities establishes invariance:
  \Eqpic{S_KH_w} {420}{44} { \put(-7,3){
    \put(0,-6) { \Includepichtft{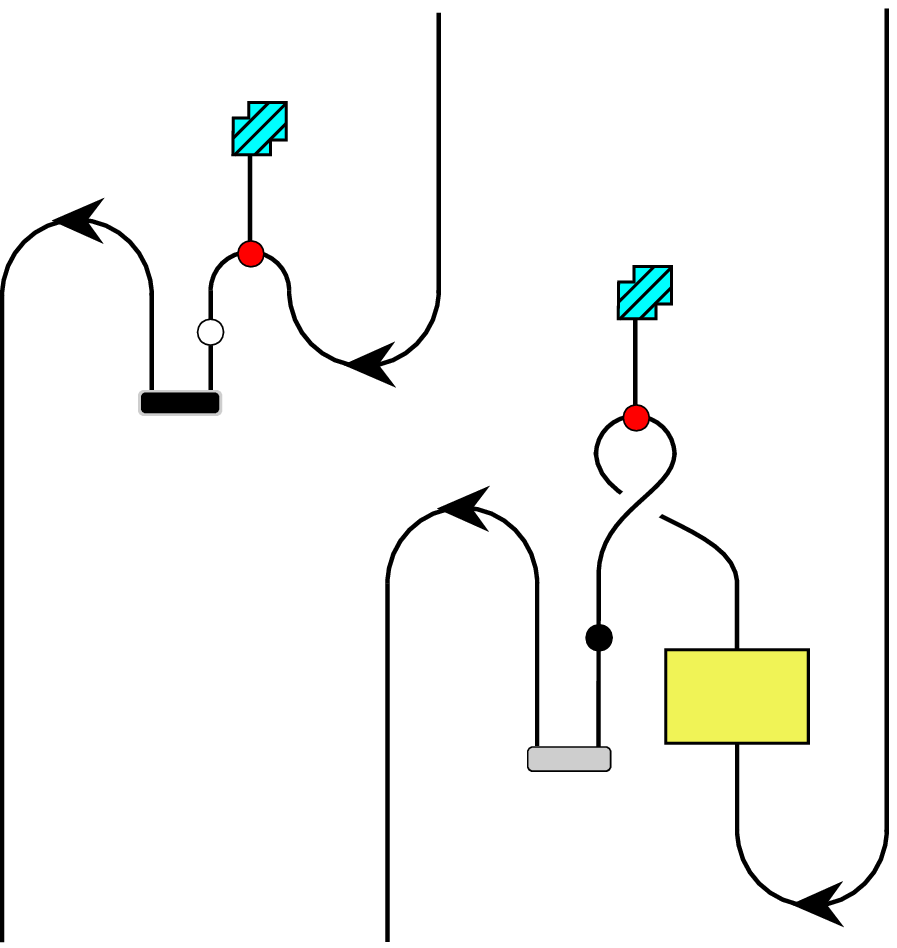}
  \put(-4.5,-9.2)  {\sse$ \Hss $}
  \put(12.7,51.1)  {\sse$ Q^{-1} $}
  \put(37.7,-9.2)  {\sse$ \Hss $}
  \put(32.5,87)    {\sse$ \lambda $}
  \put(44.6,106)   {\sse$ \Hss $}
  \put(58.4,11.8)  {\sse$ Q $}
  \put(74.5,69)    {\sse$ \lambda $}
  \put(93.6,106)   {\sse$ \Hss $}
  \put(74,24)      {\sse$ \omega^{-1}$}
  }
  \put(127,42)     {$=$}
    \put(160,-6) { \Includepichtft{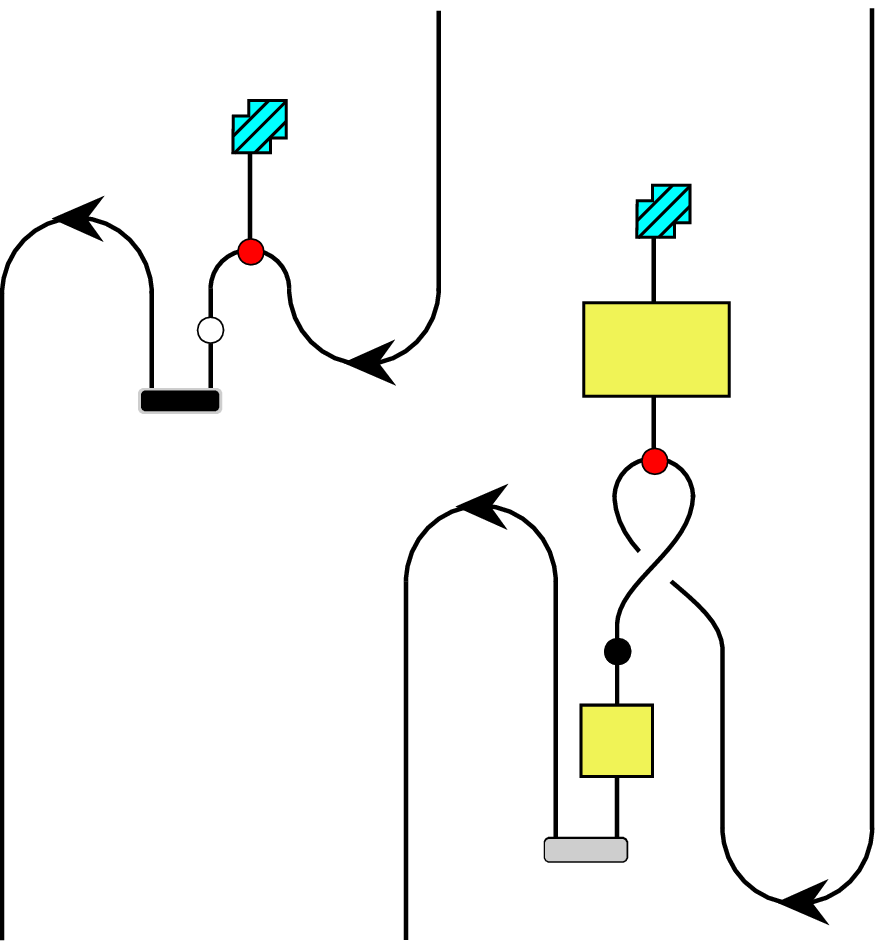}
  \put(-4.5,-9.2)  {\sse$ \Hss $}
  \put(12.7,51.1)  {\sse$ Q^{-1} $}
  \put(38.9,-9.2)  {\sse$ \Hss $}
  \put(32.5,87)    {\sse$ \lambda $}
  \put(44.6,106)   {\sse$ \Hss $}
  \put(60.4,2.1)   {\sse$ Q $}
  \put(77.5,78)    {\sse$ \lambda $}
  \put(93,106)     {\sse$ \Hss $}
  \put(65.3,62.5)  {\sse$ \omega^{-1}$}
  \put(65,20.1)    {\sse$ \omega$}
  }
  \put(285,42)     {$=$}
    \put(320,-6) { \Includepichtft{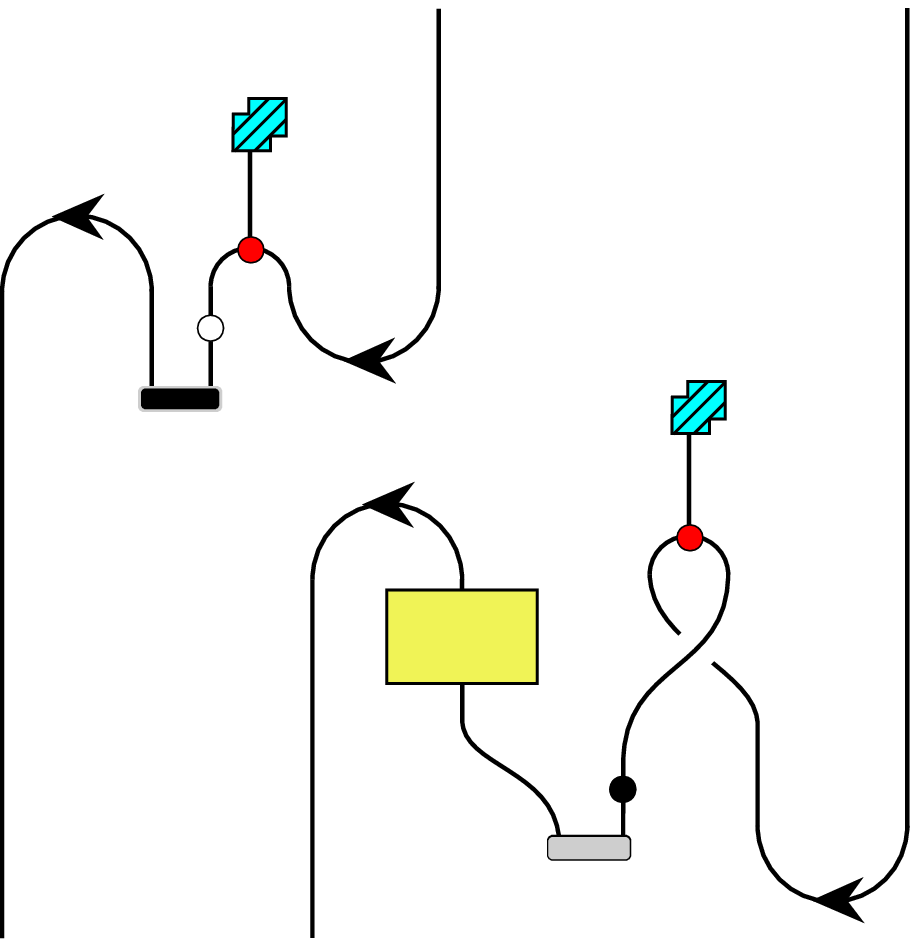}
  \put(-4.5,-9.2)  {\sse$ \Hss $}
  \put(13,51.1)    {\sse$ Q^{-1} $}
  \put(28.9,-9.2)  {\sse$ \Hss $}
  \put(32.5,87)    {\sse$ \lambda $}
  \put(43.6,106)   {\sse$ \Hss $}
  \put(61.1,2.1)   {\sse$ Q $}
  \put(80.5,56)    {\sse$ \lambda $}
  \put(95.6,106)   {\sse$ \Hss $}
  \put(43.5,30.1)  {\sse$ \omega^{-1}$}
  } } }
The first of these equalities follows by pushing the automorphism $\omega$
through the product and using that $\omega$ commutes with the antipode of $H$,
while the second equality follows by using $\lambda \cir \omega \eq \lambda$
from Lemma \ref{lambda-omega} and $(\omega\oti\omega)\circ Q\eq Q$.
\\
That $[\pi_{g:n}^{(\Fomega)^{\otimes n}_{}\!}(\gamma),
(\id_{\Hs}\oti(\omega^{-1})^*)^{\otimes g}] \eq 0$ holds as well  for any genus
$g$ and any of the generators $\gamma \eq S_k, a_m, b_m, d_m, e_m$
is shown in a completely analogous manner.
\\[3pt]
It remains to consider the action of the generators $t_{j,k}$.
To this end we just observe that because of Proposition \ref{prop:Corw},
a version of Proposition \ref{prop:triple} holds in which $F$ is replaced by
$\Fomega$, while Lemma \ref{lem:Hpastloopsw} implies that there are versions
of Lemma \ref{lem:removeQq_t} and Lemma \ref{lem:twist_S_remove} in which $F$
is replaced by $\Fomega$. Combining these versions of Proposition
\ref{prop:triple} and Lemmas \ref{lem:removeQq_t} and \ref{lem:twist_S_remove},
one arrives at a corresponding $\omega$-twisted version of Lemma
\ref{lem:t_act_Cor}. Now notice that in the case of $\omega \eq \id_H$,
invariance under the action of $t_{j,k}$ follows from Lemma \ref{lem:t_act_Cor}
by just invoking that $F$ is Frobenius. As a consequence, the twisted version
of Lemma \ref{lem:t_act_Cor} allows us to deduce invariance in precisely the
same manner as for $\omega \eq \id_H$.
\end{proof}


  \vskip 5.5em

\noindent{\sc Acknowledgments:}
We thank Benson Farb for a helpful correspondence.
JF is grateful to Hamburg university, and in particular to CSc and Astrid
D\"orh\"ofer, for their hospitality during the time when this study was initiated.
\\
JF is largely supported by VR under project no.\ 621-2009-3993.
CSc is partially supported by the Collaborative Research Centre 676 ``Particles,
Strings and the Early Universe - the Structure of Matter and Space-Time'' and
by the DFG Priority Programme 1388 ``Representation Theory''.

\newpage


  \newcommand\wb{\,\linebreak[0]} \def\wB {$\,$\wb}
  \newcommand\Bi[2]    {\bibitem[#2]{#1}}
  \newcommand\inBO[9]  {{\em #9}, in:\ {\em #1}, {#2}\ ({#3}, {#4} {#5}),
                         p.\ {#6--#7} {{\tt [#8]}}}
  \renewcommand\J[7]   {{\em #7}, {#1} {#2} ({#3}) {#4--#5} {{\tt [#6]}}}
  \newcommand\JO[6]    {{\em #6}, {#1} {#2} ({#3}) {#4--#5} }
  \newcommand\BOOK[4]  {{\em #1\/} ({#2}, {#3} {#4})}
  \newcommand\prep[2]  {{\em #2}, preprint {\tt #1}}
  \def\aspm  {Adv.\wb Stu\-dies\wB in\wB Pure\wB Math.}
  \def\coma  {Con\-temp.\wb Math.}
  \def\comp  {Com\-mun.\wb Math.\wb Phys.}
  \def\isjm  {Israel\wB J.\wb Math.}
  \def\joal  {J.\wB Al\-ge\-bra}
  \def\jktr  {J.\wB Knot\wB Theory\wB and\wB its\wB Ramif.}
  \def\jpaa  {J.\wB Pure\wB Appl.\wb Alg.}
  \def\momj  {Mos\-cow\wB Math.\wb J.}
  \def\nupb  {Nucl.\wb Phys.\ B}
  \def\plms  {Proc.\wB Lon\-don\wB Math.\wb Soc.}
  \def\pcps  {Proc.\wB Cam\-bridge\wB Philos.\wb Soc.}
  \def\slnm  {Sprin\-ger\wB Lecture\wB Notes\wB in\wB Mathematics}
  \def\taac  {Theo\-ry\wB and\wB Appl.\wb Cat.}
  \def\thmp  {Theor.\wb Math.\wb Phys.}

\end{document}